\documentclass[11pt]{article}

\usepackage[a4paper,margin=1in]{geometry}
\usepackage{amsmath,amssymb,amsthm,mathtools}
\usepackage{enumitem}
\usepackage{microtype}
\usepackage{amsrefs}
\usepackage{thmtools}
\usepackage{thm-restate}
\usepackage{hyperref}
\makeatletter
\@ifundefined{newcounteralias}{}{%
	\renewcommand\thmt@autorefsetup{%
		\@xa\def\csname\thmt@envname autorefname\@xa\endcsname\@xa{\thmt@thmname}%
	}%
}
\makeatother
\usepackage[title]{appendix}

\hypersetup{
  colorlinks=true,
  linkcolor=blue,
  citecolor=blue,
  urlcolor=blue
}

\numberwithin{equation}{section}

\theoremstyle{plain}
\newtheorem{theorem}{Theorem}[section]

\newtheorem{proposition}[theorem]{Proposition}
\newtheorem{lemma}[theorem]{Lemma}
\newtheorem{corollary}[theorem]{Corollary}
\theoremstyle{definition}
\newtheorem{definition}[theorem]{Definition}
\newtheorem{example}[theorem]{Example}
\theoremstyle{remark}
\newtheorem{remark}[theorem]{Remark}

\theoremstyle{plain}
\newtheorem*{theorem*}{Theorem}

\allowdisplaybreaks

\newcommand{\R}{\mathbb{R}}
\newcommand{\Sph}{\mathbb{S}}
\newcommand{\Hyp}{\mathbb{H}}

\newcommand{\g}{\mathfrak{g}}
\newcommand{\kalg}{\mathfrak{k}}
\newcommand{\palg}{\mathfrak{p}}
\newcommand{\m}{\mathfrak{m}}

\newcommand{\tr}{\operatorname{tr}}
\newcommand{\diag}{\operatorname{diag}}
\newcommand{\inner}[2]{\left\langle #1,#2\right\rangle}

\newcommand{\C}{\mathbb C}

\newcommand{\Sn}{\mathbb S}
\newcommand{\Hn}{\mathbb H}
\newcommand{\Q}{\mathbb Q}
\newcommand{\eps}{\varepsilon}
\newcommand{\ii}{\mathbf i}

\newcommand{\Ker}{\operatorname{ker}}
\newcommand{\nab}{\bar\nabla}

\DeclareMathOperator{\id}{id}

\DeclareMathOperator{\rank}{rank}

\title{Isoparametric Submanifolds in Products of Space Forms}
\author{Jinxuan Chen, \,\,\, Xiaobo Liu\thanks{Research was partially supported by NSFC grants 12341105 and 12526302.}, \,\, \, Wanxu Yang}
\date{}

\begin{document}
\maketitle

\begin{abstract}
In this paper, we obtain a classification of isoparametric submanifolds with arbitrary codimensions in products of simply-connected space forms, except for isoparametric hypersurfaces in $\Sph^n\times\Sph^m$ and $\Hyp^n\times \Hyp^m$.
\end{abstract}

\section{Introduction}

Isoparametric hypersurfaces in space forms are hypersurfaces with constant principal curvatures.
The study of these hypersurfaces was initiated by Cartan, Levi-Civita, and Segre in 1930's, 
and their complete classification was achieved only recently by works of many mathematicians (see the survey article \cite{chi2020isoparametric} for details).
A definition of higher codimensional isoparametric submanifolds in space forms was given
by Terng in \cite{terng1985isoparametric}.
These submanifolds have also been classified by a combination of works by
Thorbergsson \cite{Th91}, Palais-Terng \cite{PT87}, Dadok \cite{Dadok}, and Wu \cite{Wu}.
According to a definition given by Heintze, Liu, and Olmos in \cite{HLO}, a submanifold $L$ in an
arbitrary Riemannian manifold $M$ is {\it isoparametric} if $L$ has flat normal bundle, image of
a small neighbourhood of the origin in every normal space of $L$  under the exponential map is totally geodesic, and close by local parallel submanifolds of $L$ have
parallel mean curvature vector fields. 
In recent years, there have been many works on isoparametric hypersurfaces in
products of space forms (see, for example, \cite{CS2019}, \cite{Cui2025},  \cite{de2022isoparametric}, 
\cite{delima2025}, \cite{SS2023}, \cite{gao2024isoparametric}, \cite{gao2024hypersurfaces}, \cite{JulioB},
\cite{MSSV2025}, \cite{TanXieYan2026},  \cite{urbano2019hypersurfaces}, \cite{Wang2026}, etc.).
The purpose of this paper is to study isoparametric submanifolds
with arbitrary codimensions in products of space forms.


Let $\Q^n_{\eps}$ be the simply connected space form of dimension $n$ with sectional curvature $\eps$. 
Then $\Q^n_{1} = \mathbb{S}^n$ is the $n$-dimensional unit sphere, 
$\Q^n_{0} = \mathbb{R}^n$ is the $n$-dimensional Euclidean space,
and $\Q^n_{-1} = \mathbb{H}^n$ is the $n$-dimensional hyperbolic space with sectional curvature equal to $-1$.
In this paper, we study isoparametric submanifolds in $\Q^n_{\eps_1}\times\Q^m_{\eps_2}$, where $\eps_1\in\{-1,1\}$, and $\eps_2\in\{-1,0,1\}$. Since the product
of two Euclidean spaces is again a Euclidean space, we do not need to consider the case
$\eps_1=\eps_2=0$. The main result of this paper is the following

\begin{restatable}{theorem}{maintheorem}\label{main}
    Let $L$ be an isoparametric submanifold in $\Q^n_{\epsilon_1}\times \Q^m_{\epsilon_2}$. Unless
    $\eps_1 = \eps_2$ and $\text{codim}\ L=1$, $L$ is an open subset of $L_1\times L_2$ where either

    (i) $L_1\subset \Q^n_{\epsilon_1}$ and $L_2\subset \Q^m_{\epsilon_2}$ are isoparametric submanifolds, or

    (ii) $L_1\subset \Hyp^n\times \R^{m_1}$ is a flat horospherical hypersurface (see Example \ref{ex:flat-horosphere} for definition), $L_2\subset \R^{m-m_1}$ is an  isoparametric submanifold, and $m_1$ is a positive integer less than $m$.
\end{restatable}

\noindent
Note that (ii) occurs only if $\eps_1 = -1$ and $\eps_2=0$.
Since isoparametric submanifolds in 
space forms have been classified. This theorem also gives a complete classification for
all isoparametric submanifolds in $\Q^n_{\epsilon_1}\times \Q^m_{\epsilon_2}$
except for isoparametric hypersurfaces in $\Sph^n\times \Sph^m$ and $\Hyp^n\times \Hyp^m$. 

We would like to mention that isoparametric hypersurfaces in $\Q^{2}_{\epsilon_1}\times\Q^{2}_{\epsilon_2}$
have been classified in \cite{urbano2019hypersurfaces}, \cite{SS2023},  
\cite{gao2024isoparametric}, \cite{gao2024hypersurfaces}.
Isoparametric hypersurfaces in $\Q^{n}_{\epsilon_1}\times \mathbb{R}^m$
have been classified in \cite{de2022isoparametric} (for $m=1$) and \cite{TanXieYan2026} (for $m \geq 1$).
Under an extra condition that there exist distinguished points, a classification for isoparametric hypersurfaces 
in other spaces  $\Q^n_{\epsilon_1}\times \Q^m_{\epsilon_2}$ was also given 
in \cite{delima2025}. The codimension-1 case of Theorem \ref{main} recovers 
the above results except for hypersurfaces in $\Sph^n\times \Sph^m$ and $\Hyp^n\times \Hyp^m$.
The corresponding part in the proof of Theorem \ref{main}  simplifies arguments in
\cite{TanXieYan2026} and  \cite{delima2025}. Moreover, the codimension-1 case of 
Theorem \ref{main} also implies a classification
of isoparametric hypersurfaces in $\Sph^n \times \Hyp^m$, which was not obtained before.

This paper is organized as follows: In Section \ref{sec:preliminaries}, we review 
the definition of isoparametric submanifolds in general Riemannian manifolds given in \cite{HLO}, Lie triple systems in symmetric spaces, Riemannian curvature tensor on space forms, and many other basic results. In Section \ref{sec:lts}, we classify the Lie triple systems in $\Q^n_{\epsilon_1}\times \Q^m_{\epsilon_2}$. In Section \ref{sec:localdecomp}, we prove that there is a ``stability'' for Lie triple systems corresponding to sections of $L$ and use it to prove the local decomposition and the global decomposition. In Section \ref{sec:higher-codim-diagonal}, we show the non-existence of $L$ in higher codimensional diagonal cases in $\Q^n_{\eps_1}\times\Q^m_{\eps_2}$ with $\eps_1=\eps_2$. In Section \ref{sec:reduction}, we reduce the classification of isoparametric submanifolds in $\Q^n_{\eps_1}\times\Q^m_{\eps_2}$ with $\eps_2=0$ to the classification of isoparametric hypersurfaces. In Section \ref{sec:constant-angle}, we prove the constant angle theorems. In Section \ref{sec:classification-hypersurface}, we use the constant angle theorems to classify isoparametric hypersurfaces in the $\eps_1\ne\eps_2$ cases. In Section \ref{sec:theorem}, we prove the classification results. Appendix~\ref{app:algebraic} gives some basic algebraic results. A small portion of this paper was prepared with the help of artificial intelligence tools, a disclosure
of which is provided in Appendix~\ref{app:AI}.

\section{Preliminaries}\label{sec:preliminaries}

\subsection{Isoparametric Submanifolds}

In this subsection, we review the definition and some basic properties of isoparametric submanifolds in arbitrary Riemannian manifolds as given in \cite{HLO}.

\begin{definition} \label{defn:HLO}
Let $M$ be a complete Riemannian manifold.  An immersed submanifold $L\subset M$ is \emph{isoparametric} if:

\item (i) The normal bundle $\nu L$ is flat.
\item (ii) For every $x\in L$, there is small neighbourhood $V$ of the origin in $\nu_x L$ such that
$\Sigma_x := \exp (V) $ is a totally geodesic submanifold in $M$.  $\Sigma_x$ is called a {\it section}
of $L$ at $x$.
\item (iii) For every $x \in L$, there exists a neighbourhood $U$ of $x$ in $L$ and $\eps > 0$ such that
for any parallel normal vector field $\xi$ defined over $U$ with $\| \xi \| < \eps$, 
\begin{equation} \label{eqn:Uxi}
U_\xi : = \{ \exp (\xi(p)) \mid p \in U\}
\end{equation}
 is a smooth submanifold of $M$ with constant mean curvature in radial directions.
\end{definition}

\begin{remark}
The submanifold $U_\xi$ defined by equation \eqref{eqn:Uxi} is called a {\it local parallel submanifold} of $L$.
By \cite{HLO}*{Theorem 2.4}, condition (iii) can be replaced by requiring $U_\xi$
to have parallel mean curvature vector field.
Moreover we say
    $L$ is  \emph{almost isoparametric} if it satisfies conditions (i) and (iii).  
\end{remark}

\begin{proposition}[\cite{HLO}*{Proposition 2.3}]\label{prop:HLO}
    If $L$ has flat normal bundle and totally geodesic sections, then locally $M$ splits as $L\times \Sigma$ with metric $g_1(x)\oplus g_2$ where $g_1(x)$ is a metric on $L$ depending on $x\in \Sigma$ and $g_2$ is a fixed metric on $\Sigma$.
\end{proposition}

We will always assume $L$ is connected. The following immediate corollary of Proposition \ref{prop:HLO} will be useful.

\begin{corollary}\label{cor:section}
    For any $x,y\in L$, let $\Sigma_x$ and $\Sigma_y$ be sections of $L$ at $x$ and $y$ respectively. There exist  small neighborhoods $U_x$ and $U_y$ such that $x \in U_x \subset \Sigma_x$,  $y \in U_y \subset \Sigma_y$,  and there exists an isometry $f: U_x \rightarrow U_y$ with $f(x)=y$. 
\end{corollary}

The following Proposition gives a characterization for almost isoparametric submanifolds.

\begin{proposition}[\cite{HLO}*{Proposition 2.1}] \label{prop:VolConst}
    Let $L$ be an immersed submanifold in $M$ with flat normal bundle. Then
    $L$ is almost isoparametric if and only if for each $p\in L$, there exists a neighborhood
    $U$ of $p$ in $L$ and $r >0$ such that for all parallel normal vector field $\xi$ defined over $U$ with
     $\|\xi\| < r$, the endpoint map $\phi_\xi: U\rightarrow U_\xi$ defined by 
     $\phi_\xi (x)= \exp (\xi(x))$ for $x \in U$ is volume preserving up to a constant factor.
\end{proposition}

Fix a parallel normal vector field $\xi$ over an open subset $U \subset L$. For every $p \in U$, let
$\gamma_p(t) := \exp_p (t \, \xi(p))$. Recall a {\it Jacobi field} $Y$
along the geodesic $\gamma_p$ is a vector field $Y(t) \in T_{\gamma_p(t)} M$  satisfying the {\it Jacobi equation} 
\begin{equation}\label{eq:jacobi-equation}
    Y''(t) + R(Y(t), \gamma'_p(t)) \, \gamma'_p(t) =0.
\end{equation}
We can choose Jacobi fields $Y_1, \ldots, Y_{\dim L}$ along $\gamma_p$ such that
$\{Y_1(0), \ldots, Y_{\dim L}(0) \}$ form an orthonormal basis of $T_p L$, and  
$Y_i'(0)=-A_{\xi(p)} Y_i(0)$ for all $i=1, \ldots, \dim L$.
Proposition \ref{prop:VolConst} and its proof imply that 
\begin{equation}\label{eq:HLO-determinant-condition}
\det (d_p \phi_{t \xi}) = |Y_1\wedge...\wedge Y_{\dim L}(t)|
\quad\text{is independent of }p
\end{equation}
for all sufficiently small $t$ and $U$ if $L$ is almost isoparametric. Here
$\phi_{t \xi}: U \to U_{t \xi}$ is the endpoint map and $d_p$ is the differential at point $p$.

The following is a straightforward check using the definition of isoparametric submanifold.
\begin{lemma} \label{lem:adding-euclidean-factor-no-change}
    Let $L$ be a submanifold in $M\times N$ such that $L$ is contained by $M\times \{q\}$ for a fixed point $q \in N$. Then $L$ is isoparametric in $M\times \{q\}$ if and only if $L$ is isoparametric in $M \times N$.
\end{lemma}

\begin{remark} \label{rem:Euclidean-isoparametric} 
According to a definition given by Terng, a submanifold $L$  in a Euclidean space is isoparametric 
if $\nu L$ is flat and the principal curvatures along every parallel normal vector field are constant (see \cite{terng1985isoparametric}*{Definition 1}). It was shown in \cite{HLO}*{Theorem 4.2} that this definition
coincides with Definition~\ref{defn:HLO} when the ambient space is a Euclidean space. In fact, this is also true in other space forms. 
\end{remark}

If $L$ is an isoparametric submanifold in $\mathbb{R}^n$, then
there is a decomposition $TL = \oplus_{i=1}^g E_i$ such that each $E_i$ is a distribution on $L$, and there exists
parallel normal vector field ${\bf n}_i$ such that 
$A_\xi \mid_{E_i} = \langle \xi, \, {\bf n}_i \rangle \, {\rm Id}_{E_i}$ for any normal vector field $\xi$. 
$E_i$ is called a {\it curvature distribution} with {\it curvature normal} ${\bf n}_i$.
$L$ is said to be \emph{full} if it is not contained in any affine hyperplane. 
By \cite{terng1985isoparametric}*{Proposition 1.3}, $L$ is full if and only if
 ${\rm Span}\{ {\bf n}_1(p), \ldots, {\bf n}_g(p)\} = \nu_p L$ for some $p \in L$.
Hence the fullness of $L$ is determined by curvature normals at a single point.
The following splitting theorem was proved by Terng:
\begin{theorem}[\cite{terng1985isoparametric}*{Theorem 1.20}]\label{thm:palais-terng}
    If $L$ is a complete full isoparametric submanifold in $\R^{n}$ and one of the curvature normals of $L$ 
    is zero with
     the dimension of the corresponding curvature distribution equal to $m$,
     then there exists a full isoparametric submanifold $L'$ of $\R^{n-m}$ such that $L = \R^{m} \times L'$
     up to an isometry on $\R^{n}$.
\end{theorem}
\begin{remark}\label{rk:palais-terng-local}
     Terng also proved that $L'$ is compact.
\end{remark}
In this paper, we will need the following local version of Theorem \ref{thm:palais-terng} which does not
assume $L$ to be complete or full.
\begin{corollary}\label{cor:real-palais-terng} 
    If $L$ is an isoparametric submanifold in $\R^n$ and one of the curvature normals of $L$ 
    is zero with the dimension of the corresponding curvature distribution equal to $m$, then there exists a compact isoparametric submanifold $L'$ in $\R^{n-m}$ such that, up to an isometry on $\R^n$, $L$ is an open subset of $\R^{m}\times L'  \subset \R^{n}$.
\end{corollary}
\begin{proof}
    By the extendability of isoparametric submanifolds in Euclidean space
\cite{terng1987submanifold}*{Theorem 3.4}, $L$ is an open subset of a complete isoparametric submanifold 
$\widetilde L \subset \R^n$. Up to an isometry on $\R^n$, we may assume $\widetilde L$ is a full isoparametric
submanifold in a subspace $\R^{n'} \subset \R^n$ for some $n' \leq n$.
By Theorem \ref{thm:palais-terng},  there exists a compact isoparametric submanifold
$L'$ in $\R^{n'-m} \subset \R^{n-m}$ such that $\widetilde L = \R^{m}\times L'$ up to an isometry
on $\R^{n'}$.  This implies the corollary.
\end{proof}

\subsection{Totally Geodesic Submanifolds in Symmetric Spaces}

In the definition of isoparametric submanifolds, sections are totally geodesic submanifolds
of the ambient space $M$. If $M$ is a symmetric space, tangent spaces of totally geodesic submanifolds
are described by Lie triple systems. We refer to \cite{helgason}*{Chapter IV} for basic concepts and facts about symmetric spaces.

Let $M=G/K$ be a Riemannian symmetric space where $G$ and $K$ are Lie groups which form a symmetric pair.
Let $o=eK \in M$ where $e$ is the identity element of $G$.  
Let $\g$ and $\kalg$ be the Lie algebras of $G$ and $K$ respectively. 
We have the Cartan decomposition $\g=\kalg \oplus \palg$ with 
$\palg \cong T_o M$. 

\begin{definition}
A linear subspace $\m\subset\palg$ is a \emph{Lie triple system} if  $ [\m,[\m,\m]]\subset\m$
(i.e. $[[X,Y],Z]\in\m$ for all $X,Y,Z\in\m$).
\end{definition}

Lie triple system can also be defined using the curvature tensor. Recall that the curvature tensor $R$ on any Riemannian manifold $M$ is defined by
\begin{equation*}
      R(X,Y)Z := \nabla_X \nabla_Y Z- \nabla_Y \nabla_X Z- \nabla_{[X,Y]} Z
\end{equation*}
for any vector fields $X,Y,Z$ on $M$, where $\nabla$ is the Levi-Civita connection on $M$.
If $M=G/K$ is a symmetric space, the curvature tensor is given by a simple formula
 $ R(X,Y)Z=-[[X,Y],Z]$  for $X,Y,Z\in\palg$ (see \cite{helgason}*{Chapter IV Theorem 4.2}). Therefore $\m$ is a Lie triple system if and only if 
\begin{equation}\label{eq:Lie-triple}
     R(\m,\m)\m\subset\m.
\end{equation}
Sometimes this condition is easier to use since it does not rely on the choice of the base point
$o$ and the representation of $M$ as
the quotient of $G$ and $K$.

The following theorem gives a characterization of totally geodesic submanifolds in $M$.
\begin{theorem}[Theorem 7.2 of \cite{helgason}*{Chapter IV}]\label{thm:LieTriple-TotalGeodesic}
Let $M$ be a connected Riemannian symmetric space. Fix a point $p\in M$. For any totally geodesic submanifold 
$N$ containing $p$, $T_pN$ is a Lie triple system. On the other hand, given a Lie triple system $\m \subset T_p M$, $\exp_p(\m)$ is a totally geodesic submanifold in $M$.
\end{theorem}

\noindent
Note that products of space forms are symmetric spaces. We will use this theorem to study sections of isoparametric
submanifolds in such spaces.

\subsection{Fundamental Equations in Submanifold Geometry}

Major tools in the study of submanifold geometry are the Gauss, Codazzi, and Ricci equations.
Assume $M$ is a Riemannian manifold and $L \subset M$ is an immersed submaniold 
 equipped with the  Riemannian metric induced from $M$.
We will use $\bar\nabla$ and $\bar R$ to represent the Levi-Civita connection and curvature tensor  on $M$. 
The Levi-Civita connection and curvature tensor on $L$ 
are denoted by $\nabla$ and $R$
respectively. Let $\nu L$ be the normal bundle of $L$.  The connection and
curvature tensor on $\nu L$ are denoted by $\nabla^\perp$ and $R^{\perp}$ respectively.
 The shape operator and second fundamental form on $L$ are denoted by
 $A$ and ${\rm II}$. For any $X\in TM|_L$, we write $X^T$ and $X^{\perp}$ for the projections
  of $X$ to $TL$ and $\nu L$ respectively. Let $X,Y,Z,W$ be tangent vector fields on $L$, and $\eta,\xi$ be normal vector fields on $L$. Then the {\it Gauss equation} can be written as
\begin{equation}\label{eq:gauss}
    \langle \bar R(X,Y)Z, W\rangle=\langle R(X,Y)Z,W\rangle-\langle \text{II}(X,W),\text{II}(Y,Z)\rangle+\langle \text{II}(X,Z),\text{II}(Y,W)\rangle. 
\end{equation}
For the Ricci equation and Codazzi equation, we will use the versions of these equations in terms of the shape-operator instead of the second fundamental form.
The {\it Ricci equation} is given by
\begin{equation}\label{eq:ricci}
    \langle R^{\perp}(X,Y)\eta,\xi\rangle=\langle \bar R(X,Y)\eta,\xi\rangle+\langle[A_\eta,A_\xi]X,Y\rangle,
\end{equation}
and the {\it Codazzi equation} is given by
\begin{equation}\label{eq:codazzi}
    (\nabla_XA)_\eta Y-(\nabla_Y A)_\eta X=-(\bar R(X,Y)\eta)^T,
\end{equation}
where  $(\nabla_XA)_\eta Y := \nabla_X(A_\eta Y)-A_\eta(\nabla_XY)-A_{\nabla^{\perp}_X\eta}Y$
(see, for example, \cite{tojeiro2019submanifold}*{Section 1.3}).

Throughout this paper, all submanifolds $L$ are assumed to be connected and have codimension at least $1$.

\subsection{Curvature Tensor for Products of Space Forms}\label{subsec:curvature}

Recall $\Q^n_{\eps}$ is the $n$-dimensional simply connected space form with
sectional curvature $\eps$.   The curvature tensor of $\Q^n_{\eps}$  has
a simple form
\begin{equation}
    R(X,Y)Z=\eps(\langle Y,Z\rangle X-\langle X,Z\rangle Y)
\end{equation}
for any $p \in \Q^n_{\eps}$ and $X,Y,Z \in T_p \Q^n_{\eps}$
(see, for example, \cite{lee2018introduction}*{Proposition 8.36}).
Consequently, the curvature tensor for $\Q^n_{\eps_1}\times \Q^m_{\eps_2}$ can also be derived easily.
For any $p=(p_1,p_2) \in \Q^n_{\epsilon_1}\times \Q^m_{\epsilon_2}$, 
\begin{equation} \label{eqn:DecompTangent}
T_{p}(\Q^n_{\epsilon_1}\times \Q^m_{\epsilon_2})
\cong T_{p_1} \Q^n_{\epsilon_1} \times T_{p_2} \Q^m_{\epsilon_2}.
\end{equation}
So every $X \in T_{p}(\Q^n_{\epsilon_1}\times \Q^m_{\epsilon_2})$ can be written as
$X=(X_1, X_2)$ with $X_1 \in T_{p_1} \Q^n_{\epsilon_1}$ and $X_2 \in T_{p_2} \Q^m_{\epsilon_2}$.
For any 
$X=(X_1,X_2), Y=(Y_1,Y_2),Z=(Z_1,Z_2) \in T_{p}(\Q^n_{\epsilon_1}\times \Q^m_{\epsilon_2})$,
we have 

\begin{equation}\label{eq:curvature-formula}
    R(X,Y)Z=\Big(\eps_1(\langle Y_1,Z_1\rangle X_1-\langle X_1,Z_1\rangle Y_1),\eps_2(\langle Y_2,Z_2\rangle X_2-\langle X_2,Z_2\rangle Y_2)\Big)
\end{equation}



A simple application of this formula is to calculate the curvature of the diagonal submanifold
in  $\Q^n_\eps\times\Q^n_\eps$ defined by
\begin{equation}\label{eq:defn-diagonal-sphere}
    \Delta \Q^n_\eps:=\{(x,x) \mid  x  \in \Q^n_\eps\} \subset \Q^n_\eps\times\Q^n_\eps.
\end{equation}
Since geodesics in $\Delta \Q^n_\eps$ are also geodesics in $\Q^n_\eps\times\Q^n_\eps$,
$\Delta \Q^n_\eps$ is totally geodesic in $\Q^n_\eps\times\Q^n_\eps$. 
Hence by the Gauss equation \eqref{eq:gauss} and equation \eqref{eq:curvature-formula},
the curvature tensor of $\Delta \Q_\eps^n$ is given by 
$$\left\langle R^{\Delta}\Big( (X,X), (Y,Y) \Big)(Z,Z),(W,W) \right\rangle
  = 2 \eps ( \langle Y, Z \rangle \langle X, W\rangle - \langle X, Z \rangle \langle Y, W\rangle) $$
for any $X,Y,Z,W\in T_x \Q^n_\eps$.
Consequently  $\Delta\Q^n_\eps$ is a space form with constant sectional curvature equal to $\eps/2$. 

\subsection{Product Structure and Angle Function}


In \cite{urbano2019hypersurfaces}, Urbano introduced the product structure on $\mathbb{S}^2 \times \mathbb{S}^2$,
which can be easily extended to more general spaces, such as $\Q^n_{\epsilon_1}\times\Q^m_{\epsilon_2}$.
Using decomposition \eqref{eqn:DecompTangent}, the {\it product structure} on $\Q^n_{\epsilon_1}\times\Q^m_{\epsilon_2}$ can be defined as  the map
$P:T(\Q^n_{\epsilon_1}\times\Q^m_{\epsilon_2})\rightarrow T(\Q^n_{\epsilon_1}\times\Q^m_{\epsilon_2})$
given by
\begin{equation} \label{equ:P-operator}
    P(X_1,X_2) :=(X_1,-X_2).
\end{equation}
for $(X_1,X_2)\in T(\Q^n_{\epsilon_1}\times\Q^m_{\epsilon_2})$.
One can check that $P$ is a parallel self-adjoint involutive bundle isomorphism which also preserves
the metric. Let $\bar\nabla$ be the Levi-Civita connection on $\Q^n_{\eps_1}\times \Q^m_{\eps_2}$.
 For any $z \in \Q^n_{\epsilon_1}\times\Q^m_{\epsilon_2}$ and $X,Y\in T_z(\Q^n_{\epsilon_1}\times\Q^m_{\epsilon_2})$,
 we have
\begin{equation}\label{eq:P-property}
    \bar\nabla P=0,\qquad P^2= {\rm Id}, \qquad \langle P(X),Y\rangle=\langle X,P(Y)\rangle, \qquad \langle P(X),P(Y)\rangle=\langle X,Y\rangle.
\end{equation}
Define two bundle maps $P_1$ and $P_2$ on $T (\Q^n_{\eps_1}\times\Q^m_{\eps_2})$ by
\begin{equation}\label{eq:dpi1dpi2-def}
    P_1:=\frac12({\rm Id} +P),\qquad P_2:=\frac12({\rm Id}-P).
\end{equation}
Both $P_1$ and $P_2$ are parallel, i.e. $\bar\nabla P_1 = \bar\nabla P_2 =0$,
since  $P$ and ${\rm Id}$ are parallel.



The product structure can be used to define the angle function
on hypersurfaces, which was also introduced by Urbano in \cite{urbano2019hypersurfaces} for
hypersurfaces in $\mathbb{S}^2 \times \mathbb{S}^2$.
More generally, for any hypersurface
$L\subset \Q^n_{\epsilon_1}\times\Q^m_{\epsilon_2}$, we can define the 
{\it angle function} $C$ on $L$ by
\begin{equation}\label{eq:angle-function}
    C(z):=\inner{P(N(z))}{N(z)}=|{N_1(z)}|^2-|{N_2(z)}|^2
\end{equation}
for all $z \in L$, where $N$ is a unit normal vector field on $L$, $N_1 := P_1(N)$, and $N_2 := P_2(N)$.
Note that $C$ does not depend on the choice of $N$ since it does not change if $N$ is replaced by $-N$.
Moreover, we have
\begin{equation}\label{eq:N=N1+N2}
    N=N_1+N_2, \hspace{20pt} \langle N_1, \, N_2 \rangle =0.
\end{equation} 
Since  $|N|^2=|{N_1(z)}|^2 + |{N_2(z)}|^2=1$, we have $|N_1|^2=\dfrac{1+C}{2}$ and $|N_2|^2=\dfrac{1-C}{2}$. We also define the tangential part of $P(N)$ as 
\begin{equation}\label{eq:T-defn}
    \mathcal{T}:=P(N)-CN.
\end{equation}
By definition of $C$, $\langle \mathcal{T}, \, N \rangle =0$. Hence $\mathcal{T} \in TL$.
Since $P(N)=\mathcal T+CN$ and $|P(N)|=|N|=1$, we get
\begin{align} \label{equ:T-length}
   |\mathcal T|^2=1-C^2.
\end{align}
Thus $C^2 \leq 1$ and  $\mathcal T\neq 0$ if and only if $-1<C<1$.

\begin{lemma} \label{lem:nabla-C}
Let $L$ be a hypersurface in $\Q^n_{\epsilon_1}\times\Q^m_{\epsilon_2}$. Then we have $\nabla C=-2A(\mathcal T)$, where $\nabla C$ is the gradient of $C$, and $A$ is the shape operator of $L$ along $N$.
\end{lemma}
\begin{proof}
For any tangent vector field $X$ on $L$, we have
\begin{equation}  \label{eq:sec2-lemma}    
X(C)=X\left<P(N),N\right>=\left<\bar\nabla_{X}P(N),N\right>+\left<P(N),\bar\nabla_{X}N\right>.
\end{equation}
By equations \eqref{eq:P-property} and \eqref{eq:T-defn}, the first term on the right hand side of
equation \eqref{eq:sec2-lemma} is  
$$\left<\bar\nabla_{X}P(N),N\right>=\left<P(\bar\nabla_{X}N),N\right>
=\left<\bar\nabla_{X}N,P(N)\right>=\left<-A(X),{\mathcal T}\right>,$$ 
which is also equal to the second term on the right hand side of
equation \eqref{eq:sec2-lemma}.
Since $A$ is symmetric, we obtain  $X(C)=-2\left<A({\mathcal T}),X\right>$.
Therefore we have $\nabla C=-2A({\mathcal T})$.
 \end{proof}

\subsection{An Example of Isoparametric Hypersurface in \texorpdfstring{$\Hyp^n\times \R^m$}{Qne1 x M2} } 

In the statement of Theorem \ref{main}, we have used the
flat horospherical hypersurface in $\Hyp^n\times \R^m$, which is
an example of isoparametric hypersurface  constructed by de Lima and Pipoli in \cite{delima2025}. 
Recall a {\it horosphere} in $\Hyp^n$ is a totally umbilical hypersurface whose induced metric is flat
(see, for example, \cite{BCO}*{Section 1.6.1}).
The following example was given in \cite{delima2025}*{Example 11}:
\begin{example} \label{ex:flat-horosphere}
Choose a horosphere
\(H\subset \Hyp^{n}\) with a unit normal vector field \(\eta\) and an affine hyperplane
\(W\subset \R^{m}\) with a constant unit normal vector field \(\xi\). For any real number \(\alpha\ne 0\),  
a {\it flat horospherical hypersurface} $L_\alpha$ in $\Hyp^n\times \R^m$ is defined as the image of the map
\[ \Phi_\alpha(p,q,t) := \bigl(\exp_p(t \, \eta(p)), \,\, q+ \alpha t \xi \bigr) \in \Hyp^n\times \R^m, 
\hspace{20pt} {\rm for} \,\,\, 
         p\in H, \,\,\, q\in W, \,\,\, t\in \mathbb R.\]
Equivalently, 
\begin{equation} \label{eq:flat-horo-repre}
    L_\alpha=\cup_{t\in \mathbb R} H_{t}\times W_{\alpha t},
\end{equation}
where $H_{t}=\{\exp_{p}(t\eta_{p})\ |\ p\in H\} $ and $W_{\alpha t}=\{q + \alpha t\xi)\ |\ q\in W\} $.
\end{example}

A similar construction with $W \subset \mathbb{R}^m$ replaced by a horosphere 
in $\Hyp^m$ gives an example of isoparametric hypersurface in $\Hyp^n\times \Hyp^m$.
We will not need this example in this paper.

\section{Lie Triple Systems in \texorpdfstring{$\Q^n_{\epsilon_1}\times \Q^m_{\epsilon_2}$}{Qne1 x M2}}\label{sec:lts}

Let $L$ be an isoparametric submanifold in $\Q^n_{\epsilon_1}\times \Q^m_{\epsilon_2}$.
Recall that we have assumed $\epsilon_1 \neq 0$.
Fix an arbitrary point $p \in L$ and let $\m = \nu_p L$.
By Theorem \ref{thm:LieTriple-TotalGeodesic}, $\m$ is a Lie triple system
since it is the tangent space of the totally geodesic section $\Sigma_p$ at
$p$. 
Let
\begin{equation}\label{eq:pi1pi2-definition}
    \pi_1:\Q^n_{\eps_1}\times \Q^m_{\eps_2}\rightarrow\Q^n_{\eps_1},
              \quad \pi_2:\Q^n_{\eps_1}\times \Q^m_{\eps_2}\rightarrow\Q^m_{\eps_2}
\end{equation}
be the natural projection maps.
 We have a decomposition of $\m$ into three mutually orthogonal subspaces
\begin{equation}\label{eq:m0m1m2}
    \m=\m_0\oplus\m_1\oplus\m_2
\end{equation}
where $\m_1:=\{X\in\m:d\pi_2(X)=0\}$, $\m_2:=\{X\in\m:d\pi_1(X)=0\}$, and $\m_0$ is the orthogonal complement of $\m_1\oplus\m_2$ in $\m$. 

\begin{lemma} \label{lem:m0tau}
For $i=1, 2$, $d\pi_i(\m_i)$ is perpendicular to $d\pi_i(\m_0)$.
Moreover, there exists a linear isomorphism $\tau:d\pi_2(\m_0)\rightarrow d\pi_1(\m_0)$ such that
\begin{equation}\label{eq:m0general-fact}
    \m_0=\{(\tau (u), u): u\in d\pi_2(\m_0)\}.
\end{equation}
\end{lemma}

\begin{proof}
Let $i=1$ or $2$. For any $X \in \m_i$ and $Z \in \m_0$,   we have
\begin{equation}\label{eq:Lie-triple-decomposition-property} 
    \langle d \pi_i (X) , d \pi_i (Z) \rangle = \langle X, Z \rangle= 0,
\end{equation}
where the first equality follows from the definition of $\m_i$ and the second equality follows from
the definition of $\m_0$.
Hence $d\pi_i(\m_i) \perp d\pi_i(\m_0)$.

Moreover, by the definition of $\m_0$, $d \pi_i \mid_{\m_0}$ is injective. Hence
$d \pi_i: \m_0 \rightarrow d \pi_i(\m_0)$ is a linear isomorphism since it is apparently also surjective.
The linear isomorphism  
\[ \tau = d \pi_1 \circ (d \pi_2)^{-1}: d\pi_2(\m_0)\rightarrow d\pi_1(\m_0)\] 
satisfies the
requirement of the lemma.
\end{proof}

The following result gives a classification of $\m$.

\begin{proposition} \label{prop:lie-triple}
{\rm (i)} If $\dim \m_0 \geq 1$, then $\m_i = \{0\}$ whenever $\epsilon_i \neq 0$.
{\rm (ii)} If  $\dim \m_0 \geq 2$, then $\epsilon_1 = \epsilon_2$ and the linear map $\tau$ defined 
in Lemma \ref{lem:m0tau} is an isometry.
\end{proposition}

\begin{proof}

Assume $\dim \m_0 \geq 1$. We can choose $Y = (Y_1,Y_2)\in \m_0$ with $Y_1 \neq 0$ and $Y_2 \neq 0$.
If $\epsilon_2 \neq 0$ and $\m_2 \neq \{0\}$, there exists a non-zero vector $X=(0,X_2)\in\m_2$.
By the curvature formula (\ref{eq:curvature-formula}),  
\begin{equation*}
    R(Y, X) X =(0,\eps_2|X_2|^2Y_2),
\end{equation*}
which implies $(0,Y_2)\in\m_2$ by equation \eqref{eq:Lie-triple}. Hence 
$ d \pi_2(\m_0) \bigcap d \pi_2(\m_2) \ni  Y_2 \neq 0 $, which contradicts
 Lemma \ref{lem:m0tau}. Therefore we must have $\m_2 = \{0\}$ if $\epsilon_2 \neq 0$. Similarly,
replacing $X$ by a vector in $\m_1$ in the above arguments, we can show that $\m_1=\{0\}$ as we have assumed $\epsilon_1 \neq 0$. This proves part (i).

Assume $\dim \m_0 \geq 2$. By Lemma \ref{lem:m0tau}, $\dim d \pi_2(\m_0) \geq 2$.
So there exist unit vectors $X_2, Y_2 \in d \pi_2(\m_0) $ such that $\langle X_2, \, Y_2 \rangle =0$.
Let $X_1 = \tau(X_2)$ and $Y_1=\tau(Y_2)$ where $\tau$ is defined in Lemma \ref{lem:m0tau}.
Then vectors $X=(X_1, X_2)$ and $Y=(Y_1, Y_2)$ lie in $\m_0 \subset \m$. Since $\tau$ is a linear isomorphism, $X_1$ and $Y_1$ are linearly independent.
By the curvature formula (\ref{eq:curvature-formula}),  
\begin{equation}\label{eq:lemma3.1}
    R(X, Y) Y=(\eps_1 (|Y_1|^2X_1 -\langle X_1,Y_1\rangle Y_1) , \,\, \eps_2  X_2),
\end{equation}
which also lies in $\m$ by equation (\ref{eq:Lie-triple}). Note that $\eps_2  X \in \m$.
The difference between these two vectors is a vector of the form $Z=(Z_1, 0) \in \m$ with
\[ Z_1 = (\eps_1 |Y_1|^2 - \eps_2 ) X_1 - \eps_1 \langle X_1,Y_1\rangle Y_1. \]
Since $Z \in \m_1$ and $X_1, Y_1 \in d \pi_1(\m_0)$, we have $Z_1 \in  d \pi_1(\m_1) \bigcap d \pi_1(\m_0)$.
By Lemma \ref{lem:m0tau}, $Z_1 = 0$. Since $X_1$ and $Y_1$ are linearly independent, we have
\[ \eps_1 |Y_1|^2 - \eps_2  =0, \hspace{20pt} \eps_1 \langle X_1,Y_1\rangle = 0.\]
Since $\eps_1 \neq 0$ and $\eps_2 \in \{0,  \pm \eps_1\}$ by assumption, we must have $\eps_2 = \eps_1$,
$|Y_1|^2  = 1$, and $\langle X_1,Y_1\rangle = 0$. Switching $X$ and $Y$ in the above arguments also shows
that $|X_1|^2  = 1$. Applying the above arguments to every pair of vectors in an orthonormal basis of
$d \pi_2(\m_0)$ shows that $\tau$ is an isometry. This finishes the proof of part (ii).
\end{proof}

\section{Local and Global Decomposition}\label{sec:localdecomp}


\begin{theorem}\label{thm:lie-triple-stable}
    Let $L$ be an isoparametric submanifold in $\Q^n_{\eps_1}\times\Q^m_{\eps_2}$. Then for any $p \in L$,
    $\m=\nu_p L$ belongs to one of the following types of Lie triple systems with respect to the decomposition \eqref{eq:m0m1m2}:

    \hspace{10pt} {\rm (i)} $\eps_2=0$, $\m=\m_0 \oplus \m_1 \oplus \m_2$ with $\dim \m_0 + \dim \m_1  \leq 1$;

    \hspace{10pt} {\rm (ii)} $\eps_2=0$,  $\m=\m_1\oplus\m_2$ with $\dim\m_1\ge 2$.

    \hspace{10pt} {\rm (iii)} $\eps_2 \neq 0$, $\m=\m_1\oplus\m_2$; 
    
    \hspace{10pt} {\rm (iv)}  $\eps_2 \neq 0$, $\m=\m_0$ with $\dim \m_0=1$;  
     
    \hspace{10pt} {\rm (v)} $\eps_2=\eps_1 \neq 0$ and $\m=\m_0=\{(\tau(u), u): u\in d\pi_2(\m_0)\}$ with $\dim \m_0\ge2$, 
    
    \hspace{28pt} where $\tau: d\pi_2(\m_0) \rightarrow d\pi_1(\m_0)$ is a linear isometry.
    
    \noindent
    We will always assume $L$ is connected. Then for any $p, q \in L$, $\nu_p L$ and $\nu_q L$ belong to the same type.
    
    
   
\end{theorem}

\begin{proof}
The fact that $\m=\nu_p L$ belongs to one of types (i)--(v) follows directly from Proposition~\ref{prop:lie-triple} and our assumption that $\epsilon_{1}\neq 0$.  We only need to show that the type of $\nu_p L$ does not change
     when $p$ varies in $L$. 
      
      We first prove that type (i) and type (ii) can not both occur in the same $L$. 
      In type (i), $\dim d \pi_1(\m) \leq 1$. So any vectors $Y, Z, W \in\m$
      can be written as 
      \[ Y=(a X_1,Y_2), \, Z=(b X_1, Z_2), \, W=(cX_1,W_2)\] 
      for some
       $a, b, c \in \mathbb{R}$, $X_1 \in T_{\pi_1(p)} \Q^n_{\eps_1}$,
        and $Y_2, Z_2, W_2 \in T_{\pi_2(p)} \Q^m_{\eps_2}$. Since $\eps_2 =0$, by curvature equation \eqref{eq:curvature-formula}, we have
        \[ R(Y,Z))W = (\eps_1abc|X_1|^2X_1-\eps_1abc|X_1|^2X_1, \,\,0) =0.\] 
    On the other hand, in type (ii), there exist non-zero vectors $X=(X_1,0), Y=(Y_1,0)\in \m$ with 
    $\langle X_1, \,  Y_1 \rangle =0$. 
    Since $\eps_1 \ne 0$, by curvature equation \eqref{eq:curvature-formula}, we have
    \[ R(X, Y) X = (-\eps_1|X_1|^2Y_1, \,  0) \, \ne 0. \]
    By Corollary \ref{cor:section}, type (i) and type (ii) cannot both occur in the same $L$.

    Next, we show that  type (iii) and type (iv) can not both occur in the same $L$. If codimension
    of $L$ is bigger than one, then only type (iii) can occur and type (iv) never occurs for dimension reason. 
    If $L$ is an isoparametric hypersurface, we can consider the angle function $C$ defined by equation 
    \eqref{eq:angle-function}.
      Since $\dim\m=1$,  in case (iii), we have $\m=\m_1$ or $\m=\m_2$, and $C=\pm 1$. On the other hand, in case (iv), we have $-1<C<1$ since $\m=\m_0$. 
      We will prove in Theorem \ref{thm:constant-angle} that $C$ is constant on any isoparametric hypersurface.
      Hence type (iii) and type (iv) do not occur simultaneously in the same $L$. Note that Theorem \ref{thm:constant-angle} does not depend on results in sections \ref{sec:lts}--\ref{sec:reduction}. Hence there is no risk of circular logic in using Theorem \ref{thm:constant-angle} here. This finishes the discussion for $\eps_2=0$ or $\eps_2 = - \eps_1$.

    If $\eps_2=\eps_1$, $\m$ can only be in types (iii), (iv), or (v). In the above, we have showed that type (iii) and 
    type (iv) can not occur simultaneously. This is also true for type (iv) and type (v) since they have different
    dimensions for $\m$. So  we only need to show that type (iii) and type (v)  can not occur simultaneously.
     Indeed, in type (v), $\exp\m$ is isometric to the diagonal space $\Delta \Q^{\dim(\m)}_{\eps_1}$ as defined in equation (\ref{eq:defn-diagonal-sphere}), which has constant sectional curvature $\eps_1 /2$. On the other hand, in type (iii), $\exp\m$ is isometric to $\Q^{\dim(\m_1)}_{\eps_1}\times\Q^{\dim(\m_2)}_{\eps_2}$. These two spaces do not have isometric neighborhoods since they have different sectional curvatures. 
     By Corollary~\ref{cor:section}, type (iii) and type (v)  can not occur simultaneously. This finishes the proof of the theorem.
\end{proof}

In particular, in type (ii) and type (iii), called the decomposable cases, $\exp\m$ is isometric to $\Q^{\dim(\m_1)}_{\eps_1}\times\Q^{\dim(\m_2)}_{\eps_2}$.  By Proposition \ref{prop:HLO}, the local model of isoparametric submanifolds gives that there is a diffeomorphism $\Phi$ from $U$ to $L'\times \Sigma$, where $L'$ is an open neighborhood of $p$ on $L$, $U$ is a neighborhood of $L'$ in $\Q^n_{\eps_1}\times\Q^m_{\eps_2}$, and $\Sigma$ is a section of $L$ at $p$. For each $q\in L'$ we have a map $F_q$ from $\Sigma$ to $\Phi^{-1}(\{q\}\times \Sigma)\subset \Q^n_{\eps_1}\times \Q^m_{\eps_2}$, sending $x$ to $\Phi^{-1}(q,x)$. $F_q$ is continuous on $q$ by the local model. In particular, $\pi_i\circ F_q$ is continuous on $q$, where $\pi_i$ is defined by \eqref{eq:pi1pi2-definition}. Note that the rank of $d(\pi_i\circ F_q)_p$ is exactly the rank of $d\pi_i|_{\m(q)}$. Here we treat $\m$ (and later $\m_1,\m_2$) as a map on $L$. Therefore by \cite{Pugh2015real}*{Chapter 5 Exercise 43(a)}, the rank of $d(\pi_i\circ F_q)_p$ is lower-semicontinuous on $L'$, for $i=1,2$. On the other hand, $\text{rank} (d\pi_i\m)=\dim(\m_i)+\dim(\m_0)=\dim(\m_i)$ since $\m_0=0$ for type (ii) and type (iii), for $i=1,2$. Hence $\dim(\m_1)$ and $\dim(\m_2)$ are lower-semicontinuous on $L'$. In particular, they are locally constant, since $\dim(\m_1)+\dim(\m_2)=\dim(\Sigma)$ and they only take integer values. Thus $\dim(\m_1)$ and $\dim(\m_2)$ are constant on $L$. Then we have that $\nu L$ decomposes into the direct sum $\nu^1L\oplus\nu^2L$ of two constant rank smooth distributions such that $\nu^1L|_{p}\subset T_{p_1}\Q^n_{\epsilon_1}\oplus \{0\}$ and $\nu^2L|_{p}\subset \{0\}\oplus T_{p_2}\Q^m_{\epsilon_2}$ for any $p=(p_1,p_2)\in L\subset\Q^n_{\eps_1}\times\Q^m_{\eps_2}$. Therefore $TL$ decomposes into the direct sum of two constant rank smooth distributions 
\begin{equation}
    TL=D_1\oplus D_2, \qquad D_1|_p\subset T_{p_1}\Q^n_{\epsilon_1}\oplus \{0\},\qquad D_2|_p\subset \{0\}\oplus T_{p_2}\Q^m_{\epsilon_2}
\end{equation}
for any $p=(p_1,p_2)\in L\subset\Q^n_{\eps_1}\times\Q^m_{\eps_2}$. The following Lemma extends the decomposition on the tangent bundle level to a local decomposition of $L$.

\begin{lemma}\label{lem:localproduct}
    Let $L$ be an isoparametric submanifold in $\Q^n_{\eps_1}\times\Q^m_{\eps_2}$. Assume that we have the decomposition $TL=D_1\oplus D_2$ into direct sum of two constant rank smooth distribution with $D_1|_p\subset T_{p_1}\Q^n_{\epsilon_1}\oplus \{0\}$ and $ D_2|_p\subset \{0\}\oplus T_{p_2}\Q^m_{\epsilon_2}$ for any $p=(p_1,p_2)\in L\subset\Q^n_{\eps_1}\times\Q^m_{\eps_2}$. Then for any $p\in L$, there exists a neighborhood $U$ of $p$ such that $U=L_1\times L_2$, where $L_1\subset \Q^n_{\epsilon_1}$ and $L_2\subset \Q^m_{\epsilon_2}$ are isoparametric submanifolds.
\end{lemma}

\begin{proof}
At any point $q\in L$ the tangent map $(d\pi_i|_L)_q$ restricted to $D_i|_q$ is an isomorphism onto its image, where $\pi_{i}$ is defined in \eqref{eq:pi1pi2-definition} for $i=1,2$. Therefore the rank of $(d\pi_i|_L)_q$ is the rank of $D_i$ which is a constant on $L$, for $i=1,2$. By the rank theorem \cite{lee2003smooth}*{Theorem 4.12} there exists a neighborhood $U_0$ of $p$ on $L$, such that $L_1^0=\pi_1(U_0)$ and $L_2^0=\pi_2(U_0)$ are submanifolds of \(\Q^n_{\epsilon_1}\) and \(\Q^m_{\epsilon_2}\), respectively.

Now consider an $F:U_0\to L_1^0\times L_2^0$ sending $q$ to $(\pi_1(q),\pi_2(q))$, which is a restriction of the inclusion of $L$ into $\Q^n_{\epsilon_1}\times \Q^m_{\epsilon_2}$. Since $L_1^0=\pi_1(U_0)$ and $L_2^0=\pi_2(U_0)$, the differential of $F$ at $p$, $(dF)_p=((d\pi_1)_p,(d\pi_2)_p)$, is an isomorphism from $T_pL=D_{1}|_p\oplus D_{2}|_p$ to $T_{\pi_1(p)}L_1^0\oplus T_{\pi_2(p)}L_2^0$. Hence \(F\) is a local diffeomorphism on an open neighborhood $U_1$ of \(p\). We can find open subsets $L_1\subset L^1_0$ and $L_2\subset L^2_0$ such that $p\in L_1\times L_2\subset F(U_1)$. Denote $F^{-1}(L_1\times L_2)$ by $U$. Then $F|_U$ is a diffeomorphism from the neighborhood $U$ of $p$ on $L$ to $L_1\times L_2$. On the other hand, since $F$ is the restriction of the inclusion $L\rightarrow \Q^n_{\eps_1}\times\Q^m_{\eps_2}$, $F|_U$ is an isometry.

Now we prove that $L_1$ and $L_2$ are isoparametric. Since $U$ splits into $L_1\times L_2\subset \Q^n_{\epsilon_1}\times \Q^m_{\epsilon_2}$, the curvature tensor $R^{\perp}$ on the normal bundle $\nu L$ splits, thus flatness of $\nu L$ implies flatness of $\nu(L_1\subset \Q^n_{\eps_1})$ and $\nu(L_2\subset\Q^m_{\eps_2})$. For any point $q\in U$, the totally geodesic section $\Sigma_q=\exp_q\nu_qL$ splits as $\Sigma_{q_1,1}\times\Sigma_{q_2,2}$ where $q_1=\pi_1(q)$ and $q_2=\pi_2(q)$, since the exponential map and the normal bundle splits. Both $\Sigma_{q_1,1}$ and $\Sigma_{q_2,2}$ are totally-geodesic by property of simply connected space forms \cite{BCO}*{Theorem 1.4.1}. Since the normal bundle $\nu L$ splits as product of flat normal bundles $\nu(L_1\subset \Q^n_{\eps_1})\times\nu(L_2\subset\Q^m_{\eps_2})$, we can collect the lifts of parallel normal frames on $\nu(L_1\subset \Q^n_{\eps_1})$ and $\nu(L_2\subset\Q^m_{\eps_2})$ to form a parallel normal frame on $\Q^n_{\epsilon_1}\times \Q^m_{\epsilon_2}$. For any lifted parallel normal vector field $(\eta_1,0)$, parallel submanifold $U_{(\eta_1,0)}$ for $U$ is simply $L_{1,\eta_1}\times L_2$, where $U_{(\eta_1,0)}$ is $\{\exp_q (\eta_1,0)_q:q\in U\}$ the parallel manifold of $U$ along parallel normal vector field $(\eta_1,0)$, and $L_{1,\eta_1}$ is $\{\exp_{q_1} \eta_1(q_1): q_1\in L_1\}$ the parallel manifold of $L_1$ along $\eta_1$. At any point $q=(q_1,q_2)\in U_{(\eta_1,0)}$, the mean curvature vector of $U_{(\eta_1,0)}$ also splits into $\Q^n_{\epsilon_1}$ component and $\Q^m_{\epsilon_2}$ component, with the former being the mean curvature vector of $L_{1,\eta_1}$ at $q_1$. Therefore, $U_{(\eta_1,0)}$ has constant mean curvature along $(\eta_1,0)$ implies that $L_{1,\eta_1}$ has constant mean curvature along $\eta_1$. Therefore $L_1$ is isoparametric. The same argument holds for $L_2$.
\end{proof}

The following Lemma extend the local decomposition to the global one.

\begin{lemma}\label{lem:gluing}
    Let $L$ be an isoparametric submanifold in $\Q^n_{\epsilon_1}\times \Q^m_{\epsilon_2}$. Suppose $L$ is covered by a collection of product open sets $\{L_{1,\alpha}\times L_{2,\alpha}\}_{\alpha\in I}$ for some index set $I$, where $L_{1,\alpha}\subset \Q^n_{\eps_1}$ and $L_{2,\alpha}\subset \Q^m_{\eps_2}$. Then $L$ is an open subset of $L_1\times L_2$ for some isoparametric submanifolds $L_1\subset \Q^n_{\epsilon_1}$ and $L_2\subset \Q^m_{\epsilon_2}$.
\end{lemma}
\begin{proof}
    By the unique continuation of isoparametric submanifolds in simply connected space forms (\cite{terng1987submanifold}*{Theorem 3.4},\cite{will99isopara}*{Final Remarks}), for any $\alpha,\beta\in I$, $L_{1,\alpha}$ and $L_{1,\beta}$ shares a common maximal extension which is an isoparametric submanifold in $\Q^n_{\epsilon_1}$ if $L_{1,\alpha}\times L_{2,\alpha}$ intersects with $L_{1,\beta}\times L_{2,\beta}$. By connectivity of $L$, a common maximal extension is shared for all $L_{1,\alpha}$ for any $\alpha\in I$. The same applies to $L_{2,\alpha}$'s. Therefore $L\subset L_1\times L_2$ where $L_1=\cup_{\alpha\in I} L_{1,\alpha}$ and $L_2=\cup_{\alpha\in I}L_{2,\alpha}$.
\end{proof}

\section{Higher Codimensional Diagonal Cases}\label{sec:higher-codim-diagonal}

In this section, we treat case (v) of Theorem \ref{thm:lie-triple-stable}, called the higher codimensional diagonal case. In this case, $\exp\m$ is isometric to a subset of $\Delta\Q^{r}_{\eps_1}$ as defined in equation (\ref{eq:defn-diagonal-sphere}). Recall from Section \ref{sec:preliminaries} that we denote by $A$ the shape operator, ${\rm II}$ the second fundamental form of $L$, $\nabla$ the connection on $L$, $\nabla^{\perp}$ the normal connection, $\bar{\nabla}$ the connection on $\Q^n_{\eps_1}\times \Q^m_{\eps_2}$, and $X^T$ the projection of $X\in T(\Q^n_{\eps_1}\times \Q^m_{\eps_2})|_L$ onto $TL$.

In this case, for any point $p\in L$, the Lie triple system $\m=\nu_pL$ takes the form of 
\begin{equation}\label{eq:Tisometry}
    \{(\tau(u),u):u\in d\pi_2(\m)\}
\end{equation} for some isometry $\tau:d\pi_2(\m)\rightarrow d\pi_1(\m)$. Since $\tau$ is an isometry, for any two unit normal vectors $(\tau(u),u)$ and $(\tau(v),v)$ in $\nu_pL$
\begin{equation}\label{eq:property-uTu}
    \langle (\tau(u),u),P(\tau(v),v)\rangle=\langle \tau(u),\tau(v)\rangle-\langle u,v\rangle=0,\qquad |u|=|v|=|\tau(u)|=|\tau(v)|=\dfrac12,
\end{equation}
where $P$ is defined by equation (\ref{equ:P-operator}). Therefore $P(\nu L)\perp\nu L$. Since $\nu L$ is a smooth constant rank distribution and $P$ is a smooth bundle map, $P(\nu L)$ is also such a distribution. We have the following decomposition for $TL$
\begin{equation}\label{eq:TL-decomposition}
    TL=P(\nu L)\oplus F,\qquad \text{where }F:=\Big(\nu L\oplus P(\nu L)\Big)^\perp.
\end{equation}
$F$ is $P$-invariant since its orthogonal complement $\nu L\oplus P(\nu L)$ is $P$-invariant. By equation (\ref{eq:TL-decomposition}) and that $F$ is $P$-invariant, for any $p\in L$, $N=(N_1,N_2)\in\nu_pL$, and $X=(X_1,X_2)\in F_p$, we have $\langle N_1,X_1\rangle=\frac12(\langle X,N\rangle +\langle PX, N\rangle)=0$ and $\langle N_2,X_2\rangle=\frac12(\langle X,N\rangle -\langle PX, N\rangle)=0$. To summarize the above, we have
\begin{equation}\label{eq:property-of-F}
    P(F)=F,\qquad d\pi_i(\nu_p L)=d\pi_i\circ (P(\nu L))|_p\perp d\pi_i(F_p),
\end{equation}
for $i=1,2$, and for any point $p\in L$

The following computations will be used.
\begin{lemma}\label{lem:SS-computations}
For any parallel normal vector fields $\eta$ and $\xi$ on $L$,

    (i) $A_\eta(P\xi)+A_\xi(P\eta)=0$. In particular $A_\eta(P\eta)=0$.

    (ii) $A_\eta(P\xi)$ lies in $F$.
\end{lemma}
    
\begin{proof}
    For (i), differentiating $\langle P\eta,\xi\rangle=0$ along any tangent vector field $X$, using $\nabla^{\perp}\eta=\nabla^\perp\xi=0$ and equation (\ref{eq:P-property}), we have 
\[0=X\langle P\eta,\xi\rangle 
=\langle \bar\nabla_XP\eta,\xi\rangle
  +\langle P\eta,\bar\nabla_X\xi\rangle
  =\langle \bar\nabla_X\eta,P\xi\rangle
  +\langle P\eta,\bar\nabla_X\xi\rangle
=-\langle A_\eta X,P\xi\rangle
  -\langle P\eta,A_\xi X\rangle\]
Using the self-adjointness of the shape operators, we have $0=-\langle X,A_\eta(P\xi)+A_\xi(P\eta)\rangle$. Therefore $A_\eta(P\xi)+A_\xi(P\eta)=0$. In particular $A_\eta(P\eta)=0$.

For (ii), define $B(\eta,\xi,\gamma)=\langle A_\eta(P\xi),P\gamma\rangle$ for parallel normal vector fields $\eta,\xi,\gamma$ on $L$. Self-adjoint-ness of $A_\eta$ gives $B(\eta,\xi,\gamma)=B(\eta,\gamma,\xi)$ while (i) gives $B(\eta,\xi,\gamma)=-B(\xi,\eta,\gamma)$. Therefore, $$B(\eta,\xi,\gamma)=-B(\xi,\eta,\gamma)=-B(\xi,\gamma,\eta)=B(\gamma,\xi,\eta)=B(\gamma,\eta,\xi)=-B(\eta,\gamma,\xi)=-B(\eta,\xi,\gamma)$$ which means $B=0$. Hence $A_\eta(P\xi)$ is orthogonal to $P(\nu L)$. Since it is tangent, it lies in $F$.
\end{proof}

In this section we will prove the following.

\begin{proposition}\label{prop:no-higher-codim-diagonal}
    In $\Q^n_{\eps_1}\times\Q^m_{\eps_2}$ with $\eps_1=\eps_2\ne 0$, there is no submanifold $L$ with codimension greater than $1$, flat normal bundle, and totally geodesic sections that are isometric to a subset of $\Delta \Q^r_{\eps_1}$.
\end{proposition}

For the rest of the section, assume there exists such an $L$. Since $\nu L$ is flat and $\dim \nu_pL\ge 2$ for any $p\in L$, let $N_1$ and $N_2$ be mutually orthogonal parallel local unit normal fields in $\nu L$. Denote the shape operators
$$A_1:=A_{N_1},\qquad A_2:=A_{N_2}.$$
We also consider the following two tangent vector fields:
\begin{equation}\label{eq:defn-of-Ei}
    E_1:=P(N_1),\qquad E_2:=P(N_2).
\end{equation} By equation (\ref{eq:P-property}), we have
\begin{equation}\label{eq:PnuL-property}
    |E_1|=|E_2|=1,\qquad \langle E_1,E_2\rangle=\langle N_1,N_2\rangle=0, \end{equation}
    and
    \begin{equation}\label{eq:dpiiE1E2=0}
    \begin{aligned}
    \langle d\pi_i (E_1|_p),d\pi_i(E_2|_p)\rangle=\dfrac12\Big(\langle E_1|_p,E_2|_p\rangle+(-1)^{i+1}\langle E_1|_p,P(E_2|_p)\rangle\Big)=0,\\
    \langle d\pi_i (N_1|_p),d\pi_i(N_2|_p)\rangle=\dfrac12\Big(\langle N_1|_p,N_2|_p\rangle+(-1)^{i+1}\langle N_1|_p,P(N_2|_p)\rangle\Big)=0,
    \end{aligned}
\end{equation}
for $i=1,2$ and for any point $p\in L$.
By Lemma \ref{lem:SS-computations} (i) and (ii), we define a local vector field
\begin{equation}\label{eq:defn-of-alpha}
    \alpha:=A_{1}E_2=-A_{2}E_1\in \Gamma(F).
\end{equation}

\begin{lemma}\label{lem:SS-computations-2}
    $\bar\nabla_{E_1}E_1=\nabla_{E_1}E_1=\bar\nabla_{E_2}E_2=\nabla_{E_2}E_2=0$, $\bar\nabla_{E_1}E_2=\nabla_{E_1}E_2=P(\alpha)=-\bar\nabla_{E_2}E_1=-\nabla_{E_2}E_1$, and $[E_1,E_2]=2P(\alpha)$.
\end{lemma}
\begin{proof}
    By property of $P$ (\ref{eq:P-property}), for $i=1,2$, 
    \begin{equation}\label{eq:SS-computations-2}
        \bar\nabla_{E_i} E_i=\bar\nabla_{E_i}(PN_i)= P(\bar\nabla_{E_i}N_i)=P(-A_{i}E_i+\nabla^{\perp}_{E_i}N_i).
    \end{equation}
    By Lemma \ref{lem:SS-computations} (i), and $\nabla^\perp N_i=0$ for $i=1,2$, both terms in the last expression of equation (\ref{eq:SS-computations-2}) vanish.
    Therefore $\bar\nabla_{E_1}E_1=\nabla_{E_1}E_1=\bar\nabla_{E_2}E_2=\nabla_{E_2}E_2=0$.

    Again by property of $P$ (\ref{eq:P-property}),
    \begin{equation}\label{eq:SS-computations-2-1}
   \bar\nabla_{E_1}E_2=\bar\nabla_{E_1}(PN_2)=P(\bar\nabla_{E_1}N_2)=P(-A_{2}E_1+\nabla^{\perp}_{E_1}N_2).
    \end{equation} By $\nabla^\perp N_2=0$, the second term in the last expression of equation (\ref{eq:SS-computations-2-1}) vanishes, and the first term is exactly $P(\alpha)$.

    Again by property of $P$ equation (\ref{eq:P-property}),
    \begin{equation}\label{eq:SS-computations-2-2}
   \bar\nabla_{E_2}E_1=\bar\nabla_{E_2}PN_1=P\bar\nabla_{E_2}N_1=P(-A_{1}E_2+\nabla^{\perp}_{E_2}N_1).
    \end{equation} By $\nabla^\perp N_1=0$, the second term in the last expression of equation (\ref{eq:SS-computations-2-2}) vanishes, and the first term is exactly $-P(\alpha)$.

    Since $\alpha\in \Gamma(F)$ and $P$ preserves $F$, $P(\alpha)$ is still in $\Gamma(F)$, in particular it is in $TL$, which implies $\bar\nabla_{E_1}E_2=P(\alpha)=\nabla_{E_1}E_2$ and $\bar\nabla_{E_2}E_1=-P(\alpha)=\nabla_{E_2}E_1$.
\end{proof}
Therefore $$\text{II}(E_1,E_1)=\text{II}(E_1,E_2)=\text{II}(E_2,E_2)=0.$$
The Gauss equation equation (\ref{eq:gauss}) gives the sectional curvature $K_L(E_1,E_2)=K_{\Q^n_{\eps_1}\times \Q^m_{\eps_2}}(E_1,E_2)$. At any fixed point $p=(p_1,p_2)\in L$ we can write 
\begin{equation}\label{eq:notations-u1u2}
    N_1=(\tau(u),u), \qquad N_2=(\tau(v),v)
\end{equation} where $u,v$ are some vectors in $d\pi_2(\nu_pL)$, and $\tau $ defined as in equation (\ref{eq:Tisometry}). Then by equations (\ref{eq:curvature-formula}) and \eqref{eq:dpiiE1E2=0},
\begin{equation}\label{eq:KL=KSS}
    K_{\Q^n_{\eps_1}\times \Q^m_{\eps_2}}(E_1,E_2)=\dfrac{\eps_1(\langle \tau(u),\tau(u)\rangle\langle \tau(v),\tau(v)\rangle-\langle \tau(u),\tau(v)\rangle^2)+\eps_2(\langle u,u\rangle\langle v,v\rangle-\langle u,v\rangle^2)}{|E_1|^2|E_2|^2-\langle E_1,E_2\rangle^2}.
\end{equation}
Apply equations (\ref{eq:property-uTu}), (\ref{eq:PnuL-property}), and $K_L(E_1,E_2)=K_{\Q^n_{\eps_1}\times \Q^m_{\eps_2}}(E_1,E_2)$, to equation (\ref{eq:KL=KSS}) gives
\begin{equation}\label{eq:KL=fraceps2}
    K_L(E_1,E_2)=\dfrac{\eps_1}2.
\end{equation}

On the other hand, by Lemma \ref{lem:SS-computations-2} and equation \eqref{eq:PnuL-property},
\begin{equation}
    K_L(E_1,E_2)=\dfrac{\langle R_L(E_1,E_2)E_2,E_1\rangle}{|E_1|^2|E_2|^2-\langle E_1,E_2\rangle^2}=\langle \nabla_{E_1}\nabla_{E_2}E_2-\nabla_{E_2}(P(\alpha))-\nabla_{2P(\alpha)}E_2,E_1\rangle
\end{equation}
    
    The first term vanishes since $\nabla_{E_2}E_2=0$. For the third term, by equation (\ref{eq:P-property}) and $\nabla^\perp N_2=0$,
    $$\langle\nabla_{P(\alpha)}E_2,E_1\rangle=\langle\bar\nabla_{P(\alpha)}E_2,E_1\rangle=\langle \bar\nabla_{P(\alpha)}N_2, N_1\rangle=\langle\nabla^\perp_{P(\alpha)}N_2, N_1\rangle=0.$$
Therefore by Lemma \ref{lem:SS-computations-2} and equation \eqref{eq:TL-decomposition},
\begin{equation}\label{eq:KL=-alpha2}
    K_L(E_1,E_2)=\langle-\nabla_{E_2}(P(\alpha)),E_1\rangle=\langle P(\alpha),\nabla_{E_2}E_1\rangle=-|P(\alpha)|^2= -|\alpha|^2\leq 0.
\end{equation}

Now if $\eps_1=\eps_2=1$, equations (\ref{eq:KL=fraceps2}) and (\ref{eq:KL=-alpha2}) immediately gives a contradiction, which proves Proposition \ref{prop:no-higher-codim-diagonal}. 
However if $\eps_1=\eps_2=-1$, we only get
\begin{equation}\label{eq:alpha2=frac12}
    |\alpha|^2=\frac12.
\end{equation}
We will need to dive deeper in the $\eps_1=\eps_2=-1$ case.

\subsection{\texorpdfstring{$\eps_1=\eps_2=-1$}{e1=e2=-1} Case}\label{sec:HH-higher-codimension}
We first reduce the general case to the codimension $2$ case, then prove that the codimension $2$ case is impossible.
\begin{lemma}
    If there exists a submanifold $L$ as in Proposition \ref{prop:no-higher-codim-diagonal}, then there exists such an $L'$ with codimension $2$.
\end{lemma}
\begin{proof}
    Let $L$ be a submanifold in $\Q^n_{\eps_1}\times\Q^m_{\eps_2}$ with flat normal bundle and totally geodesic sections, for any $p\in L$, let $\Sigma$ be a section of $L$ at $p$ such that $\exp_p$ is a diffeomorphism on the preimage of $\Sigma$ by $\exp_p$. By \cite{HLO}*{Proposition 2.3}, there exists a neighborhood $U$ of $p$ in $\Q^n_{\eps_1}\times\Q^m_{\eps_2}$ and an isometry
    \begin{equation}\label{eq:defn-of-Phi}
        \Phi: V\times\Sigma\rightarrow U,\qquad \Phi(q_1,q_2)=\exp^U_{q_1}(N_{q_2}|_{q_1}),
    \end{equation}
    where $V$ is a neighborhood of $p$ in $L$, and $N_{q_2}$ is a local parallel normal vector field on $L$ constructed by parallel transport of the tangent vector $\exp^{-1}_p(q_2)\in T_p\Sigma$ in the normal bundle. Here the Riemannian metric on $V\times\Sigma$ is 
    \begin{equation}\label{eq:g-decomposition}
        g_{V\times\Sigma}=g_V(q_2)\oplus g_\Sigma,
    \end{equation} where $g_V(q_2)$ is a Riemannian metric on $V$ changing according to $q_2\in \Sigma$.

    In Proposition \ref{prop:no-higher-codim-diagonal}, the codimension of $L$, denoted by $r$, is assumed to be greater than $1$. If $r=2$ to begin with, then we are done. From now on, assume $r>2$. 

    Let $\{N_1,N_2,...,N_r\}$ be a parallel normal frame on $\nu(V\times\{q_{2}\}\subset V\times \Sigma)$. By \cite{HLO}*{Proposition 2.2}, $\{N_1,N_2,...,N_r\}$ extends to a collection of vector fields in $\Gamma(T(V\times\Sigma))$, still denoted by $\{N_1,N_2,...,N_r\}$, such that its restriction $\{N_1,N_2,...,N_r\}|_{V\times\{q_2\}}$ on the parallel submanifold $V\times \{q_2\}$ for any $q_2\in\Sigma$ is a parallel normal frame on $\nu (V\times \{q_2\})$.  $\exp_{p}\R\{N_1,N_2,...,N_{r}\}$ is isometric to $\Hyp^{r}(\sqrt2)$ by assumption on $L$, then $\exp_{p}\R\{N_1,N_2,...,N_{r-2}\}=:\Sigma'$ is isometric to a $\Hyp^{r-2}(\sqrt2)$ embedded in $\Hyp^{r}(\sqrt2)$ with codimension $2$. We have the following codimension-$2$ submanifold
    \begin{equation}\label{eq:defn-Vsigma}
        V\times \Sigma'\subset V\times \Sigma.
    \end{equation}
    Since $T_{p,p}(V\times \Sigma')=T_pV\oplus \R\{N_1,...,N_{r-2}\}_p$, by equation (\ref{eq:defn-of-Phi}), for any point $q=(q_1,q_2)\in V\times\Sigma'$,
    \begin{equation}\label{eq:RN1N2N3}
        T_{q}(V\times \Sigma')=T_{q_1}V\oplus T_{q}(\{q_1\}\times\Sigma')=T_{q_1}V\oplus \R\{N_1,...,N_{r-2}\}_{q_2}.
    \end{equation}

    We only need to prove that $V\times \Sigma'$ satisfies the assumptions on $L$ from Proposition \ref{prop:no-higher-codim-diagonal}, then $\Phi$ maps $V\times\Sigma'$ to a submanifold in $\Q^n_{\eps_1}\times\Q^m_{\eps_2}$ satisfying the same assumptions.
    
    By properties of simply connected space forms \cite{BCO}*{Theorem 1.4.1}, at any point $q=(q_1,q_2)\in V\times\Sigma'$, there is a submanifold $H_{q}:=\exp_q\R\{N_{r-1},N_r\}$ isometric to $\Hyp^2(\sqrt2)$ that is orthogonal to $\{q_1\}\times\Sigma'$ at $q$, by equation \eqref{eq:RN1N2N3}. For any point $q=(q_1,q_2)\in V\times\Sigma$, $\{q_1\}\times \Sigma$ is a totally geodesic submanifold in $V\times \Sigma$, and $\{q_1\}\times H_q$ is a totally geodesic submanifold in $\{q_1\}\times \Sigma$, then $\{q_1\}\times H_q$ is a totally geodesic submanifold in $V\times \Sigma$. Also, $\{q_1\}\times H_{q}$ is orthogonal to $\{q_1\}\times \Sigma'$ at $q$. $\{q_1\}\times H_{q}$ is also orthogonal to $V\times\{q_2\}$ at $q$ by equation (\ref{eq:g-decomposition}). Therefore $\{q_1\}\times H_{q}$ is orthogonal to $V\times\Sigma'$ since $T_{q}(V\times\Sigma')$ is the direct sum $T_q(\{q_1\}\times \Sigma')\oplus T_q(V\times\{q_2\})$. This concludes the totally geodesic sections assumption on $L$ from Proposition \ref{prop:no-higher-codim-diagonal}. The section $\{q_1\}\times H_{q}$ is isometric to $H_{q}$, which is isometric to $\Hyp^2(\sqrt2)$, which concludes the section type assumption on $L$ from Proposition \ref{prop:no-higher-codim-diagonal}.

    Fix a $q=(q_1,q_2)\in V\times \Sigma'$, denote the Riemannian connection on $V\times\Sigma$ by $\bar\nabla$, the connection on the normal bundle $\nu(V\times\{q_2\})$ by $\nabla^{\perp,0}$, the connection on the normal bundle $\nu(V\times\Sigma')$ by $\nabla^{\perp,1}$, the connection on the normal bundle $\nu(\{q_1\}\times \Sigma'\subset\{q_1\}\times\Sigma)$ by $\nabla^{\perp,2}$, and the connection on $\{q_1\}\times\Sigma$ by $\nabla^{1}$. Also consider the following orthogonal projections $d\pi_1^\perp: T_q(V\times\Sigma)\rightarrow \nu_q(V\times\{q_2\})$, $d\pi_2^\perp:\nu_q(V\times\{q_2\})\rightarrow \nu_q(V\times\Sigma')$, and  $d\pi_0^\perp:T_q(V\times\Sigma)|_{V\times\Sigma'}\rightarrow \nu_q(V\times\Sigma')$. Note that 
    \begin{equation}\label{eq:pi0pi1pi2}
        d\pi^\perp_0=d\pi_2^\perp\circ d\pi_1^\perp.
    \end{equation}
    For any $X\in T_{q}(V\times\Sigma')$, 
    \begin{equation}\label{eq:X=x1+x2}
        X=X_1+X_2
    \end{equation} where $X_1\in T_{q_1}V\oplus\{0_{q_2}\}$ and $X_2\in \{0_{q_1}\}\oplus T_{q_2}\Sigma'$. By the construction equation (\ref{eq:defn-Vsigma}), the normal bundle $\nu(V\times\Sigma')$ is $\R\{N_{r-1},N_r\}|_{V\times\Sigma'}$. Then for $i=r-1,r$, by equations (\ref{eq:pi0pi1pi2}),  (\ref{eq:X=x1+x2}), and the fact that $\{q_1\}\times\Sigma$ is totally geodesic in $V\times\Sigma$, we have \begin{align}
        \nabla^{\perp,1}_XN_{i}=&d\pi_0^\perp(\bar\nabla_XN_i)=d\pi_0^\perp(\bar\nabla_{X_1}N_i+\bar\nabla_{X_2}N_i)
        =d\pi_2^\perp\circ d\pi_1^\perp(\bar\nabla_{X_1}N_i)+d\pi_2^\perp\circ d\pi_1^\perp(\bar\nabla_{X_2}N_i)\nonumber\\
        =&d\pi_2^\perp(\nabla^{\perp,0}_{X_1}N_i)+d\pi_2^\perp\circ d\pi_1^\perp(\nabla^1_{X_2}N_i)=d\pi_2^\perp(\nabla^{\perp,0}_{X_1}N_i)+\nabla^{\perp,2}_{X_2}N_i,
    \end{align}
    in which the first term vanishes by flatness of $\nu(V\times\{q_2\})$, and the second term vanishes by flatness of $\nu(\{q_1\}\times \Sigma'\subset\{q_1\}\times\Sigma)$. We conclude the flat normal bundle assumption on $L$ from Proposition \ref{prop:no-higher-codim-diagonal}.
\end{proof}

Now we may assume that $L$ has codimension $2$.

By equations (\ref{eq:TL-decomposition}) and (\ref{eq:PnuL-property}), we have the following decomposition of $TL$:
\begin{equation}\label{eq:dim2-decomp-of-TL}
    TL=\R E_1\oplus \R E_2\oplus F.
\end{equation}
By Lemma \ref{lem:SS-computations}, we can decompose the shape operators $A_1$ and $A_2$ as
\begin{equation}
    A_1(X)=\langle A_1(X),E_2\rangle E_2+S_1(X),\qquad A_2(X)=\langle A_2(X),E_1\rangle E_1+S_2(X)
\end{equation}
for any $X\in F$, where $S_1$ and $S_2$ are defined by
\begin{equation}\label{eq:defn-of-Si}
    S_1:=\pi_F\circ A_1|_F,\qquad S_2:=\pi_F\circ A_2|_F.
\end{equation} where $\pi_F$ is the bundle map, projection vectors in $TL$ orthogonally into $F$. Note that $S_i$ is symmetric for $i=1,2$. Indeed, for $X,Y\in F$ and $i=1,2$,
\begin{equation*}
    \langle S_i (X),Y\rangle=\langle \pi_F\circ A_i(X),Y\rangle=\langle A_i(X),Y\rangle=\langle X,A_i(Y)\rangle=\langle X,\pi_F\circ A_i(Y)\rangle=\langle X,S_i(Y)\rangle.
\end{equation*}
By self-adjoint-ness of shape operators and equation (\ref{eq:defn-of-alpha}), we also have
\begin{equation}\label{eq:property-of-Si}
    A_1(X)=\langle X,\alpha\rangle E_2+S_1(X),\qquad A_2(X)=-\langle X,\alpha\rangle E_1+S_2(X).
\end{equation}

We will study the behavior of $\alpha$, $S_1$, and $S_2$ via the Ricci equation and the Codazzi equation.

\begin{lemma}\label{lem:HH-computations-2} 
    $[S_1,S_2]=0$.
\end{lemma}
\begin{proof}
    Using the Ricci equation (\ref{eq:ricci}) and the flat normal bundle, we have \begin{equation}\label{eq:ricci-HH}
        \langle \bar R(X,Y)N_1,N_2\rangle+\langle [A_1,A_2]X,Y\rangle=0
    \end{equation} for any tangent vectors $X=(X_1,X_2)$ and $Y=(Y_1,Y_2)$ at any fixed $p\in L$. At $p\in L$, we also set $N_1=(\tau(u),u)$ and $N_2=(\tau(v),v)$ where $u,v$ are defined the same way as equation (\ref{eq:notations-u1u2}). The curvature formula equation (\ref{eq:curvature-formula}) now gives
$$\langle \bar R(X,Y)N_1,N_2\rangle=$$
$$(-\langle Y_1,\tau(u)\rangle\langle X_1,\tau(v)\rangle + \langle X_1,\tau(u)\rangle\langle Y_1,\tau(v)\rangle,-\langle Y_2,u\rangle\langle X_2,v\rangle+\langle X_2,u\rangle\langle Y_2,v\rangle).$$
By equation (\ref{eq:property-of-F}), this is $0$ if $X$ or $Y$ is in $F$. Now, plug in $X=E_1$ and $Y\in F$ to equation (\ref{eq:ricci-HH}), and by equation (\ref{eq:defn-of-alpha}), we have
$$0+\langle -A_1(\alpha)-A_2A_1(E_1),Y\rangle=0.$$ Combining with Lemma \ref{lem:SS-computations-2} and equation (\ref{eq:defn-of-Si}), we have
\begin{equation}\label{eq:S1alpha=0}
    S_1(\alpha)=0.
\end{equation}
Together with equation (\ref{eq:property-of-Si}), we immediately get 
\begin{equation}\label{eq:A1G=frac12E2}
    A_1(\alpha)=\frac12E_2.
\end{equation} Similarly, plug in $X=E_2$ and $Y\in F$ to equation (\ref{eq:ricci-HH}), by equation (\ref{eq:defn-of-alpha}), we have
$$0+\langle A_1A_2(E_2)-A_2(\alpha),Y\rangle=0.$$
By Lemma \ref{lem:SS-computations-2} and equation (\ref{eq:defn-of-Si}), we have
\begin{equation}\label{eq:S2alpha=0}
    S_2(\alpha)=0.
\end{equation}
Combining with equation (\ref{eq:property-of-Si}) and the fact that $|\alpha|^2=\frac12$, we get 
\begin{equation}\label{eq:A2G=-frac12E1}
    A_2(\alpha)=-\frac12E_1.
\end{equation} 
Next we plug in $X$ and $Y$ both in $F$ to equation (\ref{eq:ricci-HH}), we have
\begin{equation}\label{eq:A1A2XY=0}
    0+\langle[A_1,A_2]X,Y\rangle=0.
\end{equation}
Applying equations (\ref{eq:property-of-Si}), \eqref{eq:S1alpha=0}, \eqref{eq:S2alpha=0}, and the self-adjoint-ness of $S_1$ and $S_2$, to equation \eqref{eq:A1A2XY=0}, gives $[S_1,S_2]=0$.
\end{proof}

For $i=1,2$, by equation (\ref{eq:P-property}) and $\nabla^\perp N_i=0$,
$\bar\nabla_\alpha E_i=P(\bar\nabla_\alpha N_i)=-P(A_i\alpha)$, which is $-P(\frac12E_2)$ by equation (\ref{eq:A1G=frac12E2}) if $i=1$, and $-P(-\frac12E_1)$ by equation (\ref{eq:A2G=-frac12E1}) if $i=2$. In both cases they are normal by equation (\ref{eq:defn-of-Ei}). Therefore for $i=1,2,$ 
\begin{equation}\label{eq:nablaalphaEi=0}
    \bar\nabla_\alpha E_1=-\frac12N_2,\qquad \bar\nabla_\alpha E_2=\frac12N_1,\qquad\nabla_\alpha E_1=\nabla_\alpha E_2=0.
\end{equation}

\begin{lemma}\label{lem:HH-computations-3}
    $S_1^2(P(\alpha))=S^2_2(P(\alpha))=\dfrac{P(\alpha)}2$, $S_1S_2(P(\alpha))+S_2S_1(P(\alpha))=0$.
\end{lemma}
\begin{proof}
The proof is done by applying the Codazzi equation (\ref{eq:codazzi}) to $A_i$ for $i=1,2$. Note that since $\nabla^{\perp}N_i=0$, $(\nabla_XA)_{N_i}Y=\nabla_X(A_iY)-A_i(\nabla_XY)-A_{\nabla^{\perp}_XN_i}Y=\nabla_X(A_iY)-A_i(\nabla_XY)$ for any tangent vector field $X,Y$ on $L$ and $i=1,2$. Therefore equation (\ref{eq:codazzi}) on $A_i$ becomes
\begin{equation}\label{eq:codazzi-1}
    \nabla_X(A_iY)-A_i(\nabla_XY)-\nabla_Y(A_iX)+A_i(\nabla_YX)=-(\bar R(X,Y)N_i)^T.
\end{equation} for $i=1,2$. Fix a point $p\in L$, we can write 
\begin{equation}
    (N_1)_p=(\tau(u),u), \qquad (N_2)_p=(\tau(v),v),\qquad (E_1)_p=(\tau(u),-u),\qquad (E_2)_p=(\tau(v),-v),
\end{equation}
where $u=d\pi_2(N_1)_p$, $v=d\pi_2(N_2)_p$, and $\tau $ defined as in equation (\ref{eq:Tisometry}).

Apply the Codazzi equation (\ref{eq:codazzi-1}) to $A_1$ with $X=E_1$, $Y=E_2$ at $p$. For the right hand side of equation (\ref{eq:codazzi-1}), by the curvature formula equation (\ref{eq:curvature-formula}), equations (\ref{eq:property-uTu}), and (\ref{eq:dpiiE1E2=0}),
\begin{equation}
    \bar R(E_1,E_2)N_1=(-\langle \tau(v),\tau(u)\rangle \tau(u)+\langle \tau(u),\tau(u)\rangle \tau(v),-\langle -v,u\rangle (-u)+\langle-u,u\rangle (-v)) =\frac12N_2,
\end{equation}
which is normal. Therefore the right hand side of equation (\ref{eq:codazzi-1}) is $0$. By equation (\ref{eq:defn-of-alpha}), Lemma \ref{lem:SS-computations-2}, and Lemma \ref{lem:SS-computations}, the left hand side of equation (\ref{eq:codazzi-1}) becomes
\begin{equation}
    \nabla_{E_1}\alpha-A_1(P(\alpha))-\nabla_{E_2}0+A_1(-P(\alpha)).
\end{equation}
Therefore  \begin{equation}\label{eq:nablaE1alpha=2A1Palpha}
    \nabla_{E_1}\alpha=2A_1(P(\alpha)).
\end{equation}
Taking inner product with $E_2$, the left hand side of equation (\ref{eq:nablaE1alpha=2A1Palpha}) becomes $\langle \nabla_{E_1}\alpha,E_2\rangle=E_1\langle\alpha,E_2\rangle-\langle\alpha,\nabla_{E_1}E_2\rangle$, which is $-\langle\alpha,P(\alpha)\rangle$ by the decomposition of $TL$ equation (\ref{eq:TL-decomposition}) and Lemma \ref{lem:SS-computations-2}. The right hand side of equation (\ref{eq:nablaE1alpha=2A1Palpha}) becomes $2\langle P(\alpha),\alpha\rangle$ by self-adjoint-ness of $A_1$ and equation \eqref{eq:defn-of-alpha}. Therefore 
\begin{equation}\label{eq:alphaPalpha=0}
    \langle\alpha,P(\alpha)\rangle=0.
\end{equation} Now plug in $X=P(\alpha)$ to equation (\ref{eq:property-of-Si}), we get, for $i=1,2$, 
\begin{equation}\label{eq:AiPalpha=SiPalpha}
    A_i(P(\alpha))=S_i(P(\alpha)).
\end{equation}

Similarly, we apply the Codazzi equation equation (\ref{eq:codazzi-1}) to $A_2$ with $X=E_1$, $Y=E_2$ at $p$. For the right hand side of equation (\ref{eq:codazzi-1}), by the curvature formula equation (\ref{eq:curvature-formula}), equation (\ref{eq:property-uTu}), and (\ref{eq:dpiiE1E2=0}), \begin{equation}
    \bar R(E_1,E_2)N_2=(-\langle \tau(v),\tau(v)\rangle \tau(u)+\langle \tau(u),\tau(v)\rangle \tau(v),-\langle -v,v\rangle (-u)+\langle-u,v\rangle (-v)) =-\frac12N_1,
\end{equation} which is normal. Hence the right hand side of (\ref {eq:codazzi-1}) vanishes.  By equation (\ref{eq:defn-of-alpha}), Lemma \ref{lem:SS-computations-2}, and Lemma \ref{lem:SS-computations}, the left hand side of equation (\ref{eq:codazzi-1}) becomes\begin{equation}
    \nabla_{E_1}0-A_2(P(\alpha))-\nabla_{E_2}(-\alpha)+A_2(-P(\alpha)).
\end{equation}
Therefore \begin{equation}\label{eq:nablaE2alpha=2A2Palpha}
    \nabla_{E_2}\alpha=2A_2(P(\alpha)).
\end{equation}

Next apply the Codazzi equation equation (\ref{eq:codazzi-1}) to $A_1$ with $X=E_1$, $Y=\alpha$ at $p$. Write $\alpha_p=(\alpha_1,\alpha_2)$ where $\alpha_i=d\pi_i(\alpha_p)$ for $i=1,2$. For the right hand side of equation (\ref{eq:codazzi-1}), by the curvature formula equation (\ref{eq:curvature-formula}), equations (\ref{eq:property-uTu}), and (\ref{eq:property-of-F}), we have
\begin{equation}
    \bar R(E_1,\alpha)N_1=(-\langle \alpha_1,\tau(u)\rangle \tau(u)+\langle \tau(u),\tau(u)\rangle \alpha_1,-\langle \alpha_2,u\rangle (-u)+\langle-u,u\rangle \alpha_2) =\frac12P(\alpha),
\end{equation}
 which is tangent. By equations (\ref{eq:A1G=frac12E2}),  (\ref{eq:nablaE1alpha=2A1Palpha}), (\ref{eq:nablaalphaEi=0}), and Lemma \ref{lem:SS-computations}, the left-hand side of equation (\ref{eq:codazzi-1}) gives, 
 $$\nabla_{E_1}(\frac12E_2)-2A_1^2(P(\alpha))-\nabla_\alpha0+A_1(0).$$
 Then apply equation (\ref{eq:SS-computations-2}) to (\ref{eq:codazzi-1}) gives
 $$\frac12P(\alpha)-2A_1^2(P(\alpha))=-\frac12P(\alpha).$$ By equation (\ref{eq:AiPalpha=SiPalpha}), we have $2S_1^2(P(\alpha))=P(\alpha)$.

Next apply the Codazzi equation equation (\ref{eq:codazzi-1}) to $A_2$ with $X=E_2$, $Y=\alpha$ at $p$.
For the right hand side of equation (\ref{eq:codazzi-1}), by the curvature formula equation (\ref{eq:curvature-formula}), equations (\ref{eq:property-uTu}), and  (\ref{eq:property-of-F}), we have
\begin{equation}
    \bar R(E_2,\alpha)N_2=(-\langle \alpha_1,\tau(v)\rangle \tau(v)+\langle \tau(v),\tau(v)\rangle \alpha_1,-\langle \alpha_2,v\rangle (-v)+\langle-v,v\rangle \alpha_2) =\frac12P(\alpha),
\end{equation}
 which is tangent. Here we take a fixed point $p\in L$ and write $\alpha_p=(\alpha_1,\alpha_2)$ for the computation, but the computation holds for all $p$. By equations (\ref{eq:A2G=-frac12E1}),  (\ref{eq:nablaE2alpha=2A2Palpha}),  (\ref{eq:nablaalphaEi=0}), and Lemma \ref{lem:SS-computations}, the left-hand side of equation (\ref{eq:codazzi-1}) gives, 
 $$\nabla_{E_2}(-\frac12E_1)-2A_2^2(P(\alpha))-\nabla_\alpha0+A_2(0).$$
 Then by  Lemma \ref{lem:SS-computations-2}, equation (\ref{eq:codazzi-1}) gives
 $$\frac12P(\alpha)-2A_2^2(P(\alpha))=-\frac12P(\alpha).$$ By equation (\ref{eq:AiPalpha=SiPalpha}), we have $2S_2^2(P(\alpha))=P(\alpha)$.

Apply the Codazzi equation equation (\ref{eq:codazzi-1}) to $A_1$ with $X=E_2$, $Y=\alpha$ at $p$. The curvature term vanishes, since by equation (\ref{eq:dim2-decomp-of-TL}), $\alpha\in F$, $E_2$, and $N_1$ are mutually orthogonal componentwise. Apply equations (\ref{eq:A1G=frac12E2}), (\ref{eq:nablaE2alpha=2A2Palpha}), (\ref{eq:nablaalphaEi=0}), and (\ref{eq:defn-of-alpha}) to (\ref{eq:codazzi-1}) gives $$\nabla_{E_2}(\frac12E_2)-2A_1\circ A_2(P(\alpha))-\nabla_\alpha\alpha+A_1(0)=0,$$ which by Lemma \ref{lem:SS-computations-2} gives $2A_1A_2(P(\alpha))=-\nabla_\alpha\alpha$, which implies $2A_1S_2(P(\alpha))=-\nabla_\alpha\alpha$ by equation (\ref{eq:AiPalpha=SiPalpha}), then $2S_1S_2(P(\alpha))=-\nabla_\alpha\alpha$ by equation \eqref{eq:defn-of-Si}, $S_i$ is symmetric, and equation \eqref{eq:S2alpha=0}.

At last, apply the Codazzi equation equation (\ref{eq:codazzi-1}) to $A_2$ with $X=E_1$ and $Y=\alpha$ at $p$. The curvature term vanishes, since by equations (\ref{eq:property-of-F}) and (\ref{eq:dpiiE1E2=0}), $\alpha\in F$, $E_2$, and $N_1$ are mutually orthogonal componentwise. Apply equations (\ref{eq:A2G=-frac12E1}),  (\ref{eq:nablaE1alpha=2A1Palpha}),  (\ref{eq:nablaalphaEi=0}), and  (\ref{eq:defn-of-alpha}) to (\ref{eq:codazzi-1}) gives 
$$\nabla_{E_1}(-\frac12E_1)-2A_2A_1(P(\alpha))-\nabla_\alpha(-\alpha)+A_2(0)=0,$$
which by Lemma \ref{lem:SS-computations-2} gives $2A_2A_1(P(\alpha))=\nabla_\alpha\alpha$, which implies $2A_2S_1(P(\alpha))=\nabla_\alpha\alpha$ by by equation (\ref{eq:AiPalpha=SiPalpha}), then $2S_2S_1(P(\alpha))=\nabla_\alpha\alpha$ by equation \eqref{eq:defn-of-Si}, $S_i$ is symmetric, and equation \eqref{eq:S1alpha=0}. Therefore $2S_2S_1(P(\alpha))+2S_1S_2(P(\alpha))=\nabla_\alpha\alpha-\nabla_\alpha\alpha=0$.
\end{proof}

Now we can finally prove Proposition \ref{prop:no-higher-codim-diagonal}. By Lemma \ref{lem:HH-computations-2} and Lemma \ref{lem:HH-computations-3}, $$S_1S_2(P(\alpha))=\frac12S_1S_2(P(\alpha))+\frac12S_2S_1(P(\alpha))+\frac12[S_1,S_2](P(\alpha))=0.$$
Therefore $0=S_1^2S_2(P(\alpha))=S_2S_1^2(P(\alpha))=\frac12S_2(P(\alpha))$ by $S_1^2(P(\alpha))=\frac12P(\alpha)$ and $[S_1,S_2]=0$. Therefore $S_2^2(P(\alpha))=0$, which implies $P(\alpha)=0$ since $2S^2_2(P(\alpha))=P(\alpha)$, which contradicts equation (\ref{eq:alpha2=frac12}), since $|P(\alpha)|^2=|\alpha|^2$. This contradiction rules out
codimension $2$, and by the reduction above it rules out all codimensions $r>1$.

\section{Reduction to Isoparametric Hypersurfaces}\label{sec:reduction}

In this section, we treat case (i) of Theorem \ref{thm:lie-triple-stable}. In this case, $\Q^n_{\eps_1}\times\Q^m_{\eps_2}$ can be written as $\Q^n_\eps\times\R^m$, and by definition of $\m_0,\m_1,$ and $\m_2$, a Lie triple system corresponding to a totally geodesic section of $L$ is one of the following type:
\begin{equation}\label{eq:caseitypes}
    \begin{aligned}
        (a)&\ \m_0\oplus \m_2 \text{ with }\dim \m_0=1;\\
        (b)&\ \m_1\oplus \m_2 \text{ with }\dim \m_1=1;\\
        (c)&\ \m_2.
    \end{aligned}
\end{equation}
\begin{lemma}\label{lem:no-type-a-decompose}
    If there is no type (a) Lie triple system, then $L$ is an open subset of $L_1\times L_2$ where $L_1\subset \Q^n_{\eps_1}$ and $L_2\subset\Q^m_{\eps_2}$ are isoparametric submanifolds.
\end{lemma}
\begin{proof}
By Proposition \ref{prop:HLO}, the local model of isoparametric submanifolds gives that there is a diffeomorphism $\Phi$ from $U$ to $L'\times \Sigma$, where $L'$ is an open neighborhood on $L$, $U$ is a neighborhood of $L'$ in $\Q^n_\eps\times\R^m$, and $\Sigma$ is the section of $L$ at $p$, for some $p\in L$. For each $q\in L$ we have a map $F_q$ from $\Sigma$ to $\Phi^{-1}(\{q\}\times \Sigma)\subset \Q^n_\eps\times \R^m$, sending $x$ to $\Phi^{-1}(q,x)$. $F_q$ is continuous on $q$ by the local model. In particular, $\pi_i\circ F_q$ is continuous on $q$, where $\pi_i$ is defined by \eqref{eq:pi1pi2-definition}. Note that the rank of $d(\pi_i\circ F_q)_p$ is exactly the rank of $d\pi_i|_{\m(q)}$. Here we treat $\m$ (and later $\m_1,\m_2$) as a map on $L$. Therefore by \cite{Pugh2015real}*{Chapter 5 Exercise 43}, the rank of $d(\pi_i\circ F_q)_p$ is lower-semicontinuous on $L$, for $i=1,2$. On the other hand, $\text{rank} (d\pi_i\m)=\dim\m_i+\dim\m_0=\dim\m_i$ since $\m_0=0$ for types (b) and (c), for $i=1,2$. Hence $\dim\m_1$ and $\dim\m_2$ lower-semicontinuous on $L$. Since $\dim\m_1$ is lower-semicontinuous, for any point of type (c), nearby points has lower or equal $\dim\m_1$, hence not in type (b). Since $\dim\m_2$ is lower semi-continuous, for any point of type (b), nearby points has lower or equal $\dim\m_2$, hence not in type (c). By connectivity of $L$, these two types cannot occur in the same $L$. In both cases, since $\nu L$ decomposes as $\m_1\oplus\m_2$, $TL$ decomposes globally as $P_1(TL)\oplus P_2(TL)$ a direct sum of two constant rank distributions, where $P_1$ and $P_2$ are defined by equation \eqref{eq:pi1pi2-definition}. Then Lemma \ref{lem:localproduct} and Lemma \ref{lem:gluing} finish the proof.
\end{proof}

Therefore we may assume that there is a point $p\in L$ with  type (a) Lie triple system. Under this assumption we have the following Theorem.

\begin{theorem}\label{thm:reduction}
Let $L$ be an isoparametric submanifold in $\Q^n_{\eps}\times\R^m$ ($\varepsilon=\pm1$) such that its section at a point $p^0=(p^0_1,p^0_2)\in L\subset \Q^n_{\eps}\times\R^m$ corresponds to a Lie triple system of the form $\m=R(v,w)\oplus(0\oplus B)$, where $v$ and $w$ are nonzero vector in $T_{p^0_1}\Q^n_{\eps}$ and $T_{p^0_2}\R^m$, and $B$ is a subspace of $T_{p^0_2}R^m$ with $w\notin B$. Then up to an isometry induced by translation and rotation in the Euclidean factor, around $p^0$, $L$ is locally an open subset of $L_1\times L_2\times\{0\}\subset Q^n_\eps \times \R^{m_1}\times \R^{m_2}\times\R^{m_3}$ where $L_1\subset Q_\varepsilon^n\times\R^{m_1}$ is an isoparametric hypersurface in $Q_\eps^n\times R^{m_1}$ with its angle function satisfies that $-1<C<1$, and $L_2\subset\R^{m_2}$ is an open set of a full compact isoparametric submanifold. Here $m=m_1+m_2+m_3$, and the angle function is defined by equation \eqref{eq:angle-function}.
\end{theorem}

Fix $p^0$. Since the result is local around $p^0$, we may restrict $L$ to a smaller neighborhood around $p^0$. This is done several times in this section, and for simplicity, we will still denote the smaller neighborhood as $L$. For starter, we restrict to a smaller $L$ so that $\nu L$ is globally flat, that is, any normal vector can be extended to a global parallel normal vector field.

We first introduce some notations. Note that they can be defined when $L$ is an isoparametric hypersurface as well, in which case, the $E$ below will simply be a trivial distribution.

For a vector field $X$ in $\Gamma(T(\Q^n_{\eps}\times\R^m))$, we have
\begin{equation}\label{eq:factor-components}
X=P_1(X)+P_2(X).
\end{equation}
Here $P_1(X_q)\in T_{q_1}\Q^n_\eps\oplus 0$, $P_2(X_q)\in 0\oplus T_{q_2}\R^m$, for any $q=(q_1,q_2)\subset \Q^n_{\eps}\times\R^m$.

Let $E:=\ker(d\pi_1|_{\nu L})$. There is a neighborhood of $p^0$ in $L$, still denoted by $L$, such that on $L$, the Lie triple system decomposes as $\m=\m_0\oplus\m_2$ with $\dim\m_0=1$, and there is a unit normal field $N$ 
along $L$ with the property that, 
\begin{equation}\label{eq:normal-decomposition}
\nu L=\mathbb{R}N\oplus E,
\end{equation}
and for any $q=(q_1,q_2)\in L$,
\begin{equation}\label{eq:diagonal-normal}
\begin{aligned}
E&=\ker(d\pi_1|_{\nu L}),\qquad
N=a\nu_1+b\nu_2,\\
\nu_1|_q&\in T_{q_1}Q_\varepsilon^n\oplus0,\qquad \nu_2|_q\in0\oplus T_{q_2}\mathbb{R}^m,\\
E_q&\subset (0\oplus T_{q_2}\mathbb{R}^m)\cap (\nu_2|_q)^\perp.
\end{aligned}
\end{equation}
Here we can choose $N$ so that $N_{p^0}=\frac{(v,w)}{|(v,w)|}\in \m$ as appeared in Theorem \ref{thm:reduction}, and $L$ small enough so that $a$ and $b$ are non-zero on $L$. The smooth-ness of $E$ follows from that it is the constant-rank kernel of $d\pi_1|_{\nu L}$, and the smooth-ness of $N$ follows from the fact that $\R N$ is the orthogonal complement of $E$ in $\nu L$. 

The diagonal tangent direction paired with $N$ is
\begin{equation}\label{eq:diagonal-tangent}
{\mathcal T}=b\nu_1-a\nu_2.
\end{equation}
One can check that the ${\mathcal T}$ here is in fact the normalized version of ${\mathcal T}$ defined in equation \eqref{eq:T-defn}. In this section, for computational simplicity, we use the normalized ${\mathcal T}$ here. We have that, ${\mathcal T}$ has unit length, and is tangent to $L$, since $\nu_1,\nu_2\perp E$ and $\mathcal T\perp N$. Moreover
\begin{equation}\label{eq:rotation}
 \nu_1=aN+b{\mathcal T},
 \qquad
 \nu_2=bN-a{\mathcal T}.
\end{equation}
The tangent bundle decomposes orthogonally as
\begin{equation}\label{eq:tangent-decomposition}
TL=D_1\oplus\mathbb{R}{\mathcal T}\oplus D_2,
\end{equation}
where, at $q=(q_1,q_2)$,
\begin{equation}\label{eq:D1D2-definitions}
(D_1)_q=T_qL\cap (T_{q_1}Q_\varepsilon^n\oplus0),
\qquad
(D_2)_q=T_qL\cap (0\oplus T_{q_2}\mathbb{R}^m).
\end{equation}
Equivalently, $(D_1)_q$ is the orthogonal complement of $\nu_1|_q$ inside $T_{q_1}Q_\varepsilon^n\oplus0$, while $(D_2)_q$ is the orthogonal complement of $E_q\oplus\mathbb{R}\nu_2|_q$ inside $0\oplus T_{q_2}\mathbb{R}^m$. We also have that 
\begin{equation}\label{eq:dpi-blocks}
d\pi_{\mathbb{R}^m}(D_1)=0,
\qquad
d\pi_{\mathbb{R}^m}({\mathcal T})=-a \nu_2
\end{equation}
where $\pi_{\R^m}$ is the projection from $\Q^n_{\eps}\times\R^m$ to $\R^m$.
Since  $a\ne0$, $\ker(d\pi_{\mathbb{R}^m}|_{TL})=D_1$. Hence $D_1$ is integrable. Since $D_1$ is tangent to the fibers of the constant-rank map $\pi_{\mathbb{R}^m}|_L$, its integrated manifolds lie in first-factor slices of the form $\Q^n_\eps\times\{z\}$ and are hypersurfaces of $\Q^n_\eps$.

If a vector $Z$ on $L$ lies in $0\oplus T\R^m$ and decomposes as $Z=Z_2+l\nu_2+Z_E$, with $Z_2\in D_2$, $Z_E\in E$, and $l$ a real number, then equation \eqref{eq:rotation} gives a projection formula
\begin{equation}\label{eq:projection-formula}
Z^\top=Z_2-al{\mathcal T},
\qquad
Z^\perp=Z_E+blN.
\end{equation}

We shall use curvature distributions. For normal vector fields $\rho,\sigma\in \nu L$, then $P_1\rho$ and $P_1\sigma$ are both scalars of $\nu_1$. The product curvature tensor therefore gives
\begin{equation}\label{eq:ricci-curvature-term-zero}
\langle \overline R(X,Y)\rho,\sigma\rangle=0
\end{equation}
for all tangent vectors $X,Y$, because the only nonzero curvature contribution comes from the $Q_\varepsilon^n$-factor and the two normal first-factor components are scalar of each other. The Ricci equation equation \eqref{eq:ricci} and the flat-ness of normal bundle then gives
\begin{equation}\label{eq:shape-commute}
[A_\rho,A_\sigma]=0.
\end{equation}
There exists a neighborhood in $L$, still denoted by $L$, such that on $L$ we have smooth simultaneous eigendistributions $\mathcal E_1,\ldots,\mathcal E_s$ for all shape operators of $L$ for some integer $s$
\begin{equation}\label{eq:simultaneous-eigendistribution}
    TL=\oplus_{i=1}^s \mathcal E_i.
\end{equation}
For each $\mathcal E_i$, there is a vector field $\eta_i\in\Gamma(\nu L)$, called the curvature normal, characterized by
\begin{equation}\label{eq:curvature-normal}
A_\rho|_{\mathcal E_i}=\langle\rho,\eta_i\rangle\operatorname{id}_{\mathcal E_i}
\end{equation}
for any $\rho\in\Gamma(\nu L)$. Equivalently, 
\begin{equation}\label{eq:alphaXY=}
    \text{II}(X,Y)=\langle X,Y\rangle\eta_i,\qquad \text{II}(X,Z)=0,
\end{equation}for $X,Y\in \Gamma(\mathcal E_i)$, and $Z\in \Gamma(\mathcal E_j)$ with $i\ne j$, where $\text{II}$ is the second fundamental form on $L$.

\subsection{Determinant Rigidity}\label{subsec:determinant-rigidity}


Fix $q\in L$. Let $\zeta\in\Gamma(\nu L)$ be a parallel normal vector field. Write $P_1\zeta=h\nu_1$ where $h$ is a real-valued function on $L$. Along the normal geodesic $\exp_q t\zeta_q$, by the curvature formula equation \eqref{eq:curvature-formula} and the fact that the second factor of the ambient space is Euclidean, the curvature operator $X\mapsto\overline R(X,\zeta)\zeta$ equals $\varepsilon h^2I$ on $(D_1)_q$ and vanishes on $(\mathbb{R}{\mathcal T}\oplus D_2)_q$. Let $Y$ be a Jacobi field along $\exp_pt\zeta_p$ with initial value $Y(0)$ and $Y'(0)=-A_\zeta Y(0)$. Solving the Jacobi equation \eqref{eq:jacobi-equation} for $Y$ gives 
$Y(t)=c_\eps(t)\Psi_t (Y(0))-s_\eps(t)\Psi_t(A_\zeta Y(0))$ for 
$Y(0)\in D_1$, and $Y(t)=\Psi_t (Y(0))-t\Psi_t (A_\zeta Y(0))$ for $Y(0)\in D_2$, 
where $\Psi_t (\cdot)$ is the parallel transport from $T_pL$ to $\exp_pt\zeta$, where, for $h\neq0$,
\begin{equation}\label{eq:cs-functions}
(c_\varepsilon,s_\varepsilon)=
\begin{cases}
\left(\cos(ht),\dfrac{\sin(ht)}{h}\right),&\varepsilon=1,\\[1.1em]
\left(\cosh(ht),\dfrac{\sinh(ht)}{h}\right),&\varepsilon=-1.
\end{cases}
\end{equation} 

In other words, the operator $J_\zeta(q,t)$ ($J(t)$ for short)  defined by
$J(t)(Y(0)) := \Psi_t^{-1}(Y(t))$  has the form
\begin{equation}\label{eq:Jacobi}
J(t)=c_\varepsilon(t)P_{D_1}+P_{\mathbb{R}{\mathcal T}\oplus D_2}
-\bigl(s_\varepsilon(t)P_{D_1}+tP_{\mathbb{R}{\mathcal T}\oplus D_2}\bigr)A_\zeta,
\end{equation}
where $P_{D_1}$ and $P_{\mathbb{R}{\mathcal T}\oplus D_2}$ are the projections from $TL$ to $D_1$ and $\mathbb{R}{\mathcal T}\oplus D_2$, respectively. Set 
\begin{equation}\label{eq:h=0cases}
    c_\eps=1, \quad s_\eps=t, \quad \text{if }h=0.
\end{equation}

\begin{lemma}\label{lem:det-rigidity}
Assume $h\neq0$. If $J(t)$ is given by equation \eqref{eq:Jacobi} and $\det J(t)$ is a polynomial in $t$, then it is not possible that $\varepsilon=1$. If $\varepsilon=-1$, then
\begin{equation}\label{eq:det-rigidity-conclusion}
A_\zeta (D_1)\subset D_1,
\qquad
(A_\zeta|_{D_1})^2=h^2I,
\end{equation}
and the $+h$- and $-h$-eigenspaces of $A_\zeta|_{D_1}$ have the same dimension.
\end{lemma}

\begin{proof}
Take orthonormal basis of $D_1|_q$ and $(\mathbb{R}{\mathcal T}\oplus D_2)_q$, respectively, which combined to be an orthonormal basis on $T_qL$. Write $A=A_\zeta$ in terms of this orthonormal basis in block matrix form relative to $D_1|_q\oplus(\mathbb{R}{\mathcal T}\oplus D_2)_q$:
\begin{equation}\label{eq:A-blocks}
A=\begin{pmatrix}A_1&A_2\\ A_2^T&A_3\end{pmatrix},
\end{equation}
with $A_1,A_3$ self-adjoint. In particular we choose orthonormal basis so that $A_1$ and $A_3$ are diagonal. Therefore 
\begin{equation}\label{eq:Jt-blocks}
    J(t)=\begin{pmatrix}c_\eps(t) I-s_\eps(t) A_1& -s_\eps(t) A_2\\ -tA_2^T&I-tA_3\end{pmatrix},
\end{equation}
by equation \eqref{eq:Jacobi}, where $I$ is the identity matrix in respective blocks. Therefore the determinant of $J(t)$ is a finite sum of terms $t^re^{\lambda t}$, and these are linearly independent functions of $t$ for distinct pairs $(r,\lambda)$ by Lemma \ref{lem:spectrum-unique}. By our assumption that $J(t)$ is a polynomial, all terms with $t^re^{\lambda t}$ with $\lambda\ne0$ vanishes.

From equation \eqref{eq:Jacobi}, the coefficient of the $t^0$ term in $\det J(t)$ is
\begin{equation}\label{eq:t0-coefficient}
\det(c_\varepsilon(t)I-s_\varepsilon(t)A_1),
\end{equation}
because every contribution from the lower-left block or from $A_3$ carries a factor of $t$. If $\lambda_1,\ldots,\lambda_d$ are the eigenvalues of $A_1$, then $(c_\eps(t)-s_\eps(t)\lambda_k)$ is an eigenvalue of $(c_\eps(t)I-s_\eps(t)A_1)$ for $k=1,2,...,d$. Here $d$ is the dimension of $D_1$. Then if $\eps=1$, the diagonal terms in equation \eqref{eq:t0-coefficient} are
\begin{equation}\label{eq:spherical-factor}
\cos(ht)-\frac{\lambda_k}{h}\sin(ht)
=\frac{1+\ii\lambda_k/h}{2}e^{\ii ht}+\frac{1-\ii\lambda_k/h}{2}e^{-\ii ht}.
\end{equation}
for $k=1,2,...,d$. Both coefficients of exponential terms are nonzero because $\lambda_k$ is real, so equation \eqref{eq:t0-coefficient} has nonzero $e^{\ii hd t}$ terms. Hence by Lemma \ref{lem:spectrum-unique}, $\det (c_\eps(t)I-s_\eps(t)A_1)$ is not a polynomial, which implies that $\varepsilon=1$ is impossible.

For $\varepsilon=-1$, the diagonal terms of equation \eqref{eq:t0-coefficient} are
\begin{equation}\label{eq:hyperbolic-factor}
\cosh(ht)-\frac{\lambda_k}{h}\sinh(ht)
=\frac{1-\lambda_k/h}{2}e^{ht}+\frac{1+\lambda_k/h}{2}e^{-ht}.
\end{equation}
Putting $x=e^{ht}$, since $A_1$ is symmetric, we may apply Lemma \ref{lem:S2=a2I-eigenvalues-a--a} to $\det(c_{-1}(t)I-s_{-1}(t)A_1)$, which gives that the eigenvectors of $A_1$ are $\lambda_k=\pm h$, and the two signs have equal multiplicity.

Let $D_1^+=\ker(A_1-hI)$ and $D_1^-=\ker(A_1+hI)$, and set $A_{2,+}=I_+  A_2$, $A_{2,-}=I_-  A_2$, where $I_\pm$ is the diagonal matric with  entries $1$ on rows corresponding to $D_1^\pm$ and entries $0$ on rows corresponding to $D_1^\mp$. Since $D_1^+\perp D_1^-$, $I_+I_-=0$, consequently
\begin{equation}\label{eq:cross-terms-zero}
A_{2,+}^TA_{2,-}=A_{2,-}^TA_{2,+}=0.
\end{equation}
Note that $I=I_++I_-$ and $A_2=A_{2,+}+A_{2,-}$.
$(c_{-1}I-s_{-1}A_1)$, the upper-left block of $J(t)$, is a diagonal matrix $e^{-ht}I_+ +e^{ht}I_-$. Then its inverse is $e^{ht}I_++e^{-ht}I_-$. Its determinant is $1$ because the two eigenspaces $D_\pm$ have the same dimension. Therefore the Schur complement formula \cite{horn2005basic}*{Theorem 1.1}  gives
\begin{align}\label{eq:Schur}
\det J(t)=&\det\left(I-tA_3-(-t\big(A_{2,+}^T+A_{2,-}^T\big))(e^{ht}I_++e^{-ht}I_-)(-s_{-1}\big(A_{2,+}+A_{2,-}\big))\right)\nonumber\\
=&\det\left(I-t\left(A_3+\frac{e^{2ht}-1}{2h}
A_{2,+}^TA_{2,+}
+\frac{1-e^{-2ht}}{2h}A_{2,-}^TA_{2,-}
\right)\right)
\end{align}
where equation \eqref{eq:cross-terms-zero} removes the cross terms. Denote $(A_3+\frac{e^{2ht}-1}{2h}
A_{2,+}^TA_{2,+}
+\frac{1-e^{-2ht}}{2h}A_{2,-}^TA_{2,-})$ by $G(t)$, whose entries does not contain power of $t$. Then $\det J(t)=1+t \tr G(t)+R(t)$, where $R(t)$ consists of terms with $t^k$ coefficient with $k\ge 2$. The terms of the form $t e^{2ht}$ and $t e^{-2ht}$ in equation \eqref{eq:Schur} can only be found in the $t \tr G(t)$ term, and their coefficient are respectively $-\operatorname{tr}(A_{2,+}^TA_{2,+})/(2h)$ and $\operatorname{tr}(A_{2,-}^TA_{2,-})/(2h)$, by equation \eqref{eq:Schur}. Polynomiality of $\det J(t)$ forces both to vanish, and since $h\ne0$, we have $A_{2,+}=A_{2,-}=0$. Hence $A_2=0$.
\end{proof}

\subsection{Parallel-ness of the \texorpdfstring{$E$}{E}}\label{subsec:parallel-ness}
The following Lemma will be useful. The Lemma has a different assumption on $\eps_1$ and $\eps_2$, as opposed to the general assumption $\eps_2=0$ of this section, for it is used in later sections as well.
\begin{lemma}\label{lem:cartanformula}
    Let $L\subset \Hyp^n\times \Q^m_\eps$ be an isoparametric submanifold with $\dim(P_1(\nu L))=1$. Suppose that there is a point $p\in L$ and  a local unit normal vector field $N=a\nu_1+b\nu_2$  around $p$ on $L$, where $N$ is perpendicular to $E:=\ker(d\pi_1|_{\nu L})$ in $\nu L$, $\nu_1$ lies in $T\Hyp^n\oplus 0$, $\nu_2$ lies in $0\oplus T\Q^m_\eps$ are unit vector fields on $T(\Hyp^n\times \Q^m_\eps)|_L$, and $a,b$ are non-zero functions. If $A_N(D_1)\subset D_1$ and $A_N|_{D_1}$ has eigenvalue $a$ and $-a$ with the same multiplicities, then such $L$ does not exists. In particular, if $L$ is an isoparametric hypersurface satisfying the same assumption, $N$ is a local unit normal vector field of $L$,  and $A_N(D_1)\subset D_1$ and $A_N|_{D_1}$ has eigenvalue $a$ and $-a$ with the same multiplicities, then such $L$ does not exists.
\end{lemma}
\begin{proof}
    
Since $D_{1}=\ker(P_{2}|_{TL})$, it is 
integrable. Let $M_{0}$ be an integrated submanifold of $D_{1}$, then $M_{0}$ lies in a slice
$\Hn^n\times \{y_0\}$, which is totally geodesic in $\Hyp^{n}\times \Q^m_\eps$.
For any tangent vector $X,Y\in D_{1}$, using the fact that $\nu_2$ lies in $0\oplus T\Q^m_\eps$, we have 
$$\langle A_NX,Y\rangle=\langle -\bar\nabla_XN,Y\rangle=\langle-\bar\nabla_X(a\nu_1),Y\rangle=\langle -X(a)\nu_1,Y\rangle-\langle a\bar\nabla_X\nu_1,Y\rangle.$$
Here the first term vanishes since $D_1\perp \nu_1$, and the second term gives $a\langle A^{M_0}X,Y\rangle$, where \(A^{M_0}\) is the shape operator of $M_0$ in $\Hyp^n\times y_{0}$ along $\nu_1|_{M_0}$.
By  our assumptions, the eigenvalues of \(A|_{D_{1}}\) are \(+a\) and \(-a\), with equal positive multiplicities.  Hence the principal curvatures of \(M_0\subset\Hn^n \times \{y_{0}\}\) are $1$ and $-1$,
with equal multiplicities.
This is impossible since Cartan's formula \cite{cecil2015geometry}*{Section 3.1, Equation (3.13)} for a hypersurface in hyperbolic space with constant principal curvatures gives that, for each principal curvature \(\lambda\),
\[
\sum_{\mu\ne\lambda}m_\mu\frac{-1+\lambda\mu}{\lambda-\mu}=0,
\]
where $m_{\mu}$ is the multiplicity of $\mu$.
Taking \(\lambda=1\) and \(\mu=-1\), we get
$m_{-1}\frac{-1+1\cdot(-1)}{1-(-1)}=-m_{-1}$,
which is not zero because \(2m_{-1}=\dim D_1=n-1\ne0\).
\end{proof}

For any local normal field $\psi\in\Gamma(E)$, any tangent vector $X\in D_1$, any tangent vector $Y\in TL$, we have
$$\langle A_\psi X,Y\rangle=\langle X,A_\psi Y\rangle=-\langle \bar\nabla_Y\psi,X\rangle=-\langle P_1(\bar\nabla_Y\psi),P_1(X)\rangle=0,$$
where the last equality follows from the fact that $P_1(\psi)=0$ and $\bar\nabla P_1=0$. In other words, we have
\begin{equation}\label{eq:ApsiD1zero}
A_\psi D_1=0.
\end{equation}

\begin{lemma}\label{lem:E-parallel}
The subbundle $E$ is parallel with respect to the normal connection $\nabla^\perp$ of $\nu L$. Consequently, $\mathbb{R}N$ is also parallel.
\end{lemma}

\begin{proof}
Fix $q^0=(q^0_1,q^0_2)\in U$ and $\zeta_{q^0}\in E_{q^0}$, and let $\zeta$ be the local parallel normal field on $U\subset L$ with this initial value. By equation \eqref{eq:HLO-determinant-condition}, $\det J_\zeta(q,t)$ is independent of $q$. At $q^0$, the vector $\zeta_{q^0}$ lies in $0\oplus T_{q^0_2}\mathbb{R}^m$, apply equation \eqref{eq:h=0cases} to \eqref{eq:Jacobi} gives
\begin{equation}\label{eq:vertical-J}
J_\zeta(q^0,t)=I-tA_\zeta(q^0).
\end{equation}
Thus $\det J_\zeta(q^0,t)$ is a polynomial in $t$, and therefore $\det J_\zeta(q,t)$ for every nearby $q$, is a polynomial in $t$.

Assume that $\zeta_q\notin E_q$ at some point $q$. Write $\zeta=sN+\varphi$, with $\varphi$ lies in $E$ and $s\neq0$. Its first-factor component is $s a \nu_1$, so Lemma \ref{lem:det-rigidity} applies at $q$ with $h=s a(q)\neq0$. The $\eps=1$ case is impossible. In the $\eps=-1$ case, on an open subset of $U$ where $s\neq0$, Lemma \ref{lem:det-rigidity} gives $A_\zeta D_1\subset D_1$ and eigenvalues $\pm s a$ on $D_1$, with equal multiplicities. Still denote this subset by $U$. Below we assume $\eps=-1$.

Applying equation \eqref{eq:ApsiD1zero} to $\varphi$ gives $A_{\zeta_q}|_{D_1}=sA_N|_{D_1}$, so $A_N|_{D_1}$ has eigenvalues $\pm a$ with equal multiplicities. If $\dim D_1$ is odd, this is already impossible; otherwise both signs occur with positive equal multiplicity. Therefore we can apply Lemma \ref{lem:cartanformula} to conclude that such $L$ does not exists. This contradiction proves that $\zeta$ remains in $E$. Thus $E$ is parallel with respect to $\nabla^\perp$. Since $\nu L$ is flat and $\nu L=\mathbb{R}N\oplus E$ orthogonally, $\mathbb{R}N$ is also parallel with respect to $\nabla^\perp$.
\end{proof}

From now on the local unit normal vector $N$ of the line $\mathbb{R}N$ can be assumed to be parallel. Also by Lemma \ref{lem:E-parallel}, for any parallel normal vector field $\zeta$ with $\zeta_q\in E_q$ at some point $q\in L$, $\zeta$ lies in $E$.

\subsection{Euclidean Shape Operators and No Tilted Curvature Normals}\label{subsec:EuclideanShapeOperator-No-tilted-curv}

Recall that in equation \eqref{eq:simultaneous-eigendistribution}, we have restrict $L$ to a smaller neighborhood such that simultaneous eigendistributions $\mathcal E_i$'s of shape operators of $L$ can be defined on $L$. In particular, the simultaneous eigendistributions $F_i$'s for the family $\{A_\zeta:\zeta\in \Gamma(E)\}$ can be defined on $L$.

\begin{lemma}\label{lem:E-shape-and-spectrum}
Let $\zeta\in\Gamma(E)$ be a parallel normal vector field on $L$. Then 
\begin{equation}\label{eq:Azeta-basic}
A_\zeta D_1=0,
\qquad
A_\zeta {\mathcal T}=0,
\qquad
A_\zeta D_2\subset D_2.
\end{equation}
Moreover the eigenvalues of $A_\zeta$ are constant. Consequently, for any $i$, the curvature normal corresponding to the simultaneous eigendistribution for the family $\{A_\zeta:\zeta\in \Gamma(E)\}$ is parallel and is in $E$.
\end{lemma}

\begin{proof}
$\zeta$ lies in $0\oplus T\R^m$. Since $0\oplus T\R^m$ is a parallel distribution with respect to $\bar\nabla$. Hence $\overline\nabla_X\zeta$ also lies in $0\oplus T\R^n$ for every $X\in TL$. Since $\nabla^{\perp}\zeta=0$, $\overline\nabla_X\zeta=-A_\zeta X$, thus $A_\zeta X$ lies in $TL\cap(0\oplus T\R^n)$, hence lies in $D_2$. Equation \eqref{eq:ApsiD1zero} and the symmetry of $A_\zeta$ gives equation \eqref{eq:Azeta-basic}.

By $\zeta\in 0\oplus T\R^m$ and equations \eqref{eq:Jacobi} and \eqref{eq:h=0cases}, $J_\zeta(t)=I-tA_\zeta$. By equation \eqref{eq:HLO-determinant-condition}, $\det(I-tA_\zeta)$ is independent of the base point, and the eigenvalues of $A_\zeta$ are constant. Let $\zeta_1,\ldots,\zeta_{\dim E}$ is a parallel orthonormal frame of $E$. By equation \eqref{eq:shape-commute}, their corresponding shape operators have simultaneous eigendistributions $\oplus_j F_j$. For $F_j$, we have $A_{\zeta_\alpha}|_{F_j}=\lambda_{j\alpha}I|_{F_j}$ for all $j$ and all $\alpha$.  Then each $\lambda_{j\alpha}$ is constant. The curvature normal $\mu_j=\sum_\alpha\lambda_{j\alpha}\zeta_\alpha$ corresponding to eigendistribution $F_j$, is therefore parallel and in $E$, for any $j$.
\end{proof}

Let $\eta_i=h_iN+\mu_i$ be the curvature normal corresponding to simultaneous eigendistribution $\mathcal E_i$ of all shape operators of $L$, with $\mu_i\in \Gamma(E)$ and $h_i$ be a function on $L$, for all $i$.

For $i$ with $\mu_i\neq0$, for any normal vector fields $\zeta$ in $\Gamma(E)$, $A_\zeta|_{\mathcal E_i}=\langle\zeta,h_iN+\mu_i\rangle I|_{\mathcal E_i}=\langle\zeta,\mu_i\rangle I|_{\mathcal E_i}$. Therefore, the distribution $\mathcal E_i$ is also an (possibly non-maximal) eigendistribution for the family $\{A_\zeta:\zeta\in \Gamma(E)\}$, and $\mu_i$ is a curvature normal corresponding to eigendistribution $\mathcal E_i$ for the family $\{A_\zeta:\zeta\in \Gamma(E)\}$. By Lemma \ref{lem:E-shape-and-spectrum}, $\mu_i$ is parallel in $E$, and $|\mu_i|$ is constant. Equation \eqref{eq:curvature-normal} gives $A_{\mu_i}|_{\mathcal E_i}=|\mu_i|^2I$, while equation \eqref{eq:Azeta-basic} gives $A_{\mu_i}(D_1\oplus\mathbb{R}{\mathcal T})=0$. Since $|\mu_i|\neq0$ and eigendistributions are mutually orthogonal,
\begin{equation}\label{eq:Ei-in-D2}
\mathcal E_i\subset D_2.
\end{equation}

\begin{lemma}\label{lem:no-oblique}
Either $h_i=0$ or $\mu_i=0$ for all $i$. 
\end{lemma}

\begin{proof}
Fix an $i$ and assume $\mu_i\neq0$. Let $X$ be a local unit vector field in $\mathcal E_i$. Since $X$ lies in $0\oplus T\R^m$, so does $\overline\nabla_XX$. By equation \eqref{eq:TL-decomposition}, we can write $\overline\nabla_XX=Z_2+g\nu_2+Z_E$, with $Z_2\in \Gamma(D_2)$, $Z_E\in \Gamma(E)$, and $g$ being a function on the neighborhood of $L$ where $X$ is defined. From equation \eqref{eq:projection-formula} we have $(\bar\nabla_XX)^\perp=Z_E+bgN$. On the other hand, by equation \eqref{eq:alphaXY=}, $\text{II}(X,X)=h_iN+\mu_i$. Comparing the $N$-components gives
\begin{equation}\label{eq:q-value}
g=\frac{h_i}{b},
\qquad
\langle\overline\nabla_XX,\nu_2\rangle=\frac{h_i}{b}.
\end{equation}
By equation \eqref{eq:Ei-in-D2}, $\langle {\mathcal T},X\rangle=0$. Therefore, by ${\mathcal T}=b\nu_1- a\nu_2$, and equation \eqref{eq:q-value}, we have
\begin{equation}\label{eq:nablaTX-value}
\langle\nabla_X{\mathcal T},X\rangle=\langle\bar\nabla_X{\mathcal T},X\rangle=-\langle {\mathcal T},\bar\nabla_XX\rangle=\frac{h_i a}{b}.
\end{equation}
Apply the Codazzi equation \eqref{eq:codazzi} to $A_{\mu_i}$, $X$ and ${\mathcal T}$ gives,
\begin{equation}\label{eq:codazzi-sec6}
    (\nabla_XA)_{\mu_i} \mathcal T-(\nabla_{\mathcal T} A)_{\mu_i} X=-(\bar R(X,\mathcal T)\mu_i)^T
\end{equation}
Since $\mu_i$  lies in $0\oplus T\R^m$ the right hand side of equation \eqref{eq:codazzi-sec6} vanishes. By equation \eqref{eq:Azeta-basic} and the fact that $\mu_i$ is parallel, the left hand side of equation \eqref{eq:codazzi-sec6} gives
\begin{equation*}\label{eq:Codazzi-mu}
\nabla_X(A_{\mu_i} {\mathcal T})-A_{\mu_i}(\nabla_X{\mathcal T})-\nabla_{\mathcal T}(A_{\mu_i} X)+A_{\mu_i}(\nabla_{\mathcal T} X)=-A_{\mu_i}(\nabla_X{\mathcal T})-\nabla_{\mathcal T}(A_{\mu_i} X)+A_{\mu_i}(\nabla_{\mathcal T} X).
\end{equation*}
Taking the inner product with $X$, and using $A_{\mu_i} X=|\mu_i|^2X$,  symmetry of $A_{\mu_i}$, the constancy of $|\mu_i|$, and equation \eqref{eq:nablaTX-value}, we have
\begin{align*}
   & -\langle A_{\mu_i}(\nabla_X{\mathcal T}),X\rangle-\langle\nabla_{\mathcal T}(|{\mu_i}|^2X),X\rangle+\langle A_{\mu_i}(\nabla_{\mathcal T}X),X\rangle\\
    =&-\langle \nabla_X{\mathcal T},A_{\mu_i} X\rangle-\langle\nabla_{\mathcal T}(|{\mu_i}|^2X),X\rangle+\langle \nabla_{\mathcal T}X,A_{\mu_i} X\rangle\\
    =&|{\mu_i}|^2(-\langle \nabla_X{\mathcal T},X\rangle-\langle\nabla_{\mathcal T}X,X\rangle+\langle \nabla_{\mathcal T}X, X\rangle)
    =-|{\mu_i}|^2\langle \nabla_X{\mathcal T},X\rangle=-|{\mu_i}|^2\frac{h_i a}{b}.
\end{align*}
This along with the right hand side of equation \eqref{eq:codazzi-sec6}
being $0$  gives, \begin{equation}\label{eq:oblique-contradiction}
-|{\mu_i}|^2\frac{h_i a}{b}=0,
\end{equation}
which force $h_i=0$ since $|{\mu_i}|\ne0$, and $a,b\ne 0$.
\end{proof}

\subsection{Projection to the Euclidean factor}

Let $f:=\pi_{\R^m}|_L:L\to{\mathbb{R}}^m.$ By $\ker df=D_1$, \(\rank df\) is constant. By the rank theorem \cite{lee2003smooth}*{Theorem 4.12}, there exists an open neighborhood on $L$, still denoted by $L$, such that 
\begin{equation*}
        \widehat L:=f(L),
\end{equation*}
is an immersed submanifold in $\R^m$, and $f:L\to\widehat L$ is a submersion with connected fibers. For any point $q\in L$, using the decompositions \eqref{eq:normal-decomposition} and \eqref{eq:tangent-decomposition} and the definition of $f$, we immediately get the following pointwise characterizations
\begin{equation}\label{eq:pointwise-df-on-components}
    \begin{aligned}
        &\ker((df)_q)=D_1|_q,\quad (df)_q(\mathcal T_q)=-a\,(d\pi_{2})_q(\nu_2|_q),\quad \ker((df)_q|_{D_2|_q})=0,\\ 
        &(d\pi_{2})_q(N_q)=b\,\,(d\pi_{2})_q(\nu_2|_q),\quad \ker((d\pi_{2})_q|_{E_q})=0,\\
        & T_{f(q)}\R^m=\R(d\pi_{2})_q(\nu_2|_q)\oplus (df)_q(D_2|_q)\oplus(d\pi_{2})_q(E_q)\\
        & T_{f(q)}\widehat L=\R(d\pi_{2})_q(\nu_2|_q)\oplus (df)_q(D_2|_q),\quad \nu_{f(q)}^{\R^m}\widehat L=(d\pi_{2})_q(E_q),
    \end{aligned}
\end{equation}
where $\nu^{\R^m}\widehat L$ denotes the normal bundle of $\widehat L$ in $\R^m$. The decompositions in \eqref{eq:pointwise-df-on-components} are all orthogonal, since $\nu_2|_q$ $D_2|_q$ and $E_q$ are mutually orthogonal and they do not have $\Q^n_\eps$ components. Denote 
\begin{equation}\label{eq:defn-widehat-E}
    \widehat E:=d\pi_{2}(E).
\end{equation}

\begin{lemma}\label{prop:E-descends}
For any parallel unit normal vector field $\zeta$ in $E$, $d\pi_{2}(\zeta)$ is a well-defined unit normal vector field in $\nu^{\R^m} \widehat{L}$. Consequently, $\widehat E=\nu^{\R^m}\widehat L$. 
\end{lemma}
\begin{proof}
Let $\widehat q$ be any point in $\widehat L$, and $q$ be in $f^{-1}(\widehat q)$. Let \(X\in\Gamma(D_1)\) be a local tangent vector field defined around $q$. By equation \eqref{eq:Azeta-basic} and $\nabla^\perp\zeta=0$, 
\begin{equation}\label{eq:XinD1-zetainE-nabla=0}
    \bar\nabla_X\zeta=-A_{\zeta}X+\nabla^\perp_X\zeta=0.
\end{equation}
Let \(\gamma:[0,1]\to L\) be any \(C^1\)-curve contained in $f^{-1}(\widehat q)$. Since $\ker (df)=D_1$, tangent vectors of $\gamma$ lie in $D_1$. Since $\zeta\in E \subset 0\oplus T\R^m$, at any point $p=(p_1,\widehat q)\in f^{-1}(\widehat q)$, $\zeta_p\in 0_{p_1}\oplus T_{\widehat q}\R^m$ can be identified naturally with $(d\pi_{2})_p\zeta_p\in T_{\widehat q}\R^m$.  By equation \eqref{eq:XinD1-zetainE-nabla=0}, $\zeta\circ \gamma$, identified with $d\pi_{2}(\zeta\circ\gamma)\in T_{\widehat q}\R^m$, is constant on $\gamma$. Therefore, for any $q,q'\in f^{-1}(\widehat q)$, by the connectivity assumption on the fibers of $f$, we can choose $\gamma$ such that it connects $q,q'$ and conclude
\begin{equation}\label{eq:pi-zeta-q=pi-zeta-q'}
        d\pi_{2}(\zeta(q'))=d\pi_{2}(\zeta(q)).
\end{equation}
Therefore $d\pi_{2}\zeta$ is a well-defined vector field. The normality of $d\pi_{2}\zeta$ follows from the pointwise description $\nu_{f(q)}^{\R^m}\widehat L=(d\pi_{2})_q(E_q)$ from equation \eqref{eq:pointwise-df-on-components}. Since $\zeta\in E$ and $1=|\zeta|$, we have $|d\pi_{2}\zeta|=|\zeta|=1$.

Let $\{\zeta_i\}$ be a parallel unit normal frame on $E$. For any $\widehat{q}\in\widehat L$ and any $q,q'\in f^{-1}(\widehat q)$, by equation \eqref{eq:pi-zeta-q=pi-zeta-q'}, $\text{span}\{d\pi_{2}(\zeta_i(q))\}=\text{span}\{d\pi_{2}(\zeta_i(q'))\}$.  On the other hand, $\text{span}\{d\pi_{2}(\zeta_i(q))\}=\widehat{E}_{\widehat q}$ by the definition of $\widehat{E}$ \eqref{eq:defn-widehat-E}. Therefore $\widehat{E}_{\widehat q}=d\pi_{2}(E_q)$ is independent of $q\in f^{-1}(\widehat{q})$, which by $\nu_{f(q)}^{\R^m}\widehat L=(d\pi_{2})_q(E_q)$ \eqref{eq:pointwise-df-on-components}, implies $\widehat{E}_{\widehat q}=\nu_{\widehat q}^{\R^m}\widehat L$ and $\widehat{E}=\nu^{\R^m}\widehat L$.
\end{proof}

Define the following sum of all simultaneous eigendistribution of shape operators of $L$ whose curvature normal is a non-zero normal vector field in $\Gamma(E)$, as well as its orthogonal complement in $TL$,
\begin{equation}\label{eq:input-DE-def}
        D_E:=\bigoplus_{\mu_i\neq 0}\mathcal E_i,\qquad D_N=D_E^{\perp_{TL}}.
\end{equation}
Recall that, under the assumption that $\mu_i\ne0$, we have equation \eqref{eq:Ei-in-D2}, which implies
\begin{equation}\label{eq:input-DE-D2}
        D_E\subset D_2.
\end{equation}
By the definition of $D_N$, we also have that, for any point $p\in L$
\begin{equation}\label{eq:DN-property}
    (D_N)_p=\bigcap_{\zeta\in\Gamma(E)}\ker(A_{\zeta_p}).
\end{equation}

\begin{lemma}\label{lem:Lhat-isoparametric-new}
The submanifold \(\widehat L\subset{\mathbb{R}}^m\) is isoparametric. Moreover, $T\widehat L=df(D_N)\oplus df(D_E)$, $df(D_E)$ is the sum of eigendistributions whose curvature normals are non-zero, and $df(D_N)$ is the sum of eigendistributions whose curvature normals are zero.
\end{lemma}

\begin{proof}
Let $\zeta$ be a parallel normal vector field in $E$, and $\widehat \zeta:=d\pi_{2}(\zeta)$, which is a well-defined vector field in $\Gamma(\widehat E)$ by Lemma \ref{prop:E-descends}. Let \(\widehat X\in T_{\widehat q}\widehat L\). Choose \(q\in f^{-1}(\widehat q)\), and choose an \(X
\in T_qL\) with \(df_q(X)=\widehat X\).  By \(\nabla^\perp\zeta=0\),
\begin{equation}\label{eq:zhat-derivative-shape}
    A_\zeta X=-\bar\nabla_{(P_1X+P_2X)}\zeta
\end{equation}
By $E\subset 0\oplus T\R^m$ and \cite{doCarmo1992}*{Chapter 6 Exercise 1(a)}, the last term in equation \eqref{eq:zhat-derivative-shape} is simply $(0,-\nabla^{\R^m}_{\widehat{X}}\widehat \zeta)$.  Since $A_\zeta X\in TL$, we have $-\nabla^{\R^m}_{\widehat X}\widehat \zeta=(df)_q(A_\zeta X)\in T\widehat L$, which gives
\begin{equation}\label{eq:zhat-normal-parallel-shape-identification}
        \widehat\nabla^{\perp}_{\widehat X}\widehat\zeta=0,\quad \widehat A_{\widehat \zeta}((df)_qX)=(df)_q(A_\zeta X),
\end{equation}
where $\widehat A_{\widehat \zeta}$ is the shape operator of $\widehat L$ along $\widehat \zeta$, and $\widehat\nabla^{\perp}$ is the Riemannian connection on $\nu^{\R^m}\widehat L$. Therefore a parallel normal frame on $E$ gives a parallel normal frame on $\nu^{\R^m}\widehat L$. Thus \(\nu^{{\mathbb{R}}^m}\widehat L\) is flat.

For each parallel \(\widehat\zeta\), equation \eqref{eq:zhat-normal-parallel-shape-identification} shows that the set of eigenvalues of \(\widehat A_{\widehat\zeta}\)  is a subset of eigenvalues of \(A_\zeta\).  The latter are constant by Lemma \ref{lem:E-shape-and-spectrum}. Then by Remark \ref{rem:Euclidean-isoparametric}, $\widehat L$ is an isoparametric submanifold.

It remains to identify the nonzero curvature distributions. Since $\widehat L$ is an isoparametric submanifold in Euclidean space, by \cite{palais2006critical}*{Proposition 2.1.2}, shape operators of $\widehat{L}$ have local simultaneous eigendistributions. Let \(\mathcal E_i\) be a simultaneous eigendistribution of \(L\), with curvature normal \(\eta_i=h_iN+\mu_i\), and $\mu_i$ lies in $E$. For $\zeta$ lies in $E$, equation \eqref{eq:curvature-normal} gives
\begin{equation}\label{eq:E-eigenvalue-mui}
        A_\zeta|_{\mathcal E_i}=\langle \zeta,\eta_i\rangle\id_{\mathcal E_i}=\inner{\zeta}{\mu_i}\id_{\mathcal E_i}.
\end{equation}
By Lemma \ref{lem:no-oblique}, either $\mu_i=0$ or $h_i=0$. If \(\mu_i=0\), that is, $\mathcal E_i\subset D_N$, then by equation \eqref{eq:E-eigenvalue-mui}, \(df(\mathcal E_i)\) is contained in the simultaneous eigendistribution of shape operators of \(\widehat L\) with eigenvalue $0$.  If \(\mu_i\neq0\), that is, $\mathcal E_i\subset D_E$, then \(\eta_i=\mu_i\in E\). Taking \(\zeta=\mu_i\) in equation \eqref{eq:E-eigenvalue-mui} and using equation \eqref{eq:zhat-normal-parallel-shape-identification} gives
\begin{equation}\label{eq:nonzero-hat-curv-dist}
        \widehat A_{\widehat\mu_i}(df(Y))=|\mu_i|^2df(Y)
        \qquad (Y\in\mathcal E_i).
\end{equation}
By $\ker (df)=D_1|_{TL}$ and equation \eqref{eq:input-DE-D2}, $df(Y)$ is non-zero, and \(df(\mathcal E_i)\) is a curvature distribution of \(\widehat L\) with non-zero curvature normal. Since $TL=D_E\oplus D_N$ and $T\widehat{L}=df(TL)$, we have that $df(D_N)$ and $df(D_E)$ are exactly the sum of curvature distributions of shape operators of $\widehat L$ with zero and non-zero curvature normals, that is, $T\widehat{L}=df(D_E)\oplus df(D_N)$.
\end{proof}

Now we may apply Corollary \ref{cor:real-palais-terng} to $\widehat L$. Up to an isometry by translation and rotation on $\R^{m}$, we get a local product decomposition
\begin{equation}\label{eq:PT-Lhat-product}
        \widehat L\subset\R^{m_1}\times L_2\times\{0\},
        \qquad
        0\oplus TL_2|_{\widehat L}=df(D_E),
        \qquad
        T\R^{m_1}|_{\widehat L}\oplus 0=df(D_N)
\end{equation}
where $\widehat L$ is an open subset in $\R^{m_1}\times L_2\times\{0\}$, \(L_2\subset{\mathbb{R}}^{m_2}\) is an open subset of a compact isoparametric submanifold in $\R^{m_2}$, and $m_1+m_2+m_3=m$.

\begin{lemma}
    Let $L'$ be an isoparametric submanifold satisfying the assumptions on $L$ in Theorem \ref{thm:reduction}, and additionally satisfying the assumption that $\widehat{L}'=\pi_{\R^m}(L')$ is full. If Theorem \ref{thm:reduction} is true for all such $L'$, then Theorem \ref{thm:reduction} is true.
\end{lemma}
\begin{proof}
    Let $L$ be an isoparametric submanifold satisfying the assumptions in Theorem \ref{thm:reduction}. $\widehat{L}$ is full in $\R^{m_1+m_2}\times\{0\}$ by equation \eqref{eq:PT-Lhat-product}. By Lemma \ref{lem:adding-euclidean-factor-no-change}, $L$ is an isoparametric submanifold in $\Q^n_\eps\times\R^{m_1+m_2}\times\{0\}$. If $L$ is locally decomposable as an open subset of $L_1\times L_2\times \{0\}$ in $\Q^n_\eps\times\R^{m_1+m_2}\times\{0\}$ as required by Theorem \ref{thm:reduction}, then $L$ is also decomposable in $\Q^n_\eps\times\R^{m}$, simply by multiplying $L_1\times L_2\times\{0\}$ with the origin in the extra factor $\R^{m_3}$.
\end{proof}
Therefore we may reduce the Theorem \ref{thm:reduction} to the case in which $\widehat L$ is full. From now on, we assume that 
\begin{equation*}
    m_3=0, \qquad m=m_1+m_2.
\end{equation*}

The following computations will be useful.

Since $df(D_E)\subset 0\oplus T\R^{m_2}$ and $\ker (d\pi_{2})=T\Q^n_\eps\oplus 0$, we have $D_E\subset d\pi_{2}^{-1}(0\oplus T\R^{m_2})\subset T\Q^n_\eps\oplus 0\oplus T\R^{m_2}$. Since $D_E\subset D_2$ and $D_2\perp (T\Q^n_\eps\oplus 0)$, we have $D_E\subset 0\oplus T\R^{m_2}$.

By equation \eqref{eq:Azeta-basic}, $D_1\oplus\R\mathcal T\subset D_N$. 
Also, $df(D_N)=T\R^{m_1}\oplus 0$ gives $D_N\subset T\Q^n_\eps\oplus T\R^{m_1}\oplus0$.  Since for any $p\in L$, $\mathcal T_p\in D_N\subset T\Q^n_\eps \oplus T\R^{m_1}\oplus0$, we have $d\pi_{2}(\mathcal T_p)=-a\,d\pi_{2}(\nu_2|_p)\in T\R^{m_1}\oplus0$, which implies $d\pi_{2}(N_p)\in \R^{m_1}\oplus0$ and $N\in T\Q^n_\eps \oplus T\R^{m_1}\oplus0$. Since $T_{f(p)}\R^{m_1}\oplus 0\subset T_{f(p)}\widehat L$, we have $\nu_p^{\R^m}\widehat L=\widehat E_p\subset 0\oplus T\R^{m_2}$. Then by $\widehat E=d\pi_{2}(E)$, we have $E\subset d\pi_{2}^{-1}(0\oplus T\R^{m_2})=T\Q^n_\eps\oplus 0\oplus T\R^{m_2}$. Since $\ker (d\pi_{2})=T\Q^n_\eps\oplus 0$ and $E\perp (T\Q^n_\eps\oplus0)$, we have $E\subset 0\oplus T\R^{m_2}$.

To summarize the above computations, we have
\begin{equation}\label{eq:DEDNT-m1m2}
\begin{aligned}
    &D_E\subset 0\oplus T\R^{m_2},\, D_1\oplus\R\mathcal{T}\subset D_N\subset T\Q^n_\eps\oplus  T\R^{m_1}\oplus0,\, \R N\in T\Q^n_\eps \oplus T\R^{m_1}\oplus0,\\
    &E\subset 0\oplus T\R^{m_2},\,
    D_N\oplus \R N=(T\Q^n_\eps\oplus T\R^{m_1})|_L,\, D_E\oplus E=(T\R^{m_2})_L,\, \\
    &D_N\oplus\R N\oplus D_E\oplus E=(T\Q^n_\eps\oplus T\R^m)_L.
\end{aligned}
\end{equation}

\subsection{Lifting the splitting}

Consider the following ambient product decomposition and projections
\begin{equation}\label{eq:ambient-after-PT-new}
        Q^n_\eps\times{\mathbb{R}}^{m}
        =(Q^n_\eps\times{\mathbb{R}}^{m_1})\times{\mathbb{R}}^{m_2},
\end{equation}
\begin{equation}\label{eq:Pi-defs}
        \Pi_1:Q^n_\eps\times{\mathbb{R}}^{m}\to Q^n_\eps\times{\mathbb{R}}^{m_1},
        \qquad
        \Pi_2:Q^n_\eps\times{\mathbb{R}}^{m}\to{\mathbb{R}}^{m_2}
\end{equation}
By equation \eqref{eq:DEDNT-m1m2}, we have $D_E\subset\ker(d\Pi_1|_{TL}),\,D_N\subset\ker(d\Pi_2|_{TL})$. By $D_E\oplus D_N=TL$, we get
\begin{equation}\label{eq:kernel-identities-new}
        D_E=\ker(d\Pi_1|_{TL}),
        \qquad
        D_N=\ker(d\Pi_2|_{TL}).
\end{equation}

\begin{lemma}\label{lem:L1L2decomp}
    $L$ is an open subset of $L_1\times L_2$, where $L_1$ and $L_2$ are submanifolds in $\Q^n_\eps\times\R^{m_1}$ and $\R^{m_2}$, respectively. Moreover $TL_1\oplus0=\ker(d\Pi_2|_{TL})$, and $0\oplus TL_2=\ker(d\Pi_1|_{TL}).$
\end{lemma}
\begin{proof}
By equation \eqref{eq:kernel-identities-new}, $TL=\ker(d\Pi_1|_{TL})\oplus\ker(d\Pi_2|_{TL})$. There exists a neighborhood on $L$, still denoted by $L$, such that 
\begin{equation}
    \Pi_1|_L:L\to \Pi_1|_L(L)=:L_1,\quad \Pi_2|_L\to \Pi_2|_L(L)=L_2
\end{equation}have constant ranks. Then the map $(\Pi_1,\Pi_2)|_L:L\rightarrow L_1\times L_2$ is a restriction of the inclusion of $L$ into $\Q^n_\eps\times\R^{m}$. By our construction, the tangent map $d(\Pi_1,\Pi_2)|_L$ is an linear isomorphism at any point in $L$, therefore up to restricting $L$ to a smaller neighborhood around $p^0$, $(\Pi_1,\Pi_2)|_L:L\rightarrow L_1\times L_2$ is a diffeomorphism. Moreover, since it is the restriction of the inclusion of $L$ into $\Q^n_\eps\times\R^{m}$, it is also an isometry.
\end{proof}

\begin{lemma}
    $L_1$ is a hypersurface in $\Q^n_\eps\times\R^{m_1}$. $d\Pi_1|_L(N)$ is well-defined and is a unit normal vector field on $L_1$.
\end{lemma}

\begin{proof}
First compute the dimension of \(L_1\).  By equation \eqref{eq:PT-Lhat-product}, $df(D_N)=(T\R^{m_1}\oplus0)|_{\widehat L}$. Since $D_1\subset D_N$, $\ker(df|_{D_N})=\ker(df)=D_1$. Thus by equation \eqref{eq:DEDNT-m1m2} $\dim L_1=\dim D_N=\dim (\Q^n_\eps\times\R^{m_1})-\dim\R N=n+m_1-1.$
The ambient space \(\Q^n_\eps\times{\mathbb{R}}^{m_1}\) has dimension \(n+m_1\).  Hence \(L_1\) is a hypersurface.

We now prove that \(d\Pi_1|_L(N)\) is well-defined along \(L_1\). 

Let $Y\in D_E$. By parallel-ness of $N$, $\bar\nabla_{Y}N=-A_NY+\nabla^\perp_YN=-A_NY$. Suppose $Y$ is in an eigendistribution $\mathcal E_i\subset D_E$, by the definition \eqref{eq:input-DE-def}, the curvature normal $\eta_i=h_iN+\mu_i$ has $h_i=0$. Therefore by the characterization of curvature normals equation \eqref{eq:curvature-normal}, $A_NY=\langle N,\eta_i\rangle Y=0$ since $\R N\perp E$. Therefore
\begin{equation}\label{eq:ANY=0}
    \bar\nabla_YN=A_NY=0.
\end{equation}
For any point $\widehat q\in L_1$. Let \(\gamma:[0,1]\to L\) be any \(C^1\)-curve contained in $(\Pi_1|_L)^{-1}(\widehat q)$. Since $\ker (d\Pi_1|_L)=D_E$, tangent vectors of $\gamma$ lie in $D_E$. Since $N$ lies in $T\Q^n_\eps\oplus T\R^{m_1}\oplus 0$, at any point $p=(\widehat q,p_2)\in \Pi_1^{-1}(\widehat q)$, $N_p\in T_{\widehat q}(\Q^n_\eps\times \R^{m_1})\oplus 0_{q_2}$ can be identified naturally with $(d\Pi_1)_p\zeta_p\in T_{\widehat q}(\Q^n_\eps\times \R^{m_1})$.  By equation \eqref{eq:ANY=0}, $N\circ \gamma$, identified with $d\Pi_1(N\circ\gamma)\in T_{\widehat q}(\Q^n_\eps\times \R^{m_1})$, is constant on $\gamma$. Therefore, for any $q,q'\in (\Pi_1|_L)^{-1}(\widehat q)$, by the connectivity assumption on the fibers of $\Pi_1|_L$, we can choose $\gamma$ such that it connects $q,q'$ and conclude
\begin{equation}\label{eq:Pi-zeta-q=Pi-zeta-q'}
        d\Pi_1(\zeta(q'))=d\Pi_1(\zeta(q)).
\end{equation}
We conclude that $d\Pi_1(N)$ is a well-defined vector field on $L_1$.

For any point $p\in L$, by equation \eqref{eq:DEDNT-m1m2}, for any tangent vector \(X_p\in D_N\),
\begin{equation*}\label{eq:N1-normal-L1-new}
        \inner{d\Pi_1N_p}{d\Pi_1X_p}_{Q^n_\eps\times{\mathbb{R}}^{m_1}}
        =\inner{N_p}{X_p}_{Q^n_\eps\times{\mathbb{R}}^m}-\inner{d\Pi_2N_p}{d\Pi_2X_p}_{{\mathbb{R}}^{m_2}}=0.
\end{equation*}
Since $d\Pi_1(N_p)=d\Pi_1(N_q)$ for $p,q\in L$ with $\Pi_1(p)=\Pi_2(q)$, $d\Pi_1(N_p)^\perp=d\Pi_1(N_q)^\perp$. This together with the dimension of $D_N$ being $n-1+m_1$, proves $\nu^{\Q^n_\eps\times \R^{m_1}}L_1=\R d\Pi_1(N)$. By equation \eqref{eq:DEDNT-m1m2}, $N$ has no $\mathbb{R}^{m_2}$ component. Therefore $d\Pi_1(N)$ has unit length, since $N$ has unit length.
\end{proof}

By the definition of $D_N$ \eqref{eq:input-DE-def} and Lemma \ref{lem:no-oblique}, for any vector $X\in\mathcal E_i\subset D_N$, $A_N(X)=\langle N,\eta_i\rangle X\in D_N$. Therefore we have 
\begin{equation}\label{eq:ANDNsubsetDN}
    A_ND_N\subset D_N.
\end{equation}

For any vector $X\in D_N$, by the product structure \eqref{eq:ambient-after-PT-new}, we have $\nabla^{Q^n_\eps\times{\mathbb{R}}^{m_1}}_{d\Pi_1X}d\Pi_1(N)
        =d\Pi_1(\bar\nabla_XN)
        =-d\Pi_1(A_NX),$
        where the first equality is an elementary exercise \cite{doCarmo1992}*{Chapter 6 Exercise 1(a)} applying to $N$ since $N$ lies in $T(\Q^n_\eps\times\R^{m_1})\oplus0$, and the last equality is from the parallel-ness of $N$.
Therefore $A^{L_1}_{d\Pi_1(N)}(d\Pi_1X)=d\Pi_1(A_NX)$, since $d\Pi_1(N)$ is unit hence also parallel in $\nu^{\Q^n_\eps\times \R^{m_1}}L_1$. Here $d\Pi_1(X)$ is tangent in $L_1$ since $X\in D_N$. By equation \eqref{eq:ANY=0} ($A_ND_E=0$) and equation \eqref{eq:ANDNsubsetDN}, under the decomposition \(TL=D_N\oplus D_E= (TL_1\oplus TL_2)|_L\), we have
\begin{equation}\label{eq:shape-direct-sum-new}
        A_N^L=A_{d\Pi_1(N)}^{L_1}\oplus 0.
\end{equation}

It remains to prove that \(L_1\) is isoparametric.  $\nu^{\Q^n_\eps\times \R^{m_1}} L_1$ is flat since it is codimension $1$. Also since $L_1$ has codimension $1$, all sections of $L_1$ are normal geodesics hence totally geodesic.

Parallel normal vector fields on $L_1$ are of the form $\xi=c\cdot d\Pi_1(N)$ for some constant number $c$. Let $(x,y)$ be a point in $(Q^n_\eps\times{\mathbb{R}}^{m_1})\times{\mathbb{R}}^{m_2}$ with $x\in Q^n_\eps\times{\mathbb{R}}^{m_1}$ and $y\in{\mathbb{R}}^{m_2}$. By equation \eqref{eq:Jacobi} and equation \eqref{eq:shape-direct-sum-new}, we have that the operator $J(t)$ on $L$ along $cN$ satisfies
\begin{equation}\label{eq:Jacobi-direct-sum-final}
        J^{cN}_L((x,y),t)
        =J^{c\cdot d\Pi_1(N)}_{L_1}(x,t)\oplus I_{T_yL_2},
        \qquad
        \det J^{cN}_L((x,y),t)
        =\det J^{c\cdot d\Pi_1(N)}_{L_1}(x,t).
\end{equation}
By equation \eqref{eq:HLO-determinant-condition}, $\det J^{cN}_L((x,y),t)$ is independent of $(x,y)$, which again by equation \eqref{eq:HLO-determinant-condition}, implies that $L_1$ is almost isoparametric. Therefore $L_1$ is isoparametric.

At last, we prove that the angle function of $L_1$ satisfies that $-1<C<1$. Denote the natural projection from $\Q^n_\eps\times\R^{m_1}$ to the $\Q^n_\eps$ by $\Pi_{\Q^n_\eps}$, and the natural projection from $\Q^n_\eps\times\R^{m_1}$ to $\R^{m_1}$ by $\Pi_{\R^{m_1}}$. Since $d\Pi_1(N)$ is the normal vector field for $L_1$, we can compute the two components of it by $d\pi_{1}\circ d\Pi_1(N)=d\pi_{1}(N)= a\ne0$, and $d\Pi_{\R^{m_1}}\circ d\Pi_1(N)=d\pi_{\R^{m_1}}\circ d\pi_{\R^m_\eps}(N)=b\ne0$. On the other hand $a^2+b^2=|N|^2=1$. Therefore the angle function $C$ of $L_1$, as defined in equation \eqref{eq:angle-function}, which can be computed by $ a^2-b^2$, gives $-1<C<1$.

\section{Constant Angle Function}\label{sec:constant-angle}

The angle function $C$ defined by equation \eqref{eq:angle-function} has been studied in hypersurfaces of product spaces by Tojeiro \cite{Tojeiro2010}, Garnica--Palmas--Ruiz-Hern\'andez \cite{Garnica2012}, Urbano \cite{urbano2019hypersurfaces}, Gao--Ma--Yao \cites{gao2024hypersurfaces,gao2024isoparametric}, de Lima--Pipoli \cites{de2022isoparametric,delima2025}, and Tan--Xie--Yan \cite{TanXieYan2026}.  
The constancy of this function is useful for the classification of  isoparametric hypersurfaces in \(\Q^{n}_{\epsilon_1}\times \Q^{m}_{\epsilon_2}\).  Urbano \cite{urbano2019hypersurfaces} and Gao--Ma--Yao \cites{gao2024hypersurfaces,gao2024isoparametric} proved this property in low-dimensional products.  De Lima--Pipoli \cite{de2022isoparametric} proved it for \(\Hyp^n\times\mathbb R\) and \(\Sph^n\times\mathbb R\).  Tan--Xie--Yan \cite{TanXieYan2026} extended the result to \(\Hyp^n\times\mathbb R^m\) and \(\Sph^n\times\mathbb R^m\).  De Lima--Pipoli \cite{delima2025} proved it for products \(\Q^{n}_{\epsilon_1}\times \Q^{m}_{\epsilon_2}\).
De Lima--Pipoli and Tan--Xie--Yan  calculated the Jacobi matrices and used   Kac matrices  to prove that  $C$ is constant.  Here we follow their Jacobi matrices approach but give a  simpler proof.

\begin{theorem}\label{thm:constant-angle}
Let $L$ be an isoparametric hypersurface in $\Q^n_{\eps_1}\times \Q^m_{\eps_2}$. Then the angle function $C$ is constant on \(L\).
\end{theorem}



If there is no $p\in L$ with $-1<C(p)<1$, then \(C\) only takes the values \(1\) and \(-1\). Since \(L\) is connected and \(C\) is continuous, \(C\) is constant.  Hence we may assume there exists $p\in L$ with $-1<C(p)<1$. Note that by continuity of $C$, for a small neighborhood $U\ni p$ and any $q\in U$, $-1<C(q)<1$. By connectivity of $L$, we only need to prove that $C$ is locally constant on $U$. 

There is a neighborhood of $p$ in $U$, still denoted by $U$, such that a  unit normal field $N$  is defined on $U$. Here $N=N_1+N_2$   by  equation (\ref{eq:N=N1+N2}).
Let
\begin{equation}\label{eq:defn-of-a-b}
 a=||N_1||,\qquad
b=||N_2||. 
\end{equation}
Then
$a>0, 
b>0, 
a^2+b^2=1$.
Recall the tangential part of \(P(N)\) is
$${\mathcal T}=P(N)-CN=(1-C)N_{1}-(1+C)N_{2}\in \Gamma(TL) $$ which is defined by equation \eqref{eq:T-defn}.
Denote 
\begin{equation}\label{eq:E1}
    E_1:=\frac{{\mathcal T}}{|{\mathcal T}|},
\end{equation}
 called the mixed direction. 

Choose orthonormal vectors $X_1,\ldots,X_{n-1}\in (T_{\pi_{1}(p)}\Q_{\eps_1}^n\oplus 0)$ such that $X_i\perp N_1$ and $Y_1,\ldots,Y_{m-1}\in (0\oplus T_{\pi_{2}(p)}\Q_{\eps_2}^m)$ such that  $ Y_j\perp N_2$ for $i=1,2,\cdots,n-1$, $j=1,2,\cdots m-1$.
Then $\{E_1,
 X_1,\ldots,X_{n-1},
 Y_1,\ldots,Y_{m-1}\}$,
is an orthonormal basis of \(T_pL\).
We call \(X_i\)'s pure first-factor vectors, and \(Y_\alpha\)'s pure second-factor vectors.
For convenience, we rewrite the basis as 
\begin{equation}
   E_1, E_{i+1}=X_i\ (1\le i\le n-1), E_{n+i}=Y_i\ (1\le i\le m-1).
\end{equation}
Denote the shape operator $A_N$ of $L$ along $N$ by $A$. The matrix of $A$ at \(p\) with respect to the basis $\{E_i\}_{i=1,...,n+m-1}$ by 
\begin{equation*}
    A:=(h_{i j})_{i,j=1,...,n+m-1}.
\end{equation*}

We define  $\gamma_p(r)= \exp_p(rN(p))$ for $r\ge0$ small. Parallel transport the basis $\{E_{i}\}|_{i=1}^{n+m-1}$ along $\gamma_p$ to get the frame  $\{E_1(r),E_2(r),\dots,E_{n+m-1}(r)\}$. Since $P_1$ and $P_2$ are parallel, $P_{1}(E_i),P_{2}(E_{i})$ are parallel along $\gamma_p$ for $i=1,2,\cdots,n+m-1$. Thus the first-factor vectors remain first-factor and the second-factor vectors remain second-factor. 

For each $i\in\{1,2,...,n+m-1\}$, let \(J_i(r)\) be the Jacobi field along \(\gamma_p\) determined by
$J_j(0)=E_j,
J_j'(0)=-AE_j$.
Write 
\begin{equation*}
 J_j(r)=\sum_{i=1}^{n+m-1}B_{i j}(r)E_i(r), \qquad  B_p(r)=(B_{i j}(r))_{1\le i,j\le n+m-1}.
\end{equation*}
 Using curvature formula (\ref{eq:curvature-formula}) of simply connected space forms, a standard Jacobi field computation gives the following Lemma. An explicit computation can be looked up in \cite{delima2025}*{Equations (27),(30)}.
\begin{lemma}\label{lem:jacobi-matrix-row}
  For matrix $B_{p}(r)$, we have
  
(i) $B_{1j}(r)=\delta_{1j}-rh_{1j}$,

(ii) for \(2\le i\le n\), $B_{ij}(r)=\delta_{ij}c_{+}(r)-h_{ij}s_{+}(r)$,

(iii)for \(n+1\le i\le n+m-1\),
$B_{ij}(r)=\delta_{ij}c_{-}(r)-h_{ij}s_{-}(r)$.

Here we define
 \begin{align}\label{eq:c+s+}
  c_+(r)=\begin{cases}
    \cos(ar), &\text{if }\epsilon_{1}=1\\
    1,    &\text{if }\epsilon_{1}=0\\
    \cosh(ar), &\text{if }\epsilon_{1}=-1
\end{cases},\quad s_+(r)=\begin{cases}
    \dfrac{\sin(ar)}{a}, &\text{if }\epsilon_{1}=1\\
    r, &\text{if }\epsilon_{1}=0\\
    \dfrac{\sinh(ar)}a, &\text{if }\epsilon_{1}=-1,
\end{cases} 
\end{align}
\begin{align}\label{eq:c-s-}
  c_-(r)=\begin{cases}
    \cos(br), &\text{if }\epsilon_{2}=1\\
    1, &\text{if }\epsilon_{2}=0\\
    \cosh(br), &\text{if }\epsilon_{2}=-1
\end{cases},\quad s_-(r)=\begin{cases}
    \dfrac{\sin(br)}{b}, &\text{if }\epsilon_{2}=1\\
    r, &\text{if }\epsilon_{2}=0\\
    \dfrac{\sinh(br)}b, &\text{if }\epsilon_{2}=-1
\end{cases} 
\end{align}
\end{lemma}
We call the first row (which consist of entries in (i) ) the mixed row of $B_p(r)$, the rows consist of entries in (ii)  the first-factor rows, and the rows consist of entries in (iii) the second-factor rows. Denote 
\begin{equation}\label{eq:Dp=detBp}
   D_p(r):=\det B_p(r) 
\end{equation}
for $r\ge0$ small.  Note that $D_{p}(r)$  is the Jacobi determinant along $N$ at $p$ as in equation 
\eqref{eq:HLO-determinant-condition}. For any $q \in U$, the domain of the Jacobi determinant $D_{q}(r)$ can be extended to complex  plane $\C$.  By equation \eqref{eq:Dp=detBp} and Lemma \ref{lem:jacobi-matrix-row},  $D_{q}(r)$ is   a holomorphic  function.  Moreover,  we have the following. 

\begin{lemma}\label{lem:Dp=Dq}
     $D_p(z)=D_q(z)$ for any $q\in U$ and any $z\in\C$.
\end{lemma}
\begin{proof}
Let $L_r$ be the local parallel submanifold $L_r:=\{\exp_{p} rN(p)| p\in U \}$ of $L$. Let $A_{r,p}$ be the shape operator of $L_r$ at $\gamma_p(r)$ along $\gamma_{p}'(r)$, and $H_r(p)=\tr A_{r,p}$ be the mean
curvature of $L_r$ at $\gamma_p(r)$. Standard Jacobi field theory \cite{BCO}*{Theorem 10.2.1} gives
$A_{r,p}=-B_p'(r)B_p(r)^{-1}$. Taking trace of $A_{r,p}$ and by \cite{HLO}*{Proof of Proposition 2.1}, we have
$H_r(p)= -\frac{d}{dr}\log D_p(r)$.
Since $L$ is isoparametric, $H_r$ is constant on $L_r$ for
small $r$. For any $q\in U$, we can define $B_q$ from some choice of tangential basis at $q$. Plug in $r=0$, we have $B_q(0)=Id$ does not depend on the choice of tangential basis. Thus we have $D_q(0)=1$ for every $q\in U$.
The  differential equations 
$\frac{d}{dr}D_q(r)=-H_rD_q(r)$ for all $q\in U$
have the same initial value and the same coefficients. Therefore $D_q(r)$ is independent of $q$ in $U$, that is,  $D_p(r)=D_q(r)$ for any sufficiently small $r$ and $q\in U$. Since $D_{p}(r)$ is holomorphic for any $p\in U$, then we have $D_{p}(r)=D_{q}(r)$ for $r\in \C$ by classical complex analysis results (see \cite{rudin1974real}*{Theorem 10.18}).

\end{proof}

Now we prove Theorem \ref{thm:constant-angle} case by case.

\subsection{Case I: \texorpdfstring{\(\Sn^n\times\R^m\)}{S x R} and \texorpdfstring{\(\Sn^n\times\Hn^m\)}{S x H}}
In this case \(\eps_1=1\), \(\eps_2\in\{0,-1\}\).  We prove that \(a\) (which is $||N_1||$) is locally constant, which leads to that $C=2||N_1||^2-1$ is locally constant. Plug  \(z=\ii t\) into equation \eqref{eq:Dp=detBp}, where $t\in \R$, and $\ii$ is the imaginary number with $\ii^2=-1$. For $k=1,2,...,n+m-1$, denote the $k$-th row of $B_p(\ii t)$ by $b_k(\ii t)$, the $k$-th row of the identity matrix $Id_{n+m-1}$ by $e_k$, the $k$-th row of the shape operator matrix $A=\{h_{kj}\}$ by $h_{k*}$. By Lemma \ref{lem:jacobi-matrix-row},
we have
\begin{align}
 &b_1(\ii t)=e_1-\ii t h_{1*},\label{eq:B1} \\
 &b_k(\ii t):
=\cosh(at)e_k-\ii\frac{\sinh(at)}a h_{k*}
=\frac{e^{at}}2\left(e_k-\frac{\ii}{a}h_{k*}\right)
 +\frac{e^{-at}}2\left(e_k+\frac{\ii}{a}h_{k*}\right)\label{eq:B2-n} \\
 &b_j(\ii t)=\begin{cases}
      e_j\cos(bt)-\ii h_{j*}\dfrac{\sin(bt)}{b}, &\text{if} \quad\epsilon_{2}=-1,\label{eq:Bn+m-1} \\
      e_k-\ii th_{j*},  &\text{if} \quad \epsilon_{2}=0.
  \end{cases} 
\end{align}
For $k=2,3,...,n$, and  $j=n+1,n+2,...,n+m-1$.
For simplicity,  we write
\begin{equation}\label{eq:defn-of-Bk1Bk2}
    b_k(\ii t)=\frac{e^{at}}{2}b_k^1+\frac{e^{-at}}{2}b^2_k, \quad\text{where } b_k^1:=e_k-\frac{\ii}{a}h_{k*},\quad\text{and } b_k^2:=e_k+\frac{\ii}{a}h_{k*}.
\end{equation}
By equation \eqref{eq:defn-of-Bk1Bk2} and linear properties of determinants, we have
\begin{align} \nonumber
D_p(\ii t)&=\det\begin{pmatrix}
    b_1(it) \\
    b_2(it) \\
    ...      \\
    b_{n}(it)\\
    ...
\end{pmatrix}=\det\begin{pmatrix}
    b_1(it) \\
    \frac{e^{at}}{2}b_2^1+\frac{e^{-at}}{2}b^2_2\\
    ...\\
    \frac{e^{at}}{2}b_n^1+\frac{e^{-at}}{2}b^2_n\\
    ...
\end{pmatrix}\\ \nonumber
&=\frac{e^{at}}{2}\det\begin{pmatrix}
    b_1(it) \\
    b_2^1 \\
    ...\\
    \frac{e^{at}}{2}b_n^1+\frac{e^{-at}}{2}b^2_n\\
    ...
\end{pmatrix}+\frac{e^{-at}}{2}\det\begin{pmatrix}
    b_1(it) \\
    b_2^2 \\
    ...\\
    \frac{e^{at}}{2}b_n^1+\frac{e^{-at}}{2}b^2_n\\
    ...
\end{pmatrix}\\ \nonumber
&=\cdots\\ 
&=\frac{e^{(n-1)at}}{2^{n-1}}\det\begin{pmatrix}
    b_1(it) \\
    b_2^1 \\
    ...\\
    b_n^1\\
    ...
\end{pmatrix}+\cdots, \label{eq:Dp-it}
\end{align}

i.e.,
$$D_p(\ii t)=2^{-n+1}\sum_{j=0}^{n-1}e^{(n-1-2j)at}P_j(t,\cos(bt),\sin(bt)),$$
where $P_0=\det\begin{pmatrix}
    b_1(\ii t)\\
    b_2^1 \\
    ...\\
    b_n^1 \\
    ...
\end{pmatrix}$ and $P_{j}$ are polynomials by equations \eqref{eq:B1},\eqref{eq:B2-n} and \eqref{eq:Bn+m-1}, for $j=0,1,\cdots,n-1$.
Plug in $t=0$, $P_0(0,\cos(b\cdot0),\sin(b\cdot0))=\det\begin{pmatrix}
    1 &  &  \\
     & Id_{n-1}-\frac{\ii}{a}A' &  \\
    && Id_{m-1}
\end{pmatrix}$ where $A'$ is $A$ restricted to the block with rows $2\le k\le n$ and columes $2\le j\le n$. Since $A$ is a real symmetric matrix, then the real part of eigenvalues of  $Id_{n-1}-\frac{\ii}{a}A'$ does not vanish. Thus \begin{equation}\label{eq:P00ne0}
    P_0(0,\cos(b\cdot0),\sin(b\cdot0))\ne0.
\end{equation} Since $P_0$ is a polynomial of $\cos t$, $\sin t$, and $t$, for $t_j=\frac{2j\pi}{b}$, $P_0(t_j,\cos(bt_{j}),\sin(bt_{j}))=P_{0}(t_{j},1,0)$ does not vanish for large $j$ since $P_{0}(t,1,0)$ is a non-vanishing polynomial by (\ref{eq:P00ne0}). Therefore we have
\begin{align*}
    \lim_{j\rightarrow\infty}\frac1{t_j}\log|D_p(\ii t_j)|=&\lim_{j\rightarrow\infty}\frac1{t_j}\log(2^{-n+1}|e^{(n-1)at_j}P_0+e^{(n-3)at_j}P_1+...|)\\
    =&\lim_{j\rightarrow\infty}\frac1{t_j}\log(2^{-n+1}|e^{(n-1)at_j}P_0|)=\lim_{j\rightarrow\infty}\frac1{t_j}\log e^{(n-1)at_j}=(n-1)a.
\end{align*}
Since $D_p(z)=D_q(z)$ for any $q\in U$, we have $a(p)=a(q)$, which finishes the proof.

\subsection{Case II: \texorpdfstring{\(\Sn^n\times\Sn^m\)}{S x S}}

In this case $\epsilon_{1}=\epsilon_{2}=1$.
The proof is essentially the same as the previous case except for minor modifications. We again plug $r=\ii t$ into $B_p$. 
By Lemma \ref{lem:jacobi-matrix-row},
we have
\begin{align}
 &b_1(\ii t)=e_1-\ii t h_{1*},\label{eq:S-B1} \\
 &b_k(\ii t)
=\cosh(at)e_k-\ii\frac{\sinh(at)}a h_{k*}
=\frac{e^{at}}2\left(e_k-\frac{\ii}{a}h_{k*}\right)
 +\frac{e^{-at}}2\left(e_k+\frac{\ii}{a}h_{k*}\right)\label{eq:S-B2-n} \\
 &b_j(\ii t)=\cosh(bt)e_j-\ii\frac{\sinh(bt)}b h_{j*}
=\frac{e^{bt}}2\left(e_j-\frac{\ii}{b}h_{j*}\right)
 +\frac{e^{-bt}}2\left(e_j+\frac{\ii}{b}h_{j*}\right)\label{eq:S-B-n+m}
\end{align}
for $k=2,3,...,n$, and  $j=n+1,n+2,...,n+m-1$.
For simplicity,  we write
\begin{equation}\label{eq:bk-it-S-S}
    b_k(\ii t)=\frac{e^{at}}{2}b_k^1+\frac{e^{-at}}{2}b^2_k,  \quad b_j(\ii t)=\frac{e^{bt}}{2}b_j^1+\frac{e^{-bt}}{2}b^2_j
\end{equation}
where $b_k^1:=e_k-\frac{\ii}{a}h_{k*}, b_k^2:=e_k+\frac{\ii}{a}h_{k*}$ and $b_j^1:=e_j-\frac{\ii}{b}h_{j*}, b_j^2:=e_j+\frac{\ii}{b}h_{j*}$.
By equation \eqref{eq:bk-it-S-S},  linear properties of determinants,  and  calculations similar to equation \eqref{eq:Dp-it},  we have


$$D_p(\ii t)=2^{-n-m+2}\sum_{l=0}^{m-1}\sum_{j=0}^{n-1}e^{(n-1-2j)at+(m-1-2l)bt}P_{jl}(t),$$
for some polynomials $P_{jl}$. 

Here $P_{00}=\det\begin{pmatrix}
    b_1(\ii t)\\
    b_2^1 \\
    b_3^1 \\
    ...
\end{pmatrix}$. Plug in $t=0$ and by equations \eqref{eq:S-B1}, \eqref{eq:S-B2-n} and \eqref{eq:S-B-n+m},  $P_{00}(0)=\det\begin{pmatrix}
    1 &  &  \\
     & Id_{n-1}-\frac{\ii}{a}A' &  \\
    && Id_{m-1}-\frac{\ii}{b}A''
\end{pmatrix}$ where $A'$ is $A$ restricted to the block with rows $2\le k\le n$ and columes $2\le j\le n$, and $A''$ is $A$ restricted to the block with rows $n+1\le k\le n+m-1$ and columes $n+1\le j\le n+m-1$. Since $A$ is a real symmetric matrix, Then the real part of eigenvalues of   $Id_{n-1}-\frac{\ii}{a}A'$, $Id_{m-1}-\frac{\ii}{b}A''$ don't vanish. Hence $P_{00}(0)$ is not $0$, and $P_{00}$ is non-zero for large $t$. Since the leading term is non-zero and has the highest exponential in $t$, we have 
\begin{align*}
    \lim_{t\rightarrow\infty}\frac1{t}\log|D_p(\ii t)|=&\lim_{t\rightarrow\infty}\frac1{t}\log(2^{-n-m+2}|e^{(n-1)at+(m-1)bt}P_{00}+...|)=(n-1)a+(m-1)b
\end{align*}
Since the left hand side is independent of $p$ by Lemma \ref{lem:Dp=Dq}, then $(n-1)a+(m-1)b$ is also independent of $p$. This together with $a^2+b^2=1$ establish that $a$ is a constant.

\subsection{Case III: \texorpdfstring{\(\Hyp^n\times\R^m\)}{H x R}}
Different from the above two cases, here we study the real function $D_{p}(r)$, i.e., $r\in \R$.

Assume that there exists nonzero exponent terms in $D_{p}(r)$.
 For each first-factor row, Lemma \ref{lem:jacobi-matrix-row} (ii) gives
 \begin{equation}\label{eq:b-i-H-R}
  b_i(r)=\frac{e^{ar}}{2}\left(e_i-\frac{1}{a}h_{i*}\right)
       +\frac{e^{-ar}}{2}\left(e_i+\frac{1}{a}h_{i*}\right)   
 \end{equation}
for $2\leq i\leq n$,
where $e_i$  and $h_{i*}$ are defined the same as in  Case I. By Lemma \ref{lem:jacobi-matrix-row} (i) and (iii),  the mixed row and the
second-factor rows are polynomial in $r$ with degree $1$.  By equation \eqref{eq:b-i-H-R},   linear properties of determinants,  and  calculations similar to equation \eqref{eq:Dp-it} , we have 
$$D_p(r)=2^{-(n-1)}\sum_{j=0}^{n-1}
        e^{(n-1-2j)ar}P_{j}(r),$$
where each $P_{j}$ is a polynomial.  
If  $D_{p}(r)$ has a  term with an exponent that is not zero, i.e., $(n-1-2j)a(p)\neq 0$ for some $j\in \{0,1,\cdots,n-1\}$, then by  Lemma \ref{lem:Dp=Dq} and Lemma \ref{lem:spectrum-unique}, we have $(n-1-2j)a(p)=(n-1-2j)a(q)$ for any $q\in U$.  As a consequence,   $a$ is constant. 


Below we assume that $D_{p}(r)$ only have zero-exponent terms, in other words, it is a polynomial. 
Note that   $-1<C<1$ implies that $a,b\ne0$. Therefore the   point  $p$ that we have fixed and $N$ satisfies the  assumption on $p^0$ in Theorem \ref{thm:reduction} and the assumption on $\zeta$ in Lemma \ref{lem:det-rigidity}.
By Lemma \ref{lem:det-rigidity},   $A D_{1}\subset D_{1}$ and $A|_{D_1}$ has eigenvalues $a$ and $-a$ with the same multiplicities, where $D_{1}$ is defined by equation \eqref{eq:D1D2-definitions}. By Lemma \ref{lem:cartanformula}, such $L$ is not possible.

  This shows that  $D_{p}(r)$ must have non-zero exponent, and hence $a$ must be a constant.

\subsection{Case IV: \texorpdfstring{\(\Hyp^n\times \Hyp^m\)}{H x H}}


In this case, first factor and second factor are both hyperbolic space.  By Lemma \ref{lem:jacobi-matrix-row} (iv) and (v), we have  $B_i(r)=\frac{e^{ar}}{2}\left(e_i-\frac{1}{a}h_{i*}\right)
       +\frac{e^{-ar}}{2}\left(e_i+\frac{1}{a}h_{i*}\right)$,
for $2\leq i\leq n$, and $B_j(r)=\frac{e^{br}}{2}\left(e_j-\frac{1}{b}h_{j*}\right)
       +\frac{e^{-br}}{2}\left(e_j+\frac{1}{b}h_{j*}\right)$,
for $n+1\leq j\leq n+m-1$
where $e_i$ and  $h_{i*}$ are defined same as case 1. By Lemma \ref{lem:jacobi-matrix-row} (i),  the mixed row is polynomial in $r$ with degree $1$.  With the similar calcalations, we have 
\begin{equation}\label{equ:HH-spectrum}
D_p(r)=2^{-(n+m-2)}\sum_{i=0}^{n-1}\sum_{j=0}^{m-1} e^{[(n-1-2i)a+(m-1-2j)b]r}P_{ij}(r),
\end{equation}
where every $P_{ij}(r)$ is a polynomial of degree at most one. 

Similar to the case $H^{n}\times \R^m$,  if $D_{p}(r)$ has a  term with an exponent that is not zero, i.e., $(n-1-2i)a(p)+(m-1-2j)b(p)\neq 0$ for some $i\in \{0,1,\cdots,n-1\}$, $j\in \{0,1, \cdots, m-1\}$, then by Lemma \ref{lem:Dp=Dq} and Lemma \ref{lem:spectrum-unique}, we have 
\begin{equation*}
    (n-1-2i)a(p)+(m-1-2j)b(p)=(n-1-2i)a(q)+(m-1-2j)b(q)
\end{equation*}
for any $q\in U$, where $U$ is a neighborhood of $p$. Together with $a^2+b^2=1$,   $a$ is constant. 

Below assume that all terms in $D_{p}(r)$ do not have non-zero exponent.
Let $\Lambda_{n-1}=\{n-1,n-1-2,\ldots,-(n-1)\}$ and $\Lambda_{m-1}=\{m-1,m-1-2,\ldots,-(m-1)\}$.  Exponents of all term in equation \eqref{equ:HH-spectrum} have the form $ka+\ell b$ with $(k,\ell)\in\Lambda_{n-1}\times\Lambda_{m-1}$. We give a lemma for later use.

\begin{lemma}\label{lem:kl=k'l'}
  If  two different pairs  $(k,\ell),(k',\ell')\in\Lambda_{n-1}\times\Lambda_{m-1}$ give the same exponent:
$ka+\ell b=k'a+\ell'b$, then $a$ is constant on $L$.
\end{lemma}

\begin{proof}
Rearranging gives
$(k-k')a+(\ell-\ell')b=0$.
Since \((k,\ell)\ne(k',\ell')\), then $|k-k'|,|\ell-\ell'|$ are both non-zero. As a consequence,  we have
\[
\frac{a}{b}=\left|\frac{\ell-\ell'}{k-k'}\right|.
\]
Moreover, the differences \(k-k'\) and \(\ell-\ell'\) are even integers, with
$0<|k-k'|\le2(n-1),
0<|\ell-\ell'|\le2(m-1)$.
This means
\begin{equation*}\label{eq:HH-detailed-finite-ratio-set}
\frac{a}{b}\in
\mathcal R_{(n-1),(m-1)}:=
\left\{
\frac{|v|}{|u|}:
 u,v\in2\mathbb Z,
\ 0<|u|\le2(n-1),
\ 0<|v|\le2(m-1)
\right\}.
\end{equation*}
This is a finite set.
If \(a/b\in\mathcal R_{(n-1),(m-1)}\), then
$a^2+b^2=1$
gives
$a=\frac{\rho}{\sqrt{1+\rho^2}}$
for some \(\rho\in\mathcal R_{(n-1),(m-1)}\).  Thus \(a\) belongs to a finite set, and by connectivity of $L$, \(a\) is  constant.
\end{proof}

Below we assume 
\begin{equation}\label{eq:HH-detailed-no-collision}
\frac{a}{b}\notin\mathcal R_{(n-1),(m-1)}.
\end{equation}
Let 
\begin{equation}\label{eq:x,y}
   x=e^{ar},y=e^{br}.
\end{equation}
Under equation \eqref{eq:HH-detailed-no-collision}, different monomials
$x^k y^\ell$
give different exponent \(ka+\ell b\).  

Split
$TL=\R \mathcal T\oplus D_{1}\oplus D_{2}$,
where
$D_{1}|_{p}=T_pL\cap (T\Hn^n\oplus 0),
D_{2}|_{p}=T_pL\cap(0\oplus T\Hn^m)$.

Write the shape operator as
\begin{equation}\label{eq:HH-detailed-A-block}
A=
\begin{pmatrix}
h_{11}&u^T&v^T\\
u&S&B\\
v&B^T&T
\end{pmatrix}.
\end{equation}
using an orthonormal basis
with respect to the decomposition of   $\R \mathcal T|_{p}\oplus D_{1}|_{p}\oplus D_{2}|_{p}$. Since  $S,T$ are symmetric   matrices, in particular, we choose orthonormal basis such that $S,T$ are diagonal    matrices.   By Lemma \ref{lem:jacobi-matrix-row}, the Jacobi matrix is
\begin{equation}\label{eq:Bpr-with-entries}
    B_{p}(r)=
    \begin{pmatrix}
1-rh_{11}&-ru^T&-rv^T\\
-\frac{\sinh(ar)}{a}u&\cosh(ar)I_{n-1}- \frac{\sinh(ar)}{a}S&-\frac{\sinh(ar)}{a} B\\
-\frac{\sinh(br)}{b}v&-\frac{\sinh(br)}{b}B^T&\cosh(br)I_{m-1}- \frac{\sinh(br)}{b}T
\end{pmatrix}.
\end{equation}

We have the following results on $B_p(r)$.
\begin{lemma}\label{lem:S2=a2I}
    $S^2=a^2 I$ with $\tr S=0$, and $T^{2}=b^2I$ with $\tr T=0$.
\end{lemma}
\begin{proof}
    By equation \eqref{eq:Bpr-with-entries}, the coefficient of \(r^0\) in \(D_p(r)=\det B_{p}(r)\) is 
\begin{equation*}
    \det\begin{pmatrix}
       \cosh(ar)I_{n-1}- \frac{\sinh(ar)}{a}S&-\frac{\sinh(ar)}{a} B\\
-\frac{\sinh(br)}{b}B^T&\cosh(br)I_{m-1}- \frac{\sinh(br)}{b}T
    \end{pmatrix}.
\end{equation*}
Moreover, the coefficient of \(r^0\) is just 
$2^{-(n+m-2)}F_{p}(x,y)$, where $x,y$ are defined in equation \eqref{eq:x,y} and $F_p(x,y)$ is defined by
\begin{equation}\label{eq:HH-detailed-pure-F}
F_{p}(x,y):=
\det\begin{pmatrix}
 x(I-S/a)+x^{-1}(I+S/a)&-(x-x^{-1})B/a\\[2pt]
 -(y-y^{-1})B^T/b&y(I-T/b)+y^{-1}(I+T/b)
\end{pmatrix}.
\end{equation}
  By equation \eqref{eq:HH-detailed-no-collision}, different monomials of  $x,y$ give different exponents of $r$. Since  all terms of \(D_{p}(r)\) do not have nonzero exponent of $r$, every nonconstant monomial of $x,y$ in \(F_{p}\) must vanish.  Hence
\begin{equation}\label{eq:HH-detailed-F-constant}
F_{p}(x,y)\text{ is a constant polynomial of $x,y$.}
\end{equation}

Set \(y=1\) in equation \eqref{eq:HH-detailed-pure-F}.  Then
$y-y^{-1}=0,
y(I-T/b)+y^{-1}(I+T/b)=2I_{m-1}$.
So
\[
F_{p}(x,1)=2^{m-1}\det\left(x(I-S/a)+x^{-1}(I+S/a)\right).
\]
Since \(F_{p}\) is constant, the one-variable determinant above has only the monomial \(x^0\).  Applying the Lemma \ref{lem:S2=a2I-eigenvalues-a--a} to  \(F_{p}(x,1)\) gives
\begin{equation}\label{eq:HH-detailed-S-resonance}
S^2=a^2I,
\qquad
\tr S=0.
\end{equation}
Set \(x=1\).  By the same calculations, we get
\begin{equation}\label{eq:HH-detailed-T-resonance}
T^2=b^2I,
\qquad
\tr T=0.
\end{equation}
\end{proof}

\begin{lemma}\label{lem:B=0}
    $B=0$.
\end{lemma}
\begin{proof}

By equations \eqref{eq:HH-detailed-S-resonance} and \eqref{eq:HH-detailed-T-resonance},  decompose the  spaces $D_1|_p,D_2|_p$ into eigenspaces:
$D_1|_p=D_1|_p^+\oplus D_1|_p^-,
D_2|_p=D_2|_p^+\oplus D_2|_p^-$,
where
$S=aI\text{ on }D_1|_p^+,
S=-aI\text{ on }D_1|_p^-$,
and
$T=bI\text{ on }D_2|_p^+,
T=-bI\text{ on }D_2|_p^-$.
The trace conditions imply
\begin{equation} \label{eq:X+=X-,Y+=Y-}
 \dim D_1|_p^+=\dim D_1|_p^- ,
\dim D_2|_p^+=\dim D_2|_p^- .  
\end{equation}

Let $P_X=x^{-1}I_{D_1|_p^+}+ xI_{D_1|_p^-},
P_Y=y^{-1}I_{D_2|_p^+}+ yI_{D_2|_p^-}$, where $I_{D_i|_{p}^{\pm}}$ is a diagonal matrix with entries $1$ at rows corresponding to $D_i|_{p}^{\pm}$ and entries $0$ at rows rorresponding to $D_i|_{p}^{\mp}$ for $i=1,2$.   By equation \eqref{eq:X+=X-,Y+=Y-},  $\det P_{X}=\det P_{Y}=1$, in particular, $P_{X}, P_{Y}$ are both invertible.
From equations \eqref{eq:HH-detailed-pure-F}, \eqref{eq:HH-detailed-S-resonance} and \eqref{eq:HH-detailed-T-resonance}, we have
\begin{equation}\label{eq:HH-detailed-pure-matrix-after-resonance}
F_p(x,y)=
2^{n+m-2}\det\begin{pmatrix}
P_X&-\dfrac{x-x^{-1}}{2a}B\\[4pt]
-\dfrac{y-y^{-1}}{2b}B^T&P_Y
\end{pmatrix}.
\end{equation}
Since \(P_X\) is invertible, by Schur Complement formula (\cite{horn2005basic}*{Theorem 1.1}) 
\[
F_p(x,y)=2^{n+m-2}\det(P_X)\det\left(P_Y-\frac{(x-x^{-1})(y-y^{-1})}{4ab}B^TP_X^{-1}B\right).
\]
Factoring out \(P_Y\),
\begin{equation}\label{eq:HH-detailed-B-schur}
F_p(x,y)=2^{n+m-2}\det(P_X)\det(P_Y)
\det\left(I-\frac{(x-x^{-1})(y-y^{-1})}{4ab}P_Y^{-1}B^TP_X^{-1}B\right).
\end{equation}

Since $\det(P_X)=1,
\det(P_Y)=1$, 
we obtain
\begin{equation*}
F_p(x,y)=
2^{n+m-2}\det\left(I-\alpha(x,y)G(x,y)\right),
\end{equation*}
where
\[
\alpha(x,y)=\frac{(x-x^{-1})(y-y^{-1})}{4ab},
\qquad
G(x,y)=P_Y^{-1}B^TP_X^{-1}B.
\]

Since \(F_p\) is constant,
then $\frac{\partial^2F_p}{\partial x\partial y}(1,1)=0$.
But
$\alpha(1,1)=0,
\alpha_x(1,1)=0,
\alpha_y(1,1)=0,
\alpha_{xy}(1,1)=\frac1{ab}$,
where $\alpha_{x},\alpha_{y},\alpha_{xy}$ means the partial derivatives of $\alpha$ with respect to $x,y,xy$.
Also
$G(1,1)=B^TB$.
For a small scalar \(\epsilon\),
\[
\det(I-\epsilon G)=1-\epsilon\tr G+O(\epsilon^2).
\]
Therefore, at \((1,1)\), only the mixed derivative of \(\alpha\) contributes to the mixed derivative of \(F_p\), and we get
$\frac{\partial^2F_p}{\partial x\partial y}(1,1)
=-\frac{2^{n+m-2}}{ab}\tr(B^TB)$, which is $0$ by the fact that $F_{p}$ is constant.
Thus
$\tr(B^TB)=0$.
Since
$\tr(B^TB)=|B|^2$, we have  
$B=0$.
\end{proof}

\begin{lemma}\label{lem:u=0v=0}
    $u=0,v=0.$
\end{lemma}
\begin{proof}

Now use the  Jacobi matrix $B_{p}(r)$, not only the coefficient of \(r^0\).  Because \(B=0\),  the  Jacobi matrix is
\[
\begin{pmatrix}
1-rh_{11}&-ru^T&-rv^T\\[2pt]
-\dfrac{\sinh(ar)}a u&P_X&0\\[6pt]
-\dfrac{\sinh(br)}b v&0&P_Y
\end{pmatrix}.
\]
Since \(\det P_X=\det P_Y=1\), by Schur Complement formula (\cite{horn2005basic}*{Theorem 1.1}), we obtain
\begin{equation}\label{eq:HH-detailed-mixed-schur}
D_p(r)=1-rh_{11}
-r\frac{\sinh(ar)}a u^TP_X^{-1}u
-r\frac{\sinh(br)}b v^TP_Y^{-1}v.
\end{equation}
Now decompose
$u=u_++u_-,
v=v_++v_-$, where $u_{\pm}=u\cdot I_{D_1|_p^{\pm}} $ and $v_{\pm}=v\cdot I_{D_2|_p^{\pm}} $.   Since $P_X^{-1}=xI_{D_1|_p^+}\oplus x^{-1}I_{D_1|_p^-},
P_Y^{-1}=yI_{D_2|_p^+}\oplus y^{-1}I_{D_2|_p^-}$,
we get
\[
\frac{\sinh(ar)}a u^TP_X^{-1}u
=\frac{x^2-1}{2a}|u_+|^2+\frac{1-x^{-2}}{2a}|u_-|^2,
\]
and
\[
\frac{\sinh(br)}b v^TP_Y^{-1}v
=\frac{y^2-1}{2b}|v_+|^2+\frac{1-y^{-2}}{2b}|v_-|^2.
\]
Therefore
\begin{equation*}
D_p(r)=1-rh_{11}
-r\left[\frac{x^2-1}{2a}|u_+|^2+\frac{1-x^{-2}}{2a}|u_-|^2+\frac{y^2-1}{2b}|v_+|^2+\frac{1-y^{-2}}{2b}|v_-|^2\right].
\end{equation*}
Under the  assumption equation \eqref{eq:HH-detailed-no-collision}, the monomials
$x^2,
x^{-2},
y^2,
y^{-2}$
have distinct nonzero exponents.  Since  all terms in \(D_{p}(r)\) have no nonzero exponents, the coefficients of these four monomials must vanish.  Hence
$|u_+|^2=|u_-|^2=|v_+|^2=|v_-|^2=0$.
Thus
$u=0,
v=0$.
\end{proof}

By Lemma \ref{lem:B=0} and Lemma \ref{lem:u=0v=0}, we have $A D_{1}|_p\subset D_{1}|_p$,
$A D_{2}|_p\subset D_{2}|_p$. Since $p$ is chosen to be any point in $L$ with $-1<C(p)<1$, we have $A D_{1}\subset D_{1}$,
$A D_{2}\subset D_{2}$ on the neighborhood $U$ of $p$ with $-1<C<1$. Also by Lemma \ref{lem:S2=a2I}, $S^2=a^2I$.
Applying Lemma \ref{lem:cartanformula}, we conclude that such $L$ is not possible, which implies the assumption \eqref{eq:HH-detailed-no-collision} is not possible. Therefore $a$ must be constant.

\section{Classification of Isoparametric Hypersurfaces in \texorpdfstring{$\mathbb Q_{\epsilon_{1}}^{n}\times \mathbb Q^m_{\epsilon_{2}}$}{Q e1 n x Qm e2}}\label{sec:classification-hypersurface}




Urbano \cite{urbano2019hypersurfaces} and  Gao--Ma--Yao \cites{gao2024hypersurfaces,gao2024isoparametric} obtain the classifications of  isoparametric hypersurfaces in \(Q^{2}_{\epsilon_1}\times Q^{2}_{\epsilon_2}\).  De Lima--Pipoli \cite{delima2025} classify the isoparametric hypersurfaces in \(Q^{n}_{\epsilon_1}\times Q^{m}_{\epsilon_2}\) with a one-point condition.  Tan--Xie--Yan \cite{TanXieYan2026} give a classification of isoparametric hypersurfaces in $\Sph^n\times\R^m$ and $\Hyp^n\times \R^m$. In this section, we use a simpler approach to obtain the classifications of isoparametric hypersurfaces in $\Sph^n\times\R^m$ and $\Hyp^n\times \R^m$. The method also gives the classification of isoparametric hypersurfaces in $\Sph^n\times \Hyp^m$.


\begin{theorem}\label{thm:global-dem-SRH}
Let $L\subset \mathbb Q_{\epsilon_{1}}^{n}\times \mathbb Q^m_{\epsilon_{2}}$ be an isoparametric hypersurface.

(i) If $\epsilon_{1}=1, \epsilon_{2}\in\{0,-1\}$, then the angle function satisfies $C\equiv 1$ or $C\equiv -1$.  Moreover, $L$ is an open subset of $L_1\times \mathbb Q^m_{\epsilon_{2}}$  or $\mathbb Q_{\epsilon_{1}}^n\times L_2$, where $L_1$, $L_2$ are isoparametric hypersurfaces of  $ \mathbb Q^n_{\epsilon_{1}}$,$\mathbb Q^m_{\epsilon_{2}}$, respectively.

(ii) If $\epsilon_{1}=-1, \epsilon_{2}=0$, then \(L\) is an open subset of  one of the following:

\qquad (a) A product
$L_1\times \Q_{0}^{m}$,
where \(L_1\subset \Q_{-1}^{n}\) is an isoparametric
hypersurface;

\qquad (b) A product
$L=\Q_{-1}^{n}\times L_2$,
where \(L_2\subset \Q_{0}^{m}\) is an 
isoparametric hypersurface;  

\qquad (c) A flat-horospherical hypersurface.  
\end{theorem}

 By Theorem \ref{thm:constant-angle}, the angle function is constant. Note that $C\equiv1$ and $C\equiv-1$ correspond to $|N_2|=0$ or $|N_1|=0$, respectively. In these cases $TL$ decomposes as $D_1\oplus D_2$, then by Lemma \ref{lem:localproduct} and Lemma \ref{lem:gluing}, $L$ decomposes, i.e., $L$ is an open subset of product of two isoparametric submanifold of  $ \mathbb Q^n_{\epsilon_{1}}$,$\mathbb Q^m_{\epsilon_{2}}$. Since the codimension of $L$ is $1$, $L$ is either $L_1\times \Q^m_{\eps_2}$ or $\mathbb Q^n_{\eps_1}\times L_2$ where $L_1$ and $L_2$ are isoparametric hypersurfaces in $\mathbb Q^n_{\eps_1}$ and $\Q^m_{\eps_2}$, respectively. Therefore we only need to consider the case $-1<C<1$. More concretely, we will show the case $-1< C<1 $ in (i) of Theorem \ref{thm:global-dem-SRH} is impossible and the case $-1< C<1 $ in (ii) of Theorem \ref{thm:global-dem-SRH} is corresponding to the classification (c). For the rest of the section, assume $-1<C<1$.


\subsection{Derivatives of Shape Operator}
Derivatives of shape operator along the direction of ${\mathcal T}$ in product space of simply connected space forms are derived by de Lima and Pipoli \cite{delima2025} using Codazzi equations. We set the notations and recall their results here.

 For any $p\in L$,  there is a neighborhood $U$ of $p$, such that a  unit normal field $N$  is defined on $U$. Here $N=N_1+N_2$   by  equation (\ref{eq:N=N1+N2}). Set $a=||N_1||= \sqrt{(1+C)/2}$, $b=||N_2||=\sqrt
{(1-C)/2}$ as in equation (\ref{eq:defn-of-a-b}). By Theorem \ref{thm:constant-angle}, $a,b$ are constant.

Recall that ${\mathcal T}=P(N)-CN$, where $P$ is defined by \eqref{equ:P-operator}. Equation \eqref{equ:T-length} yield that  $|{\mathcal T}|=\sqrt{1-C^2}$ and is constant. 

Let $D$ be the distribution ${\mathcal T}^\perp$ on $TL|_U$.  For a tangent vector $X\in D$, by equation \eqref{eq:P-property}, 

we have 
\begin{align*}
   \langle P(X),N\rangle=\langle X,P(N)\rangle=\langle X,{\mathcal T}\rangle=0,
\end{align*}
\begin{align*}
   \langle P(X),{\mathcal T}\rangle=\langle X, P({\mathcal T})\rangle=\langle X,P(P(N)-CN)\rangle=\langle X,(1-C^2) N-C{\mathcal T}\rangle=0.
\end{align*}
Therefore $P$ preserves $D$. By equation \eqref{eq:P-property},  we denote the eigenspace splitting of $P|_D$ by
\begin{equation}\label{eq:D+D-}
  D=D_1\oplus D_2,  
\end{equation}
 where $D_{1}|_{p}=D_{p}\cap (T_{\pi_{1}(p)}\mathbb Q_{\epsilon_{1}}^n\oplus 0)$,  $D_2|_{p}=D_{p}\cap (0\oplus T_{\pi_{2}(p)}\mathbb Q^m_{\epsilon_{2}})$ and  $\dim D_1=n-1$, $\dim D_2=m-1$, for any $p\in U$. Since $C$ is constant on  $L$ and  Lemma \ref{lem:nabla-C}, we have $A{\mathcal T}=0$ for ${\mathcal T}=P(N)-CN$. Since $A{\mathcal T}=0$ and $A$ is self-adjoint, we have  $A(D)\subset D$. Consider an orthonormal frame of $D_{1}\oplus D_{2}$ constructed from an orthonormal frame of $D_1$ and an orthonormal frame of $D_2$. In this orthonormal frame, write 
\begin{align*}
    A_D=A|_D=
\begin{pmatrix}
A_1 & A_2\\
A_2^t & A_3
\end{pmatrix},
\end{align*}
where $A_1=A_1^t$, $A_3=A_3^t$, and $A_2$ is rectangular. And  we have the following equations.

\begin{lemma}(Equation 18 in \cite{delima2025})\label{lem:Riccati-equ}
    With some notations above,  we have
\[
{\mathcal T}(A_1)=(1-C)A_1^2-(1+C)A_2A_2^t+\epsilon_{1}\frac{1-C^2}{2}I_{n-1},
\]
\[
{\mathcal T}(A_3)=(1-C)A_2^tA_2-(1+C)A_3^2+-\epsilon_{2}\frac{1-C^2}{2}I_{m-1},
\]
\[
{\mathcal T}(A_2)=(1-C)A_1A_2-(1+C)A_2A_3.
\]
These equations can also be written as 
\[
{\mathcal T}(A_D)=A_D(P|_D-CI)A_D+\epsilon_{1}\frac{1-C^2}{2}\Pi_1-\epsilon_{2}\frac{1-C^2}{2}\Pi_2,
\]
where $\Pi_1$ and $\Pi_2$ are $\begin{pmatrix}
    I_{n-1} & 0 \\
    0 & 0
\end{pmatrix}$ and $\begin{pmatrix}
    0 & 0 \\
    0 & I_{m-1}
\end{pmatrix}$, respectively, that is, the projections of $D$ onto $D_1$ and $D_2$, respectively. 
\end{lemma}

\subsection{Jacobi Matrix}
Here we recall the Jacobi Matrix from the previous section and tailor it to our needs.

For $p\in L$, let $\gamma_p(r)=\exp_p(rN_p)$ and let $L_r$ be the local parallel hypersurface.  
For simplicity of notations, we define 
\begin{align}\label{eq:xr,yr}
    x(r)=s_+(r)/c_+(r),\ y(r)=s_-(r)/c_-(r), 
\end{align}
where $c_{+}(r), s_{+}(r), c_{-}(r),s_{-}(r)$ are defined by equations \eqref{eq:c+s+},\eqref{eq:c-s-}.

By Lemma \ref{lem:jacobi-matrix-row}, we can write down a Jacobi matrix $B_p(r)$. Since we have $A{\mathcal T}=0$, the first row as well as the first column only contains zero entries. We eliminate the first row and the first column to get the following reduced Jacobi matrix.
\[
J_p(r)=
\begin{pmatrix}
c_+(r)I_{n-1}-s_+(r)A_1 & -s_+(r)A_2\\
-s_-(r)A_2^t & c_-(r)I_{m-1}-s_-(r)A_3
\end{pmatrix}.
\]


Taking the determinant and by equation \eqref{eq:xr,yr} gives
\begin{align*}
    D_{p}(r)=c_+(r)^{n-1}c_-(r)^{m-1}F_p(x,y),
\end{align*}
where 
\begin{align}\label{eq:F-p-(x,y)}
  F_p(x,y)=  \det
\begin{pmatrix}
I_{n-1}-xA_1 & -xA_2\\
-yA_2^t & I_{m-1}-yA_3
\end{pmatrix}.
\end{align}

Note that factors $c_+(r), c_-(r)$ depends only on $r$ and the constant angle funtion $C$, not on $p$. Since $D_{p}(r)$ is independent of $p$ in $U$ by Lemma \ref{lem:Dp=Dq}, then $F_p(x(r),y(r))$ is also independent of $p$ in $U$ for sufficiently small $r$.


\begin{lemma}\label{lem:polynomial-identity}
Let  function pair $(\psi(r),\varphi(r))$ be $(\dfrac{\tan(ar)}{a}, r),\,  (\dfrac{\tan(ar)}{a}, \dfrac{\tanh(br)}{b})$ or $ (\dfrac{\tanh(ar)}{a}, r)$.
If a polynomial $P(x,y)\in\mathbb R[x,y]$ satisfies $P(\psi(r),\varphi(r))=0$  for all real $r$ near $0$, then $P=0$ as a polynomial. In particular, if two real coefficient polynomials $P_1$ and $P_2$ satisfies $P_1(\psi(r),\varphi(r))=P_2(\psi(r),\varphi(r))$  for all real $r$ near $0$, then $P_1=P_2$ as polynomials.

\end{lemma}

\begin{proof}

We write $P$ as a polynomial of $x$: 
\begin{align}\label{eq:P(x,y)}
    P(x,y)=p_d(y)x^d+p_{d-1}(y)x^{d-1}+\cdots+p_{1}(y)x+p_{0}(y),
\end{align}
where $p_i$ is a polynomial for $i=0,1,\cdots, d$ and $d$ is the highest degree of $x$.
The leading coefficient $p_d(y)$ is non-vanishing if $P$ is non-vanishing. Substituting with $x=\psi(z), y=\varphi(z)$, we obtain a meromorphic function of $z$ which vanishes on a real interval, hence vanishes everywhere. 

For $\psi(r)=\dfrac{\tan(ar)}{a}$, 
we choose the poles $z_k=(\pi/2+k\pi)/a$ of $\tan(az)$.  The function $\varphi$ is holomorphic at $z_{k}$.
Fix a $k$ for now, multiply  $(z-z_k)^d$  in both side of equation \eqref{eq:P(x,y)} and take $z\to z_k$, the limit gives 
\begin{align*}
   0&=\lim_{z\to z_k}(\sum_{i=0}^d p_i(\varphi(z))(\frac{\tan(az)}{a})^i(z-z_k)^i)\\
   &=\lim_{z\rightarrow z_k}(p_d(\varphi(z))(\frac{\tan(az)}{a})^d(z-z_k)^d)=p_d(\varphi(z_k))\frac{(-1)^d}{a^{2d}},
\end{align*}
i.e., $p_d(\varphi(z_k))=0$.
Note that $\varphi$ is monotonely increasing. Therefore, for any $k$, $\varphi(z_k)$ is a zero point for the non-vanishing polynomial $p_d$, which is not possible. Hence $p_d$ vanishes, and $P$ vanishes.

For $\psi(r)=\dfrac{\tanh(ar)}{a}$, we choose the poles $ z_k=\frac{i\pi}{a}\left(k+\frac12\right)$ of \(\tanh(az)\).
Multiplying by \((z-z_k)^d\) in both side of equation \eqref{eq:P(x,y)}  and taking \(z\to z_k\), the limit gives
\begin{align*}
   0&=\lim_{z\to z_k}(\sum_{i=0}^d p_i(z)(\frac{\tanh(az)}{a})^i(z-z_k)^i)\\
   &=\lim_{z\rightarrow z_k}(p_d(z)(\frac{\tanh(az)}{a})^d(z-z_k)^d)=\frac{p_d(z_k)}{a^{2d}},
\end{align*}
i.e., \(p_d(z_k)=0\) for
infinitely many \(k\).  A nonzero polynomial cannot have infinitely many zeros.
Thus \(p_d=0\), a contradiction.  Therefore \(P=0\).
\end{proof}

\begin{corollary}\label{cor:F-p-constant-coeffi}
   The coefficients  of $F_{p}(x,y)$ are independent of $p$ in $U$.
\end{corollary}
\begin{proof}
  For two points $p_1,p_2\in U$, the difference $F_{p_1}(x(r),y(r))-F_{p_2}(x(r),y(r))$ vanishes for all small $r$.  The lemma \ref{lem:polynomial-identity} implies $F_{p_1}=F_{p_2}$ as polynomials.   
\end{proof}

\subsection{Kernel of the Reduced Shape Operator}

Setting $x=y=t$ in equation \eqref{eq:F-p-(x,y)} gives 
\begin{align}\label{eq:F-t-t}
  F_p(t,t)=\det(I-tA_D).
\end{align}
 By Corollary \ref{cor:F-p-constant-coeffi}, the characteristic polynomial of $A_D$ is independent of $p$ in $U$, which means eigenvalues along with their multiplicities are independent of $p$. 
Therefore the dimension of $\ker A_D$ is locally constant. 

If $\ker A_D\neq 0$, take a unit vector field $X=X_1+X_2 \in \ker A_D$, where $X_1, X_2$ denote the component of $X$ in $D_1, D_2$. Then $A_{D}X=0$.  Differentiating $\langle A_DX,X\rangle=0$ along ${\mathcal T}$ gives 
\begin{align}\label{eq:T-A-D}
  0=\langle A_D\nabla_{\mathcal T}X,X\rangle+\langle {\mathcal T}(A_D)X,X\rangle+\langle A_DX,\nabla_{\mathcal T}X\rangle  
\end{align}
Since $A_DX=0$ and $A_{D}$ is symmetric, the first term and third term of right side of equation \eqref{eq:T-A-D} vanish. Then we have 
\begin{align}\label{eq:T-A-D-0}
    \langle {\mathcal T}(A_D)X,X\rangle=0
\end{align}
on the other hand, by $A_DX=0$, the Lemma \ref{lem:Riccati-equ} gives 
\begin{align}\label{eq:T-A-D-X-X}
    \langle {\mathcal T}(A_D)X,X\rangle=\epsilon_{1}\frac{1-C^2}{2}\|X_1\|^2-\epsilon_{2}\frac{1-C^2}{2}\|X_2\|^2.
\end{align}

\subsection{Proof of Case (i) of  Theorem \ref{thm:global-dem-SRH} }\label{subsec:Case-i-of-Thm}


We will show it is impossible for $-1<C<1$. In this case, $\epsilon_{1}=1$.
Consider the equation \eqref{eq:T-A-D-X-X},
 if $\epsilon_{2}=-1$, this is impossible unless $X=0$, so $\ker A_D=0$ and $A_D$ is invertible in the hyperbolic product.  If $\epsilon_{2}=0$, it implies only $X_1=0$.  Hence the  kernel lies entirely in $D_{2}$. In both cases we can now define the reduced space 
 \begin{equation}\label{eq:tildeD=D1oplusD2}
     \widetilde D=D_{1}\oplus \widetilde D_{2},
 \end{equation} where $\widetilde D_{2}=D_{2}$ if $\eps_2=-1$, and $\widetilde D_{2}=(\ker A_D)^\perp\cap D_{2}$ if $\eps_2=0$.  Let $\tilde m=\dim \widetilde D_{2}$.  The reduced operator 
 \begin{equation}\label{eq:tildeA=AD}
     \widetilde A:=A_D|_{\widetilde D}
 \end{equation} is symmetric and invertible, since $A(\ker A_D)=0$. We have, for the characteristic polynomials of $\widetilde A$ and $A_D$,
 \begin{equation}\label{eq:det=det}
     \det(I_{n+m-2-\tilde m}-t\widetilde A)=\det(I_{n+m-2}-tA_D),
 \end{equation} since $\widetilde A=A_D|_{\tilde D}$ and $A_D|_{\ker A_D}=0$.

Set $M=\widetilde A^{-1}$ and write it in $D_{1}\oplus \widetilde D_{2}$ blocks as $M=
\begin{pmatrix}
M_1 & M_2\\
M_2^t & M_3
\end{pmatrix}$.
Differentiating $M=\widetilde A^{-1}$ and using the reduced form of Lemma \ref{lem:Riccati-equ}  yields
\begin{align*}
    {\mathcal T}(M)=&-M{\mathcal T}(\widetilde A)M=M\left(\widetilde A(CI-P|_{\widetilde D})\widetilde A-\frac{1-C^2}{2}\Pi_1-\epsilon_{2}\frac{1-C^2}{2}\Pi_{\widetilde D_{2}}\right)M\\
    =&CI-P|_{\widetilde D}-M\left(\frac{1-C^2}{2}\Pi_1-\epsilon_{2}\frac{1-C^2}{2}\Pi_{\widetilde D_{2}}\right)M
\end{align*}
where $\Pi_{\widetilde D_{2}}$ is projection of $\widetilde D$ onto $\widetilde D_{2}$.
Taking the $D_{1}$ block and using $P|_{D_{1}}=I_{n-1}$ on $D_{1}$, we have ${\mathcal T}(M_1)=(C-1)I_{n-1}-\frac{1-C^2}{2}(M_1^2-\epsilon_{2} M_2M_2^t)$.
Taking trace for the $D_{1}$ block gives
${\mathcal T}(\tr M_1)=(C-1)(n-1)-\frac{1-C^2}{2}\bigl(\|M_1\|^2-\epsilon_{2}\|M_2\|^2\bigr)$.
Here  $C-1<0$, $\epsilon_{2}\in \{0,-1\}$, so ${\mathcal T}(\tr M_1)<0$.

On the other hand, the determinant computation gives a contradicting conclusion. Set $X_{x,y}=\operatorname{diag}(xI_{n-1},yI_{\tilde m})$.
By equation \eqref{eq:F-p-(x,y)}, $F_{p}(x,y)=\det(I-X_{x,y}\widetilde A)$. 
Expand $\det(I_{n+m-2-\tilde m}-t\widetilde A)$ as a polynomial of $t$, the highest degree term has non-zero coefficient $\det\widetilde A$, which by Corollary \ref{cor:F-p-constant-coeffi} and equation \eqref{eq:det=det}, is independent of $p$ in $U$.
Then 
\begin{align*}
 \det(M-X_{x,y})=\frac{1}{\det(\widetilde A)}\det(I-X_{x,y}\widetilde A)= \frac{1}{\det(\widetilde A)}F_{p}(x,y)  
\end{align*}
has coefficients independent of $p$ in $U$, by Corollary \ref{cor:F-p-constant-coeffi}.  Regard $\det(M-X_{x,y})$ as a polynomial in $y$.  The coefficient of $y^{\tilde m}$ is $(-1)^{\tilde m}\det(M_1-xI_{n-1})$.  Therefore $\det(M_1-xI_{n-1})$ is independent of $p$ in $U$, and in particular $\tr M_1$ is constant on $L$, which contradicts ${\mathcal T}(\tr M_1)<0$.  Hence the tilted case $-1<C<1$ cannot occur i.e., $C\equiv 1$ or $C\equiv -1$.

\subsection{Proof of the Case (ii) of Theorem \ref{thm:global-dem-SRH}}

We give a lemma first.

\begin{lemma} \label{lem:flat-horosphere}
    Let $L_{1}, L_{2}$ be  flat-horospherical hypersurfaces (defined by equation \eqref{eq:flat-horo-repre}), i.e., $L_{1}=\cup_{t\in \mathbb R} H_{t}\times W_{\alpha t} $, $\widetilde L_{1}=\cup_{t\in \mathbb R} \widetilde H_{t}\times \widetilde W_{\widetilde\alpha t}$, where $H_{0},\widetilde H_{0}$ are horospheres in $\Hyp^{n}$, $W_{0},\widetilde W_{0}$ are hyperplanes in $\mathbb R^{m}$, $\alpha, \widetilde \alpha $ are constant. If $L_{1}\cap \widetilde L_{1}\neq \varnothing $ is an open subset of $L_1$, then $L_{1}=\widetilde L_{1}.$
\end{lemma}
\begin{proof}
Let $C_{i}$ be the angle function for $L_{i}$ for $i=1,2$.
By Theorem \ref{thm:constant-angle}, the angle functions $C_1,C_2$  are constant.

Let  $U=L_{1}\cap \widetilde L_{1}$. For flat-horospherical hypersurface,  we have  $C_1=\frac{\alpha^2-1}{\alpha^2+1}$ and $C_2=\frac{\widetilde\alpha^2-1}{\widetilde\alpha^2+1}$. Since  in $U$, $C_1=C_2$. As a consequence, $\alpha=\pm\widetilde \alpha$. Let $TL_{1}=\R E_{1}\oplus D$, and $T\widetilde L_{1}= \R \widetilde E_{1}\oplus \widetilde D$, where $E_{1}, \widetilde E_{1}$ are defined in equation \eqref{eq:E1},

$D, \widetilde D$ are defined as $E_{1}^{\perp}$ in $TL$, $\widetilde E_{1}^{\perp}$ in $T\widetilde L$ respectively.
   In $U$, we can choose for $L_{1}$, $\widetilde L_{1}$  the same normal vector field ($N=\tilde N$), and they share the  same tangent bundle ($TL=T\widetilde L$). Then $\mathcal T=PN-C_{1}N=P\widetilde N-C_{2} \widetilde N=\widetilde {\mathcal{T}} $.  Hence we have $E_{1}=\widetilde E_{1}$

   and $D=\widetilde D$. 
   Then at any $x$ in $U$, the leaf of $D$ through $x$, $H_{t_0}\times W_{\alpha t_0}$ must be leaf of $\widetilde D$ through $x$, $\widetilde H_{t_1}\times \widetilde W_{\widetilde \alpha t_1}$, where $t_0,t_1\in \R$. As a consequence, $L_{1}=\widetilde L_{1}$.

\end{proof}



We will show case $-1< C<1 $ in (ii) of Theorem \ref{thm:global-dem-SRH} is corresponding to the case (c) in Theorem \ref{thm:global-dem-SRH}.
In this case, $\epsilon_{1}=-1$ and $\epsilon_{2}=0$. Combing the equation \eqref{eq:T-A-D-0} and equation \eqref{eq:T-A-D-X-X}, we have
\[
        0=-\frac{1-C^2}{2}\|X_1\|^2.
\]
Thus \(X_1=0\), and
\begin{equation}\label{eq:ker-in-Dminus}
        \Ker A_D\subset D_{2}.
\end{equation}

Let  
\[
        \widetilde D_{2}=D_{2}\cap(\Ker A_D)^\perp,
        \qquad
        \widetilde m=\dim\widetilde D_{2}.
\]
Since \(A_D\) is self-adjoint, then \((\Ker A_D)^\perp\) is $A_{D}$- invariant.  Thus
\[
        \widetilde D=D_{1}\oplus\widetilde D_{2}
\]
is also \(A_D\)-invariant, and
\[
        \widetilde A=A_D|_{\widetilde D}
\]
is  invertible.
\begin{lemma}\label{lem:tilde-m=0}
    $\Ker A_D= D_{2}$, i.e., $\widetilde m=0$.
\end{lemma}
\begin{proof}
If not, for contradiction,
assume that \(\widetilde m>0\). With the same calculations, we can also obtain equation \eqref{eq:det=det}. By Corollary \ref{cor:F-p-constant-coeffi} and equation \eqref{eq:F-t-t}, we know $\det \widetilde A$ is independent of $p$ in $U$.

  Write
\[
        M=\widetilde A^{-1}=
        \begin{pmatrix}
        M_1&M_2\\
        M_2^t&M_3
        \end{pmatrix}
\]
with respect to \(D_{1}\oplus\widetilde D_{2}\).  Differentiating $M=\widetilde A^{-1}$ and using the reduced form of Lemma \ref{lem:Riccati-equ}  yields
\begin{align*}
    {\mathcal T}(M)=&-M{\mathcal T}(\widetilde A)M=M\left(\widetilde A(CI-P|_{\widetilde D})\widetilde A+\frac{1-C^2}{2}\Pi_1\right)M\\
    =&CI-P|_{\widetilde D}+\frac{1-C^2}{2}M\Pi_1M
\end{align*}

Taking the \(\widetilde D_{2}\)-block and using $P|_{\widetilde D_{2}}=-I_{\widetilde m}$ yields
\[
        {\mathcal T}(M_3)=(1+C)I_{\widetilde m}+\frac{1-C^2}{2}M_2^tM_2.
\]
Therefore
\begin{equation}\label{eq:trace-positive}
        {\mathcal T}(\tr M_3)=(1+C)\widetilde m+\frac{1-C^2}{2}\|M_2\|^2>0.
\end{equation}

On the other hand, set
\[
        X_{x,y}=\diag(xI_{n-1},yI_{\widetilde m}).
\]
Then 
\[
\begin{aligned}
        \det(M-X_{x,y})
        &=\frac{1}{\det\widetilde A}\det(I-X_{x,y}\widetilde A)
\end{aligned}
\]
has coefficients independent of $p$ in $U$.  Regard
\(\det(M-X_{x,y})\) as a polynomial in \(x\).  The coefficient of \(x^{n-1}\) is
\[
        (-1)^{n-1}\det(M_3-yI_{\widetilde m}).
\]
Thus \(\det(M_3-yI_{\widetilde m})\) is constant on \(L\). We note that $(-1)^{\tilde{m}-1} \tr M_3$ is exactly the coefficient of the $y^{\tilde{m}-1}$ term.   In particular,     \(\tr M_3\) is also  constant on \(L\).

This contradicts equation \eqref{eq:trace-positive}.  Hence
\(\widetilde m=0\), and therefore
\[
        \Ker A_D=D_{2}.
\]
\end{proof}

By Lemma \ref{lem:tilde-m=0},
\begin{equation}\label{eq:A2A3zero}
        A_2=0,
        \qquad
        A_3=0.
\end{equation}

By equation \eqref{eq:A2A3zero},
\[
        A_D=A_1\oplus 0.
\]
\begin{lemma}
   The eigenvalues of $A_{1}$ is $a$ or $-a$. 
\end{lemma}
\begin{proof}
By Lemma \ref{lem:Riccati-equ}, equation \eqref{eq:A2A3zero} and the fact that $a^2=(1+C)/2$, we have
\begin{equation}\label{eq:A1reduced}
        {\mathcal T}(A_1)=(1-C)\left(A_1^2-a^2I_{n-1}\right).
\end{equation}

By Corollary \ref{cor:F-p-constant-coeffi} and equation \eqref{eq:F-t-t}, the eigenvalues of $A_{1}$ which is also $A_{D}$ are constant.
Let \(\lambda\) be a constant eigenvalue of \(A_1\), and let \(X\) be a local
unit eigenvector field  with respect to $\lambda$.

Since \(\lambda=\langle A_1X,X\rangle\) is constant, then $0={\mathcal T}(\lambda)={\mathcal T} \langle A_1X,X\rangle$. On the other hand, 
\begin{align*}
   {\mathcal T} \langle A_1X,X\rangle
   &=\langle {\mathcal T}(A_{1})X,X\rangle+\langle A_{1}\nabla_{{\mathcal T}}X,X\rangle+ \langle A_{1}X,\nabla_{{\mathcal T}}X\rangle\\
   &=\langle {\mathcal T}(A_{1})X,X\rangle, 
\end{align*}
where the second $"="$ holds since $A_{1}$ is symmetric, $A_{1}X=\lambda X$ and $X$ is unit. As a consequence, 
\begin{align}
    \langle {\mathcal T}(A_{1})X,X\rangle=0.
\end{align}
By equation \eqref{eq:A1reduced},
\[
        0=\langle {\mathcal T}(A_1)X,X\rangle=(1-C)(\lambda^2-a^2).
\]
Thus every eigenvalue of \(A_1\) is \(a\) or \(-a\).
\end{proof}

Since $D_{1}=\ker d\pi_{2}|_{TL}$ and $D_{2}=\ker d\pi_{1}|_{TL}$,
The distributions \(D_{1}\) and \(D_{2}\) are integrable. 
\begin{lemma}\label{lem:flat-horospheric-proof}
Any leaf of $D_{1}$ is an open subset of horosphere in $\Hyp^{n}$, any leaf of $D_{2}$ is an open subset of hyperplane in $\R^m$.
\end{lemma}
\begin{proof}
Let \(M_+\) be any  leaf of $D_{1}$.  Since all
vectors in \(D_{1}\) have zero \(\Q_0^m\)-component, \(M_+\) lies in a slice
\(\Q_{-1}^{n}\times{q_0}\), where $q_{0}$ is a point in $\Q_0^m$. For \(X\in D_{1}\),

$\bar\nabla_XN=-A_1X$.
 Since $a$ is constant on $L$,
taking the \(\Q_{-1}^{n}\)-component gives
$a\nabla_X^{\Q_{-1}^{n}\times {q_0}}\nu_1=-A_1X$.

Therefore the shape operator of \(M_+\subset\Q_{-1}^{n}\times {q_0}\), with normal \(\nu_1\),
is
$A^{M_+}=\frac{1}{a}A_1$.
Its principal curvatures are \(+1\) or \(-1\).  Cartan's formula in
\(\Q_{-1}^{n}\) forbids the simultaneous occurrence of \(+1\) and \(-1\) with
positive multiplicities (same as the analysis in Section 7.3.4) .  Hence all principal curvatures are \(+1\), or all are
\(-1\). By \cite{ryan1971hypersurfaces}*{Theorem 1}, \(M_+\) is an open subset 
of a horosphere.

Similarly, for \(Y\in D_{2}\), \(A_DY=0\) and
$\bar\nabla_YN=0$.
The \(\Q_0^m\)-component gives
$\nabla_Y^{p_{0}\times\Q_0^{m}}\nu_2=0$,
where $p_{0}$ is a point in $\mathbb Q^{n}_{-1}$.

Therefore the \(D_{2}\)-leaves are totally geodesic and thus are open subset of affine hyperplanes in \(p_{0}\times\Q_0^m\).
\end{proof}

Now consider an integral curve $\gamma(s)$ of the unit field $E_1$, so
$\gamma'(s)=E_1$.

Because $AE_1=0$, the Weingarten formula gives
$\nab_{E_1}N=-AE_1=0$,
where $N=a\nu_1 + b\nu_2$.
Since $a,b$ are constant,
\[
0=\nab_{E_1}N=a\nab_{E_1}\nu_1 + b\nab_{E_1}\nu_2.
\]
The two terms lie in different factors of the product, then each term must vanish 
$\nab_{E_1}\nu_1=0,
\nab_{E_1}\nu_2=0$.
Thus $\nu_1$ and $\nu_2$ are parallel along the $E_1$-curves.

Now differentiate
$E_1=b\nu_1-a\nu_2$.
We get
\[
\nab_{E_1}E_1=b\nab_{E_1}\nu_1-a\nab_{E_1}\nu_2=0.
\]
Therefore
$\nab_{\gamma'}\gamma'=0$.
Hence the integral curves of $E_1$ are  geodesics in $\mathbb Q^n_{-1}\times \Q^m_0$.

Choose one $D_{1}$ leaf $H\subset \mathbb Q^n_{-1}\times q_{0}$

 which is a horosphere with its unit normal $\eta=\nu_1$ and choose one $D_{2}$ leaf 
$W\subset p _{0}\times  \Q^m_0$
which is an affine hyperplane with  its constant unit normal 
$\xi=\nu_2$, where $p_{0}, q_{0}$ are points in $\mathbb Q^n_{-1}, \Q^m_0$ respectively.
Define the parallel families
$f_t(p)=\exp^{\Q_{-1}^{n}\times q_{0}}_p(t\eta_p),
p\in H$,
and
$g_t(q)=q+t\xi,
q\in W$.
Here $f_t(H)$ is the horosphere parallel to $H$ at  distance $t$, and $g_t(W)$ is the affine hyperplane parallel to $W$ at  distance $t$.

Let $s$ be arclength parameter along an $E_1$-curve starting at $(p,q)\in H\times W$. Since
$E_1=b\nu_1-a\nu_2$,
the $\mathbb Q^n_{-1}$-component of the velocity is $b\nu_1$ and the $\Q^m_0$-component is $-a\nu_2$. Also, the fields $\nu_1$ and $\nu_2$ are parallel along this curve. Therefore the curve is
\[
s\longmapsto \bigl(f_{bs}(p),g_{-as}(q)\bigr).
\]
Consequently, the hypersurface is locally
$\Psi(p,q,s)=\bigl(f_{bs}(p),g_{-as}(q)\bigr)$. After parametrization, it is  a flat-horosphere hypersurface defined   by equation \eqref{eq:flat-horo-repre}. By Lemma \ref{lem:flat-horosphere},
we finish the proof.







\section{Main Theorem}\label{sec:theorem}
We prove the main Theorem.

\maintheorem*

\begin{proof}
    By Theorem \ref{thm:lie-triple-stable} and the paragraph following it, case (ii) and case (iii) of Theorem~\ref{thm:lie-triple-stable} give decomposition $TL=D_1\oplus D_2$ where $D_1\subset T\Q^n_{\eps_1}\oplus 0$ and $D_2\subset 0\oplus TQ^m_{\eps_2}$. By Lemma \ref{lem:localproduct}, any $x\in L$ is contained in $L_{1,x}\times L_{2,x}$, where $L_{1,x}, L_{2,x}$ are isoparametric submanifolds of $\Q^{n}_\eps$, $\R^{m}$ respectively. We have $L=\cup_{x\in L} (L_{1,x}\times L_{2,x})$.  By Lemma \ref{lem:gluing}, $L\subset L_1\times L_2$ for $L_1=\cup_{x\in L} L_{1,x}$ and $L_2=\cup_{x\in L}L_{2,x}$.

    In case (i) of Theorem \ref{thm:lie-triple-stable}, by Lemma~\ref{lem:no-type-a-decompose}, if there is no point in $L$ with type (a) Lie triple system, $L$ is decomposable into $L_1\times L_2$ for $L_1\subset \Q^n_{\eps_1}$ and $L_2\subset\Q^m_{\eps_2}$ both isoparametric. Otherwise, by Theorem \ref{thm:reduction}, for any $x\in L$ with type (a) Lie triple system, there exists a neighborhood $L_x$ of $x$ such that $L_x$ is an open subset of $L_{1,x}\times L_{2,x}$, where $L_{1,x}$ is an isoparametric hypersurface in $\Q^n_\eps\times\R^{m_1}$ with angle function satisfying $-1<C<1$, $L_{2,x}$ is an isoparametric submanifold in $\R^{m_2+m_3}$, and $m_1+m_2+m_3=m$. Theorem \ref{thm:global-dem-SRH} (i) rules out the $\eps_{1}=1$ case by the angle function result. Thus by Theorem~\ref{thm:global-dem-SRH}(ii), $\eps_{1}$ can only be $-1$ and each $L_{1,x}$ can only be a flat-horospherical hypersurface. For any $x,y\in L$ satisfying the assumption in Theorem \ref{thm:reduction}, if $(L_{1,x}\times L_{2,x})\cap(L_{1,y}\times L_{2,y})\ne\varnothing$, $L_{1,x}\cup L_{1,y}$ is a flat-horospherical hypersurface by Lemma \ref{lem:flat-horosphere}, and $L_{2,x}\cup L_{2,y}$ is an isoparametric submanifold by unique continuation of isoparametric submanifolds in Euclidean spaces \cite{terng1987submanifold}*{Theorem 3.4}. By Theorem \ref{thm:constant-angle}, the angle function is constant on $L_{1,x}\cup L_{2,x}$. Note that this implies that $\theta$ is constant on $(L_{1,x}\times L_{2,x})\cap(L_{1,y}\times L_{2,y})$ where $\theta$ is defined in equation \eqref{eq:diagonal-normal}. On the other hand, assume there exists $p\in L$ with type (b) or type (c) Lie triple system, $\theta=0$ or $\frac{\pi}{2}$ by equation \eqref{eq:diagonal-normal}. By connectivity of $L$, type (b) and (c) are not possible on $L$, since there is already a point on $L$ with type (a) and $\theta$ is constant. Therefore $L$ is a subset of $L_1\times L_2$ where $L_1=\cup_{x\in L} L_{1,x}$ is a flat-horospherical hypersurface and $L_2=\cup_{x\in L}L_{2,x}$ is an isoparametric submanifold.

    In case (v) of Theorem \ref{thm:lie-triple-stable}, Theorem \ref{prop:no-higher-codim-diagonal} shows that there is no such $L$.

    Case (vi) is the only case that is not fully classified yet. If $\eps_1=-\eps_2$, Theorem \ref{thm:global-dem-SRH}(i) shows that there is no non-decomposable $L$. The rest is exactly the case in which $\eps_1=\eps_2$ and $\text{codim}L=1$.
\end{proof}

The remaining cases are $\eps_1=\eps_2$ and codim$L=1$, that is, isoparametric hypersurfaces in $\Sph^n\times \Sph^m$ and $\Hyp^n\times \Hyp^m$ are not classified by Theorem \ref{main}.

\begin{appendices}

\section{Algebraic Results}\label{app:algebraic}

The following elementery Lemmas will also be used in the paper.
\begin{lemma}[Theorem 3.7 of \cite{teschl2012}]\label{lem:spectrum-unique}
The functions of $r$ of the form $r^ke^{\mu r}$ are linear independent in $\mathbb C$ for different $(k,\mu)$, where $k$ is any integer and $\mu$ is any complex number.
\end{lemma}

\begin{lemma}\label{lem:S2=a2I-eigenvalues-a--a}
    Let $S$ be a symmetric matrix, $I$ be the identity matrix, and $c_1$ and $c_2$ be constant numbers where $c_1\ne0$. If $F(x)=\det(x(c_1I-c_2S)+x^{-1}(c_1I+c_2S))$ is a constant function on $x$, then $S^2=\frac{c_2^2}{c_1^2}I$ and $S$ only has eigenvalues $\frac{c_2}{c_1}$ and $-\frac{c_2}{c_1}$, and their multiplicities are the same.
\end{lemma}
\begin{proof}
    Let $S'=P^{-1}SP$ be the diagonalized $S$, with the same eigenvalues as $S$ being $\lambda_1,...,\lambda_l$. Then we can rewrite $F(x)$ as $\det(P^{-1}(x(c_1I-c_2S)+x^{-1}(c_1I+c_2S))P)$ which is $\det(x(c_1I-c_2S')+x^{-1}(c_1I+c_2S'))$. Therefore $F(x)=\prod_{k=1}^l(\alpha_k x+\beta_k x^{-1})$, where $\alpha_k=c_1-c_2\lambda_k$ and $\beta_k=c_1+c_2\lambda_k$. Multiplying with $x^l$, for $F(x)$ to be a constant function on $x$, $x^lF(x)=\prod_{k=1}^l(\alpha_k x^2+\beta_k )$ must be a monomial of $x$, which implies $\alpha_k=0$ or $\beta_k=0$ for all $k$. Also, the degree of $x$ in $\prod_{k=1}^l(\alpha_k x^2+\beta_k )$ must be $k$, which implies that the number $\#\{k:\alpha_k=0\}$ equals the number $\#\{k:\beta_k=0\}$. Therefore $\lambda_k=\pm \frac{c_2}{c_1}$, and the two signs have equal multiplicity.
\end{proof}

\section{Statement on the Use of Artificial Intelligence Tools}\label{app:AI}

During the preparation of this manuscript, the authors used ChatGPT (GPT-5.5 Pro model) as an auxiliary tool for exploratory computations and literature navigation.

Specifically, ChatGPT was used to assist with preliminary checks of the Codazzi equations in Lemma \ref{lem:HH-computations-3} and eigenvalue calculations in Lemma \ref{lem:S2=a2I-eigenvalues-a--a}. All AI-generated outputs were independently examined, selected, and verified by the authors, who take full responsibility for the accuracy of these computations.

ChatGPT was also used to assist in recalling and identifying the following mathematical tool from the literature: 
Inversion method for Riccati-type equations used in Section \ref{subsec:Case-i-of-Thm}. 
The authors independently assessed the relevance and applicability of this tool, consulted the cited sources, and verified all mathematical arguments before incorporating them into the manuscript.

\end{appendices}


\vspace{30pt} 
\noindent
Jinxuan Chen  \\
Beijing International Center for Mathematical Research, \\
Peking University, Beijing, China. \\
Email: {\it jxchen@bicmr.pku.edu.cn}
\\[\baselineskip]
Xiaobo Liu \\
School of Mathematical Sciences \& \\
Beijing International Center for Mathematical Research, \\
Peking University, Beijing, China. \\
Email: {\it xbliu@math.pku.edu.cn}
\\[\baselineskip]
Wanxu Yang \\
School of Mathematical Sciences, \\
Peking University, Beijing, China. \\
Email: {\it yangwanxu@stu.pku.edu.cn}

\end{document}